\documentclass[11pt]{article}

\usepackage[dvips]{epsfig}
\usepackage{graphicx}
\usepackage{amssymb,graphics,amsmath,amsthm,amsopn,amstext,amsfonts,color}
\usepackage{algorithm,algorithmic}

\usepackage{subfigure}
\usepackage{multirow}

\usepackage{multirow}

\newcommand{\ubld}{{\mbox{\bf u}}}

\newcommand{\diag}{{\rm diag}}

\newcommand{\gap}{\vspace{0.1in}}

\newcommand{\Hbld}{{\mbox{\bf H}}}

\newcommand{\Gbld}{{\mbox{\bf G}}}

\newcommand{\wt}{\widetilde}

\newcommand{\wh}{\widehat}

\newcommand{\ol}{\overline}

\newcommand{\Ical}{\mathcal I}

\newcommand{\Ecal}{\mathcal E}

\newcommand{\Gcal}{\mathcal{G}}
\newcommand{\Pcal}{\mathcal{P}}

\newcommand{\alphabf}{\boldsymbol\alpha}
\newcommand{\betabf}{\boldsymbol\beta}
\newcommand{\zetabf}{\boldsymbol\zeta}
\newcommand{\thetabf}{\boldsymbol\theta}

\newcommand{\deltabf}{\boldsymbol\delta}

\newcommand{\mycut}[1]{{}}

\newcommand{\argmin}{\operatornamewithlimits{\arg\min}}

\newcommand{\tblue}[1]{\textcolor{black}{#1}}

\newtheorem{lemma}{Lemma}[section] 
\newtheorem{corollary}{Corollary}[section] 
\newtheorem{proposition}{Proposition}[section] 
\newtheorem{remark}{Remark}[section]

\begin{document}

\title{Fully Distributed Optimization based CAV Platooning Control under
Linear Vehicle Dynamics}

\author{
 Jinglai Shen\thanks{Jinglai Shen is with Department of Mathematics and
    Statistics, University of Maryland Baltimore County, MD 21250, USA. Email:
    {\tt\small shenj@umbc.edu}.},  \ \ \ \
 Eswar Kumar H. Kammara\thanks{Eswar Kumar H. K. is with Department of Mathematics and
    Statistics, University of Maryland Baltimore County, MD 21250, USA. Email:
    {\tt\small eswar1@umbc.edu}.}, \ \ \ \
   Lili Du\thanks{Lili Du is with  Department of Civil and Coastal Engineering,
University of Florida, Gainesville, FL 32608, USA. Email: {\tt\small lilidu@ufl.edu}.}
     }
 \date{Originally November 23, 2020}

\maketitle

%
\begin{abstract}
This paper develops distributed optimization based, platoon centered CAV car following  schemes, motivated by the recent interest in CAV platooning technologies.
Various distributed optimization or control schemes have been developed for CAV platooning. However, most existing distributed schemes for platoon centered CAV control require either centralized data processing or centralized computation in at least one step of their schemes, referred to as partially distributed schemes. In this paper, we develop fully distributed optimization based, platoon centered CAV platooning control under the linear vehicle dynamics via the model predictive control approach with a general prediction horizon. These fully distributed schemes do not require centralized data processing or centralized computation through the entire schemes. To develop these schemes, we propose a new formulation of the objective function and a decomposition method that decomposes a densely coupled central objective function into the sum of several locally coupled functions whose coupling satisfies the network topology constraint. We then exploit the formulation of locally coupled optimization and operator splitting methods to develop fully distributed schemes. Control design and stability analysis is carried out to achieve desired traffic transient performance and asymptotic stability. Numerical tests demonstrate the effectiveness of the proposed fully distributed schemes and CAV platooning control.
\end{abstract}

\noindent\textit{Keywords}: Connected and autonomous vehicle, car following control, distributed algorithm, constrained optimization, control and stability

%
\section{Introduction}

The recent advancement of connected and autonomous vehicle (CAV) technologies
provides a unique opportunity to mitigate urban traffic congestion through innovative traffic
flow control and operations. Supported by advanced sensing, vehicle communication, and portable computing technologies, CAVs can sense, share, and process real-time mobility data and conduct cooperative or coordinated driving. This has led to a surging interest in self-driving technologies. Among a number of emerging self-driving technologies, vehicle platooning technology receives substantial attention. Specifically, the vehicle platooning technology  links a group of CAVs through cooperative acceleration or speed control. It allows adjacent group members to travel safely at a higher speed with smaller spacing between them and thus has a great potential to increase lane capacity, improve traffic flow efficiency, and reduce congestion, emission, and fuel consumption \cite{bergenhem2012overview, kavathekar2011vehicle}.

There is  extensive literature on CAV platooning control. The widely studied approaches include adaptive cruise control (ACC) \cite{kesting2008adaptive, li2011model, marsden2001towards, vander2002effects, zhou2017rolling}, cooperative adaptive cruise control (CACC)  \cite{shladover2015cooperative, Shladover2012, VanA2006, zhao2020distributionally}, and platoon centered vehicle platooning control \cite{GongDu_TRB18, GShenDu_TRB16, wang2019real, wang2014rolling}. The first two approaches intend to improve an individual vehicle's safety, mobility, and string stability rather than systematical performance of the entire platoon, even though enhanced system performance is validated by analysis, simulations, or field experiments.
In contrast, the platoon centered approach aims to improve the performance of the entire platoon and seeks a control input that optimizes the platoon's transient traffic dynamics for a smooth traffic flow while achieving stability and other desired long-time dynamical behaviors.

The platoon centered CAV platooning control often gives rise to sophisticated, large-scale optimal control or optimization problems, and requires extensive computation. In order to successfully implement these control schemes, efficient real-time computation is needed \cite{wang2019real}. However, due to high computation load and the absence of roadside computing facilities, centralized computation is either inefficient or infeasible \cite{wang2016cooperative}.
By leveraging portable computing capability of
each vehicle, distributed computing is a favorable option because it has potentials to be more adaptive to different platoon network topologies, be more robust to network malfunctions,
and accommodate for communication delays effectively \cite{mesbahi2010graph, wang2016cooperative}. In spite of these advantages, the development of efficient distributed algorithms to solve platoon centered optimization or optimal control problems in  real-time is nontrivial,
especially under complicated traffic conditions and constraints.
A number of effective distributed control or optimization schemes have been proposed for CAV platooning \cite{wang2019real, wang2016cooperative, zhao2020distributionally, zhou2019distributed}. In particular, the recent paper \cite{GShenDu_TRB16} develops  model predictive control (MPC) based car-following control schemes for CAV platooning  by exploiting transportation, control and optimization methodologies. These control schemes take vehicle constraints, transient dynamics and and asymptotic dynamics of the entire platoon into account, and can be computed in a distributed manner.
The paper \cite{GongDu_TRB18} extends these distributed schemes to a mixed traffic flow including both CAVs and human-driven vehicles.
However, to the best of our knowledge,  the proposed schemes in \cite{GongDu_TRB18, GShenDu_TRB16} as well as many other existing  distributed or decentralized schemes \cite{KNS_SIOPT11} either require all vehicles to exchange information with a central component for centralized data processing or perform centralized computation in at least one step of these schemes. We refer to these schemes as {\em partially distributed} schemes.
In contrast, the distributed schemes developed in this paper do not require centralized data processing or carry out centralized computation through the entire schemes and thus are called {\em fully distributed}. Distinct advantages of fully distributed schemes include but are not limited to: (i) no data synchronization is needed such that no central computing equipment is required; and (ii) each vehicle only interacts with its nearby vehicles through a vehicle communication network. Hence, these schemes impose less restriction on vehicle communication networks and can be easily implemented on a wide range of vehicle networks. They are also suitable for a large CAV platoon in remote areas where communication network is unreliable or roadside equipments are scarce. Further, they are more robust to network malfunction or cyber attacks.

In this paper, we
develop fully distributed optimization based and platoon centered CAV car following control  schemes
over a general vehicle communication network. We propose a general $p$-horizon MPC model subject to the linear vehicle dynamics of the CAVs and various physical or safety constraints. Typically, a fully distributed optimization scheme requires the objective function and constraints of the underlying optimization problem to be decoupled \cite{HuXiaoLiu_CDC18}. However, the constrained optimization problem arising from the proposed MPC model does not satisfy this requirement since its objective function is densely coupled and its constraints are locally coupled; see Remark~\ref{remark:local_coupled_constraint} for more details. Therefore, this paper develops new techniques to overcome this difficulty.

The main contributions of this paper are summarized as follows:
\begin{itemize}
  \item [(1)] We propose a new form of the objective function in the MPC model with new sets of weight matrices. This new formulation facilitates the development of fully distributed schemes and closed loop stability analysis whereas it can achieve desired traffic transient performance of the whole platoon. Based on the new formulation, a decomposition method is developed for the strongly convex quadratic objective function. This method decomposes the central objective function
      into the sum of locally coupled (strongly) convex quadratic functions, where local coupling satisfies the network topology constraint under a mild assumption on network topology. Along with locally coupled constraints in the MPC model, the underlying optimization model is formulated as a locally coupled convex quadratically constrained quadratic program (QCQP).

  \item [(2)] Fully distributed schemes are developed for solving the above-mentioned convex QCQP arising from the MPC model using the techniques of locally coupled optimization and operator splitting methods. Specifically, by introducing copies of local coupling variables of each vehicle, an augmented optimization model is formulated with an additional consensus constraint. A generalized Douglas-Rachford splitting method based distributed scheme is developed, where only local information exchange is needed, leading to a fully distributed scheme. Other operator splitting method based distributed scheme are also discussed.

  \item [(3)] The new formulation of the weight matrices and objective function leads to different closed loop dynamics in comparison with that in \cite{GShenDu_TRB16}. Besides, since a general $p$-horizon MPC is considered, it calls for new stability analysis of the closed loop dynamics. We perform detailed stability analysis and choose suitable weight matrices for desired traffic transient performance for a general horizon length $p$. In particular, we prove that up to a horizon of $p = 3$, the closed loop dynamic matrix is Schur stable. Extensive numerical tests are carried out to test the proposed distributed schemes under different MPC horizon $p$'s and to evaluate the closed loop stability and performance.
\end{itemize}

The rest of the paper is organized as follows. Section~\ref{sect:dynamics_constraint_topology} introduces the linear vehicle dynamics, state and control constraints, and vehicle communication networks. The model predictive control model with a general prediction horizon $p$ is proposed and formulated as a constrained optimization problem in Section~\ref{sect:MPC_formulation}; fundamental properties of this optimization problem are established. Section~\ref{sect:dist_scheme} develops fully distributed schemes by exploiting a decomposition method for the central quadratic objective function, locally coupled optimization, and operator splitting methods. Control design and stability analysis for the closed loop dynamics is presented in Section~\ref{sect:control_analysis} with numerical results given in Section~\ref{sect:numerical_results}. Finally, conclusions are given in Section~\ref{sect:conclusion}.

%

\section{Vehicle Dynamics, Constraints, and Communication Networks} \label{sect:dynamics_constraint_topology}

We consider a platoon consisting of vehicles on a roadway, where the (uncontrolled) leading vehicle is labeled by the index 0 and its $n$  following CAVs are labeled by the indices $i=1, \ldots, n$, respectively. Let $x_i, v_i$ denote the longitudinal position and speed of the $i$th vehicle, respectively. Let $\tau>0$ be the sampling time, and each time interval is given by $[k\tau, (k+1)\tau)$ for $k \in \mathbb Z_+:=\{0, 1, 2, \ldots\}$.  We
consider the following  kinematic model for linear vehicle dynamics that is
widely adopted in system-level studies with the acceleration $u_i(k)$ as the control input:
\begin{eqnarray} \label{eqn:model_longit_double_integrator}
  x_i(k+1) \, = \, x_i(k) + \tau v_i(k) + \frac{\tau^2}{2} u_i(k), \qquad
  v_i(k+1)  \, = \, v_i(k) + \tau u_i(k).
\end{eqnarray}

\noindent {\bf State and control constraints.} \ Each vehicle in a platoon is subject to several important state and control constraints. For each $i=1, \ldots, n$,
\begin{itemize}
  \item [(i)] \underline{Control constraints}: 
$ a_{\min} \le u_i \le a_{\max}$, where
$a_{\min}<0$ and $a_{\max}>0$ are pre-specified  acceleration and deceleration bounds for a vehicle.

 \item [(ii)] \underline{Speed constraints}: $v_{\min} \le v_i \le v_{\max}$, where $0\le v_{\min}<v_{\max}$ are pre-specified bounds on longitudinal speed for a vehicle;

 \item [(iii)] \underline{Safety distance constraints}: these constraints guarantee sufficient spacing between neighboring vehicles to  avoid collision even if the leading vehicle comes to a sudden stop. This gives rise to the safety distance constraint of the following form:
 \begin{equation}
     x_{i-1} - x_{i}  \, \ge \, L + r \cdot v_i - \frac{(v_i-v_{\min})^2}{2a_{\min}}, \label{eqn:safety_constraint}
\end{equation}
where $L>0$ is a constant depending on vehicle length, and $r$ is the reaction time.
\end{itemize}
 Note that constraints (i) and (ii) are decoupled across vehicles, whereas the safety distance constraint (iii) is state-control coupled since such a constraint involves  control inputs of two vehicles. This yields challenges to distribution computation.
Further, we consider the identical bounds of the acceleration and deceleration and ignore their variations due to road surface condition changes in this paper, although the proposed approach can handle a general case with different  acceleration or deceleration bounds.

\gap

\noindent {\bf Communication network topology.}
We consider a general communication network whose topology is modeled by a  graph $\Gcal(\mathcal V, \mathcal E)$,
 where $\mathcal V=\{1, 2, \ldots, n\}$ is the set of nodes where the $i$th node corresponds to the $i$th CAV, and $\mathcal E$ is the set of edges connecting two nodes in $\mathcal V$.
 Let $\mathcal N_i$ denote the set of neighbors of node $i$, i.e., $\mathcal N_i =\{ j \, | \, (i, j) \in \mathcal E\}$.
The following assumption on the communication network topology is made throughout the paper:
\begin{itemize}
  \item [$\bf A.1$] The graph $\mathcal G(\mathcal V, \mathcal E)$ is undirected and connected. Further, two neighboring vehicles form a  bidirectional edge of the graph, i.e., $(1, 2), (2, 3), \ldots, (n-1, n) \in \mathcal E$.
\end{itemize}
Since the graph is undirected,  for any $i, j\in \mathcal V$ with $i \ne j$, $(i, j) \in \mathcal E$ means that there exists an edge between node $i$ and node $j$. In other words, vehicle $i$ can receive information from  vehicle $j$ and send information to vehicle $j$, and so does vehicle $j$. We also assume that the first vehicle can detect and receive  $x_0$, $v_0$ and $u_0$ from the leading vehicle.
%
%
The setting given by $\bf A.1$ includes many widely used communication networks of CAV platoons, e.g., immediate-preceding,  multiple-preceding, and  preceding-and-following networks \cite{ZhengLiBorrelliH_TCS16}.
%

%
\section{Model Predictive Control for CAV Platooning Control} \label{sect:MPC_formulation}

We exploit the model predictive control (MPC) approach for car following control of a platoon of CAVs.
Let $\Delta$ be the desired (constant) spacing between two adjacent vehicles, and $(x_0, v_0, u_0)$ be the position, speed, and control input of the leading vehicle, respectively. Define the vectors:
(i) $z(k):=\big( x_0-x_1-\Delta, \ldots, x_{n-1}-x_n-\Delta\big)(k) \in \mathbb R^n$, representing the relative spacing error; (ii) $z'(k):=\big( v_0-v_1, \ldots, v_{n-1}-v_n\big)(k) \in \mathbb R^n$, representing the relative speed between adjacent vehicles; and (iii) $u(k):=\big(u_1, \ldots, u_n\big)(k)\in \mathbb R^n$, representing the control input.
Further, let $w_i(k):= u_{i-1}(k) - u_i(k)$ for each $i=1, \ldots, n$, and $w(k):= \big( w_1, \ldots, w_n \big)(k) \in \mathbb R^n$, which stands for the difference of control input between adjacent vehicles. Hence, for any $k \in \mathbb Z_+$, $u(k) = - S_n w(k) + u_0(k) \cdot \mathbf 1$, where $\mathbf 1 :=(1, \ldots, 1)^T$ is the vector of ones, and
%
\begin{equation} \label{eqn:matrix_S_n}
    S_n \, := \, \begin{bmatrix}
1&0&0&  \hdots &0\\
1 &1& 0&  \hdots &0\\
\vdots &\vdots & \ddots  & \ddots & \vdots\\
1 & 1 &\hdots& 1  & 0\\
1 & 1 &  \hdots & 1 &1 \\
\end{bmatrix} \in \mathbb R^{n\times n}, \qquad  S^{-1}_n \, = \, \begin{bmatrix} 1& & & & \\
 -1 & 1& & &   \\
 &  \ddots & \ddots &  &  \\
   & & -1 & 1 &  \\
   &  &  & -1 & 1\\
\end{bmatrix} \in \mathbb R^{n\times n}.
\end{equation}

Given a prediction horizon $p\in \mathbb N$, the $p$-horizon MPC control is determined by solving the following constrained optimization problem at each $k\in \mathbb Z_+$, involving all vehicles' control inputs for given feasible state $(x_i(k), v_i(k))^n_{i=1}$ and $(x_0(k), v_0(k), u_0(k))$ at time $k$ subject to the vehicle dynamics model (\ref{eqn:model_longit_double_integrator}):
%
\begin{align}
 & \ \mbox{minimize}  \ \ J(u(k), \ldots, u(k+p-1) )  :=  \label{eq:MPC} \\
 & \ \frac{1}{2} \sum^p_{s=1} \Big(  \underbrace{\tau^2 u^T(k+s-1) S^{-T}_n Q_{w, s} S^{-1}_n  u(k+s-1)}_{\mbox{ride comfort}} +   \underbrace{z^T(k+s) Q_{z, s} z(k+s) + (z'(k+s))^T Q_{z', s} z'(k+s)}_{\mbox{traffic stability and smoothness}} \Big)   \notag
\end{align}
subject to: for each $i=1, \ldots, n$ and each $s=1, \ldots, p$,
%
%
%
\begin{eqnarray}
  a_{\min}  & \le  u_i(k+s-1) \ \le  a_{\max}, & \qquad v_{\min}  \, \le \, v_i(k+s) \ \le \ v_{\max}, \label{eqn:MPC:u_constriant} \\
  %
  %
 &   x_{i-1}(k+s)-x_i(k+s)  & \ge   L + r \cdot v_i(k+s)  - \frac{(v_i(k+s)-v_{\min})^2}{2a_{\min}},  \label{eqn:MPC:safety_constraint}
\end{eqnarray}
where $Q_{z, s}$, $Q_{z', s}$ and $Q_{w, s}$ are $n\times n$ symmetric positive semidefinite weight  matrices to be discussed soon. Note that when $p>1$, $(x_0(k+s+1), v_0(k+s+1), u_0(k+s))$  are unknown at time $k$ for $s=1, \ldots, p-1$. In this case, we  assume that $u_0(k+s)=u_0(k)$ for all $s=1, \ldots, p-1$ and use these $u_0(k+s)$'s and the vehicle dynamics model (\ref{eqn:model_longit_double_integrator}) to predict $(x_0(k+s+1), v_0(k+s+1))$ for $s=1, \ldots, p-1$.  Here we assume that $u_0(k)$ represents the actual acceleration of the leading vehicle at time $k$.

\begin{remark} \rm \label{remark:objective_func}
The three terms in the objective function $J$ intend to minimize traffic flow oscillations via mild control:
the first term  penalizes the magnitude of control, whereas the second and last terms penalize the  variations of the relative spacing and relative speed, respectively. The presence of the matrix $S_n$ in the first term is due to the coupled vehicle dynamics through the CAV platoon. To illustrate this, let $\wt w(k+s-1):=w(k+s-1) - u_0(k) \cdot \mathbf e_1$ for $s=1, \ldots, p$. Thus  $\wt w(k+s-1) = - S^{-1}_n u(k+s-1)$ for each $s=1, \ldots, p$.  Therefore, the first term in $J$ satisfies $ \tau^2 u^T(k+s-1)  S^{-T}_n Q_{w, s} S^{-1}_n u(k+s-1) =  \tau^2 \wt w^T(k+s-1) \, Q_{w, s} \, \wt w(k+s-1)$ for each $s$.
\end{remark}

The weight matrices $Q_{z, s}$, $Q_{z', s}$, and $Q_{w, s}$, $s=1, \ldots, p$ determine transient and asymptotic dynamics, and they depend on vehicle network topologies and can be chosen by stability analysis and transient dynamics criteria of the closed loop system. To develop fully distributed schemes for a broad class of vehicle network topologies and to facilitate control design and analysis, we make the following blanket assumption on $Q_{z, s}$, $Q_{z', s}$, and $Q_{w, s}$ throughout the rest of the paper:
\begin{itemize}
  \item [$\bf A.2$] For each $s=1, \ldots, p$, $Q_{z, s}$ and $Q_{z', s}$ are diagonal and positive semidefinite (PSD), and $Q_{w, s}$ is diagonal and positive definite (PD).
\end{itemize}

The reasons for considering this class of diagonal positive semidefinite or positive definite weight matrices are three folds: (i) Diagonal matrices have a simpler interpretation in transportation engineering so that the selection of such matrices is easier to practitioners. For instance, the diagonal $Q_{z, s}$ and $Q_{z', s}$ mean that one imposes penalties on each element of $z(k+s)$ and $z'(k+s)$ without considering their coupling. Further, by suitably choosing the weight matrices $Q_{w, s}$, it can be shown that the ride comfort term in equation (\ref{eq:MPC}), %
which corresponds to acceleration of CAVs,
is similar to imposing direct penalties on $u_i$'s, which simplifies control design. (ii) This class of weight matrices facilitates the development of fully distributed schemes for general vehicle network topologies. (iii) Closed-loop stability and performance analysis is relatively simpler (although still nontrivial) when using this class of weight matrices.
The detailed discussions of choosing diagonal, positive semidefinite or positive definite weight matrices for satisfactory closed loop dynamics will be given in Section~\ref{sect:control_analysis}.

The sequential feasibility  has been established in \cite{GongDu_TRB18, GShenDu_TRB16}  for the MPC model (\ref{eq:MPC}) when $r \ge \tau$.
Define
$
 \Pcal((x_i, v_i)^n_{i=0}, u_0) \, : = \,  \{ u \in \mathbb R^n \, | \, a_{\min} \le u_i \le a_{\max}, \, v_{\min} \le v_i + \tau u_i \le v_{\max},  \, h_{i}(u) \le 0, \ \forall \, i=1, \ldots, n\},
$
where $h_{i}(u)  := L + r ( v_i + \tau u_i) - \frac{(v_i + \tau u_i - v_{\min})^2}{2 a_{\min}} +
 (x_i - x_{i-1}) + \tau ( v_i - v_{i-1}) + \frac{\tau^2}{2}[ u_{i} - u_{i-1}]$  for each $i=1, \ldots, n$.
Specifically, the sequential feasibility implies that for any feasible $x_i(k), v_i(k), u_0(k)$ at time $k$, the constraint set $\Pcal((x_i(k), v_i(k))^n_{i=0}, u_0(k))$ is non-empty such that the MPC model (\ref{eq:MPC}) has a solution $u_*(k)$ such that  the constraint set $\Pcal((x_i(k+1), v_i(k+1))^n_{i=0}, u_0(k+1))$ is non-empty.
Using this result, we show below that under a mild assumption, the constraint sets of the MPC model have nonempty interior for any MPC horizon $p \in \mathbb N$. This result is important to the development of distributed algorithms.

\begin{corollary} \label{coro:nonempty_interior_general_p}
Consider the linear vehicle dynamics (\ref{eqn:model_longit_double_integrator}) and assume $r \ge \tau$.  Suppose the leading vehicle is such that $(v_0(k), u_0(k))$ is feasible and $v_0(k)>v_{\min}$ for all $k\in \mathbb Z_+$. Then the constraint set of the $p$-horizon MPC model (\ref{eq:MPC}) has nonempty interior at each $k$.
\end{corollary}

\begin{proof}
   Fix an arbitrary $k \in \mathbb Z_+$. Since $v_0(k)>v_{\min}$, it follows from
   \cite[Proposition 3.1]{GShenDu_TRB16} that there exists a vector denoted by $\wh u(k)$ in the interior of the set
   $\Pcal((x_i(k), v_i(k))^n_{i=0}, u_0(k))$. Let $x_i(k+1)$ and $v_i(k+1)$ be generated by $\wh u(k)$ (and $(x_i(k), v_i(k))^n_{i=0}, u_0(k)$). Since $v_0(k+1)>v_{\min}$, we deduce via \cite[Proposition 3.1]{GShenDu_TRB16} again that there exists a vector denoted by $\wh u(k+1)$ in the interior of the constraint set $\Pcal((x_i(k+1), v_i(k+1))^n_{i=0}, u_0(k+1))$. Continuing this process in $p$-steps, we derive the existence of an interior point in the constraint set of the $p$-horizon MPC model (\ref{eq:MPC}).
\end{proof}

%
\subsection{Constrained MPC Optimization Model} \label{subsect:optimization_linear_dynamics}

%
%
Consider the constrained MPC optimization model (\ref{eq:MPC}) at an arbitrary but fixed time $k \in \mathbb Z_+$ subject to the linear vehicle dynamics (\ref{eqn:model_longit_double_integrator}).
In view of the following results: for each $s=1, \ldots, p$,
\begin{align*}
  v_i(k+s) & = v_i(k) + \tau \sum^{s-1}_{j=0} u_i(k+j), \quad z'(k+s) = z'(k) + \tau \sum^{s-1}_{j=0} w(k+j), \\
  z(k+s) & = z(k) + s \tau z'(k) + \tau^2 \sum^{s-1}_{j=0} \frac{2(s-j)-1}{2} w(k+j), \quad w(k+s) = S^{-1}_n \big[ - u(k+s) + u_0(k) \cdot \mathbf 1 \big],
\end{align*}
we formulate (\ref{eq:MPC}) as the constrained convex minimization problem (where we omit $k$ since it is fixed):
%
\begin{equation}
\begin{array}{ll}
\mbox{minimize} & J(\ubld):= \frac{1}{2} \ubld^T W \ubld + c^T \ubld + \gamma,  \label{eq:Multi_user} \\  
%
\mbox{subject to} &  \ubld_i \in \mathcal X_i, \quad  (H_i( \ubld ))_s \le 0, \quad \forall \, i=1,\ldots, n, \quad \forall \, s=1, \ldots, p,
\end{array}
\end{equation}
%
where 
$\ubld:=(\ubld_1, \ldots, \ubld_n)\in \mathbb R^{np}$ with $\ubld_i:=( u_i(k), \ldots, u_i(k+p-1))\in \mathbb R^{p}$, $W$ is a PD matrix to be shown in Lemma~\ref{lem:W_PD} below, $c \in \mathbb R^{n p}$, $\gamma\in \mathbb R$,  each $\mathcal X_i :=\{ z \in \mathbb R^p \, | \, a_{\min} \cdot \mathbf 1 \le z \le a_{\max} \cdot \mathbf 1, \ (v_{\min}-v_i(k)) \cdot \mathbf 1  \le \tau S_p z \le (v_{\max}-v_i(k)) \cdot \mathbf 1 \}$ is a polyhedral set, and each $(H_i( \cdot ))_s$ is a convex quadratic function characterizing the safety distance given by (\ref{eqn:H_function_safety_distance}). Here $S_p$ is the $p\times p$ matrix of the form given by (\ref{eqn:matrix_S_n}). Further, $\mathbf u_0 := u_0(k) \cdot \mathbf 1_p \in \mathbb R^p$ for the given $u_0(k)$. An important property of the matrix $W$ in (\ref{eq:Multi_user}) is given below.

\begin{lemma} \label{lem:W_PD}
 Suppose that $Q_{z, s}$ and $Q_{z', s}$ are PSD and $Q_{w, s}$ are PD for all $s=1, \ldots, p$ (but not necessarily diagonal). Then the matrix $W$ in (\ref{eq:Multi_user}) is PD.
\end{lemma}

\begin{proof}
  Let $\ubld$ be an arbitrary nonzero vector in $\mathbb R^{np}$. Since $J(\cdot)$ is quadratic, we have $\frac{1}{2} \ubld^T W \ubld = \lim_{\lambda \rightarrow \infty} \frac{ J(\lambda \ubld) }{\lambda^2}$. In view of the equivalent formulation of $J(\cdot)$ given by (\ref{eq:MPC}), we deduce that for any $\lambda>0$, $J(\lambda \ubld) = J(  \lambda u(k), \ldots, \lambda u(k+p-1) ) \ge  \frac{\lambda^2}{2} \sum^p_{s=1}  \tau^2  u^T(k+s-1) S^{-T}_n Q_{w, s} S^{-1}_n u(k+s-1) >0$, where the first inequality follows from the fact that $Q_{z, s}$ and $Q_{z', s}$ are PSD, and the second inequality holds because $Q_{w, s}$, and thus $S^{-T}_n Q_{w, s} S^{-1}_n$,  are PD. Therefore,  $\frac{J(\lambda \ubld)}{\lambda^2} \ge \frac{1}{2} \sum^p_{s=1}  \tau^2  u^T(k+s-1) S^{-1}_n Q_{w, s} S^{-1}_n u(k+s-1) >0$, leading to $\frac{1}{2} \ubld^T W \ubld \ge \frac{1}{2} \sum^p_{s=1}  \tau^2  u^T(k+s-1) S^{-1}_n Q_{w, s} S^{-1}_n u(k+s-1) >0$. Hence, $W$ is PD.
\end{proof}

To establish the closed form expressions of the matrix $W$ and the vector $c$ in (\ref{eq:Multi_user}), we define the following matrices for any $i, j \in \{1, \ldots, p\}$:
%
\[
 V_{i, j}  \, := \, S^{-T}_n \left[\sum^p_{s=\max(i, j)} \Big( \frac{\tau^4}{4}[2(s-i)+1]\cdot [2 (s-j)+1] Q_{z, s} + \tau^2 Q_{z', s} \Big) \right] S^{-1}_n \in \mathbb R^{n\times n}.  
\]
Clearly, $V_{i, j}=V_{j, i}$ for any $i, j$.
Moreover, let $\wt Q_{w, s} := S^{-T}_n Q_{w, s} S^{-1}_n $ for $s=1, \ldots, p$.
%
%
Hence, the symmetric matrix $W$  is given by $W = E^T V E$, where
\begin{equation} \label{eqn:V_matrix}
   V \, = \,  \begin{bmatrix} V_{1,1} +\tau^2 \wt Q_{w, 1}  & V_{1, 2} & V_{1,3} & \cdots & \cdots & V_{1, p} \\ V_{2,1} & V_{2,2} +\tau^2 \wt Q_{w, 2}  & V_{2, 3} & \cdots & \cdots & V_{2, p} \\ \cdots &  & \cdots &  & \cdots & \\ \cdots  &  & \cdots &  & \cdots & \\ V_{p,1} & V_{p, 2} & V_{p, 3} & \cdots & \cdots & V_{p, p} +\tau^2 \wt Q_{w, p}  \end{bmatrix} \in \mathbb R^{np \times np},
\end{equation}
and
 $E \in \mathbb R^{np \times np}$ is the permutation matrix satisfying
\[
    \begin{bmatrix} u(k) \\ u(k+1) \\ \vdots \\ u(k+p-1) \end{bmatrix} \, = \, E \begin{bmatrix} \mathbf u_1 \\ \mathbf u_2 \\ \vdots \\ \mathbf u_n \end{bmatrix}.
\]
Specifically, the $(i, j)$-entry of the matrix $E$ is given by
\begin{equation} \label{eqn:E_matrix}
  E_{i, j} = \left\{ \begin{array}{ll} 1 &  \mbox{if} \ \  i=n\cdot k+s, \ j=p\cdot (s-1)+k+1, \quad \mbox{for} \ \ k=0, \ldots, p-1, \ s=1,\ldots, n; \\ 0, & \mbox{otherwise}. \end{array} \right.
\end{equation}
In particular, when $p=1$, $E=I_n$.

For  a fixed $k \in \mathbb Z_+$, we also define for each $s=1, \ldots, p$,
\begin{align*}
 d_s(k) \, := \, z(k) + s \tau z'(k) +  \tau^2 \sum^{s-1}_{j=0} \frac{2(s-j)-1}{2} S^{-1}_n  \cdot \mathbf 1 \cdot u_0(k), \quad
 f_s(k) & \, := \, z'(k) + \tau  \sum^{s-1}_{j=0} S^{-1}_n  \cdot \mathbf 1 \cdot u_0(k).
\end{align*}
In light of $S^{-1}_n$ given by (\ref{eqn:matrix_S_n}), we have $S^{-1}_n  \cdot \mathbf 1 = \mathbf e_1$. Therefore, we obtain
\begin{equation} \label{eqn:d_s_f_s}
 d_s(k)= z(k) + s \tau z'(k) +  \frac{\tau^2}{2} s^2 \mathbf e_1 u_0(k), \qquad f_s(k) = z'(k) + \tau s \mathbf e_1 u_0(k).
\end{equation}
In view of
\[
 z(k+s) = d_s(k) -\tau^2 \sum^{s-1}_{j=0} \frac{2(s-j)-1}{2} S^{-1}_n u(k+j), \qquad z'(k+s) = f_s(k) - \tau\sum^{s-1}_{j=0} S^{-1}_n u(k+j),
\]
 the linear terms in the objective function $J$ are given by
\begin{equation} \label{eqn:lin_term_in_J}
   -\sum^p_{i=1} \left( \, \sum^p_{s=i} \Big[ \frac{\tau^2}{2} [2(s-i)+1] d^T_{s}(k) Q_{z, s}  + \tau  f^T_s(k) Q_{z', s} \Big] \, \right)  S^{-1}_n \cdot u(k+i-1).
\end{equation}

Using the permutation matrix $E$ given in (\ref{eqn:E_matrix}), we can write the above linear terms as
$
    c^T \mathbf u \, = \, \sum^n_{i=1} c^T_{\Ical_i} \mathbf u_i,
$
where $c_{\Ical_i}$ is the subvector of $c$ corresponding to $\mathbf u_i$. Since $Q_{z, s}$ and $Q_{z', s}$ are diagonal, it is easy to obtain the following lemma via $d_s(k), f_s(k)$ in (\ref{eqn:d_s_f_s}) and the structure of $S^{-1}_n$ given by (\ref{eqn:matrix_S_n}):

\begin{lemma} \label{lem:local_couple_lin_term}
  Consider the vector $c=(c_{\Ical_1}, \ldots, c_{\Ical_n})$ given above. Then for each $i=1, \ldots, n$, the subvector $c_{\Ical_i}$ depends only on $z_i(k), z'_i(k), z_{i+1}(k), z'_{i+1}(k)$'s for $i=1, \ldots, n-1$, and $c_{\Ical_n}$ depends only on $z_n(k), z'_n(k)$.
  Further, only $c_{\Ical_1}$ depends on $u_0(k)$.
\end{lemma}

The above lemma shows that each $c_{\Ical_i}$ only depends on the information of the adjacent vehicles of vehicle $i$, and thus can be easily established from any vehicle network. This property is important for developing fully distributed schemes to be shown in Section~\ref{subsect:ist_scheme_linear}.

To find the vector form of the  safety constraint, we note that for $s=1, \ldots, p$,
\[
  x_i(k+s) \, = \, x_i(k) + s \tau v_i(k) + \tau^2 \sum^{s-1}_{j=0} \frac{2(s-j)-1}{2} u_i(k+j), \quad  v_i(k+s) = v_i(k) + \tau \sum^{s-1}_{j=0} u_i(k+j).
\]
%
%
The  safety distance constraint for the $i$-th vehicle at time $k$ is given by: for $s=1, \ldots, p$,
\begin{align}
   0 & \ge - \Big [ \, x_{i-1}(k)   + s \tau v_{i-1}(k) - (x_i(k) + s \tau v_i(k)) \, \Big ]  - \tau^2 \sum^{s-1}_{j=0} \frac{2(s-j)-1}{2} [u_{i-1}(k+j) - u_i(k+j)]  \notag \\
  & + \, L + r v_i(k) + r \tau  \sum^{s-1}_{j=0} u_i(k+j) \label{eqn:H_function_safety_distance} \\
   & - \frac{1}{2 a_{\min}} \Big[ \tau^2\big ( \sum^{s-1}_{j=0} u_i(k+j) \big)^2 + 2 \tau (v_i(k) - v_{\min}) \sum^{s-1}_{j=0} u_i(k+j) + (v_i(k) - v_{\min})^2 \Big] := (H_i(\mathbf u_{i-1},  \notag \mathbf u_i))_s,
\end{align}
where  $(H_i(\cdot, \cdot))_s$ is a convex quadratic function for each $s=1, \ldots, p$.
Hence, the set $\mathcal Z_i:=\{ \mathbf u \in \mathbb R^{np} \, | \, (H_i( \mathbf u_{i-1}, \mathbf u_i))_s \le 0, \, \forall \, i=1, \ldots, p \}$ is  closed and convex. The problem (\ref{eq:Multi_user}) becomes $\min_{\mathbf u} J(\mathbf u)$ subject to $\mathbf u_i \in \mathcal X_i$ and $\mathbf u \in \mathcal Z_i$ for all $i=1, \ldots, n$, which is a convex quadratically constrained quadratic program (QCQP) and can be solved via a second-order cone program or a semi-definite program  in the centralized manner.

\begin{remark}  \label{remark:local_coupled_constraint} \rm
 The above results show that each $\mathcal X_i$'s are decoupled from the other vehicles, whereas
 the constraint function $H_i$ for  vehicle $i$ is locally coupled with its neighboring vehicles. Specifically, $H_i$  depends not only on $\mathbf u_i$ but also on $\mathbf u_{i-1}$ of vehicle $i-1$, which can exchange information with vehicle $i$. We will explore this local coupling property to develop fully distributed schemes
 for solving (\ref{eq:Multi_user}).
\end{remark}

%
\section{Fully Distributed Algorithms for Constrained Optimization in MPC} \label{sect:dist_scheme}

We develop  fully distributed algorithms for solving the underlying optimization problem given by (\ref{eq:Multi_user}) at each time $k$ using the techniques of locally coupled convex optimization and operator splitting methods.
%

%
\subsection{Brief Overview of Locally Coupled Convex Optimization} \label{subsect:locally_coupled_opt_formulation}

One of major techniques for developing fully distributed schemes for the underlying optimization problem given by (\ref{eq:Multi_user}) is to formulate it as a locally coupled convex  optimization problem \cite{HuXiaoLiu_CDC18}.
To be self-contained, we briefly describe its formulation.

Consider a multi-agent network of $n$ agents whose communication is characterized by a connected and undirect graph $\mathcal G(\mathcal V, \mathcal E)$, where $\mathcal V=\{1, \ldots, n \}$ is the set of agents, and $\mathcal E$ denotes the set of edges. For $i \in \mathcal V$, let $\mathcal N_i$ be the set of neighbors of agent $i$, i.e., $\mathcal N_i =\{ j \, | \, (i, j) \in \mathcal E\}$. Let $\{ \Ical_1, \ldots, \Ical_n\}$ be a disjoint union of the index set $\{1, \ldots, N \}$. Hence, for any $x\in \mathbb R^N$, $(x_{\Ical_i} )^n_{i=1}$ forms a partition of $x$. We call $x_{\Ical_i}$ a local variable of each agent $i$.
For each $i$, define $\wh x_i :=\big(x_{\Ical_i}, (x_{\Ical_j})_{j \in \mathcal N_i} \big) \in \mathbb R^{n_i}$. Hence, for each $i$, $\wh x_i$ contains the local variable $x_{\Ical_i}$ and the variables from its neighboring agents (or locally coupled variables). Consider the convex optimization problem
\[
 (P): \qquad \min_{x \in \mathbb R^N} J(x), \qquad \quad \mbox{ where } \quad J(x) \, := \, \sum^n_{i=1} J_i(\wh x_i),
\]
where $J_i:\mathbb R^{n_i} \rightarrow \mathbb R \cup \{+\infty\}$ is an extended-valued, proper, and lower semicontinuous convex function for each $i$. Clearly, each $J_i$ is locally coupled such that the problem $(P)$ bears the name of ``locally coupled convex optimization''.
Although the problem $(P)$ is seemingly unconstrained, it does include constrained convex optimization since $J_i$ may contain the indicator function of a closed convex set. To impose the locally coupled convex constraint explicitly, the problem $(P)$ can be equivalently written as:
\begin{equation} \label{eqn:local_coupled_opt_formulation}
  (P'): \qquad \min_{x \in \mathbb R^N} \sum^n_{i=1} \wh J_i(\wh x_i), \quad \mbox{subject to } \quad \wh x_i \in \mathcal C_i, \quad \forall \ i=1, \ldots, n,
\end{equation}
where for each $i$, $\wh J_i:\mathbb R^{n_i} \rightarrow \mathbb R$ is a real-valued convex function, and $\mathcal C_i \subseteq \mathbb R^{n_i}$ is a closed convex set.

By introducing copies of the locally coupled variables for each agent and imposing certain consensus constraints on these copies, the paper \cite{HuXiaoLiu_CDC18} formulates the problem $(P')$ (or equivalently $(P)$) as a separable consensus convex optimization problems. Under suitable assumptions, Douglas-Rachford and other operator splitting based distributed schemes are developed; details can be found in \cite{HuXiaoLiu_CDC18}.

%
\subsection{Decomposition of a Strongly Convex Quadratic Objective Function}

The framework of locally coupled optimization problems requires that both an objective function and constraints are expressed in a locally coupled manner. Especially, the central objective function in (\ref{eqn:local_coupled_opt_formulation}) is expected to be written as the sum of a few locally coupled functions preserving certain desired properties, e.g., the (strong) convexity if the central objective function is so, where local coupling satisfies network topology constraints.
While the constraints of the problem (\ref{eq:Multi_user}) have been shown to be locally coupled (cf. Remark~\ref{remark:local_coupled_constraint}), the central strongly convex quadratic objective function,
particularly its quadratic term $\frac{1}{2} \mathbf u^T W \mathbf u$, is highly coupled
 and thus need to be decomposed into the sum of locally coupled (strongly) convex quadratic functions, where the local coupling should satisfy the network topology constraint. In this subsection, we address this decomposition problem under a mild assumption on network topology. Specifically, our results yield a decomposition into convex and strongly convex functions.

We start from a slightly general setting. Let $\boldsymbol\lambda := (\lambda_1, \ldots, \lambda_n) \in \mathbb R^n$ and $\Lambda= \diag( \boldsymbol\lambda)=\mbox{diag}(\lambda_1, \ldots, \lambda_n)$ be a diagonal matrix,  i.e., $\boldsymbol\lambda$ is the vector representation of the diagonal entries of $\Lambda$. Therefore,
the following matrix is tridiagonal:
\begin{equation} \label{eqn:triadigonal_mat_product}
  S^{-T}_n \Lambda S^{-1}_n = \begin{bmatrix} \lambda_1 + \lambda_2 & -\lambda_2 &  & &  \\
    -\lambda_2 & \lambda_2+\lambda_3 &  -\lambda_3 & &  \\
      &   \ddots & \ddots & \ddots &   \\
       & & -\lambda_{n-1} & \lambda_{n-1}+\lambda_n & -\lambda_n  \\
       & & & - \lambda_n & \lambda_n
    \end{bmatrix}.
\end{equation}

Consider a general $p \in \mathbb N$. Let $\Theta$ be a symmetric block diagonal matrix given by
\[
   \Theta =  \begin{bmatrix} \Theta_{1,1}  & \Theta_{1, 2} &  \cdots & \cdots & \Theta_{1, p} \\ \Theta_{2,1} & \Theta_{2,2}  &  \cdots & \cdots & \Theta_{2, p} \\ \cdots &  & \cdots &  & \cdots & \\ \cdots  &  & \cdots &  & \cdots & \\ \Theta_{p,1} & \Theta_{p, 2}  & \cdots & \cdots & \Theta_{p, p}  \end{bmatrix} \in \mathbb R^{np \times np},
\]
where $\Theta_{i, j} = \diag(\boldsymbol\theta_{i,j}) \in \mathbb R^{n\times n}$ is diagonal for some $\boldsymbol\theta_{i,j} \in \mathbb R^n$, and $\thetabf_{i, j} = \thetabf_{j, i}$ for all $i, j=1, \ldots, p$.
Let $(\thetabf_{i, j})_k$ denote the $k$th entry of the vector $\thetabf_{i, j}$.
For each $i=1, \ldots, n$, define the matrix
\begin{equation} \label{eqn:U_i_matrix}
   U_i \, := \, \begin{bmatrix}  (\thetabf_{1,1})_{i}  & (\thetabf_{1,2})_{i}  & \cdots & (\thetabf_{1,p})_{i}  \\
   (\thetabf_{2,1})_{i} & (\thetabf_{2,2})_{i} & \cdots & (\thetabf_{2,p})_{i} \\
   \cdots  &   \cdots   & \cdots & \\
   (\thetabf_{p,1})_{i} & (\thetabf_{p,2})_{i}  & \cdots & (\thetabf_{p,p})_{i}
    \end{bmatrix} \in \mathbb R^{p\times p}.
\end{equation}
It can be shown that $\Theta =  E^T \diag(U_1, \ldots, U_n) E$, where $E$ is the permutation matrix given by (\ref{eqn:E_matrix}).
Hence, $\Theta$ is PD (resp. PSD) if and only if each $U_i$ is PD (resp. PSD).

Let
\begin{eqnarray*}
  V & = & \underbrace{\begin{bmatrix} S^{-T}_n &  & &  & \\ & S^{-T}_n &  & &   \\ & & \ddots & & \\  & & & S^{-T}_n \end{bmatrix} }_{:=\mathbf S^{-T} } \Theta \underbrace{\begin{bmatrix} S^{-1}_n &  & &  & \\ & S^{-1}_n &  & &   \\ & & \ddots & & \\  & & & S^{-1}_n \end{bmatrix}}_{:=\mathbf S^{-1}} = \begin{bmatrix} V_{1, 1} & V_{1, 2} & \cdots & V_{1, p} \\
  V_{2, 1} & V_{2, 2} & \cdots & V_{2, p} \\
  \cdots &  & \cdots & \cdots \\
  V_{p, 1} & V_{p, 2} & \cdots & V_{p, p} \\
   \end{bmatrix},
\end{eqnarray*}
where $V_{i, j}  := S^{-T}_n \Theta_{i, j} S^{-1}_n$ is  symmetric and tridiagonal. Letting $E$ be the permutation matrix given by (\ref{eqn:E_matrix}),
a straightforward computation shows that $E^T V E$ is a symmetric block tridiagonal matrix given by
\[
  W = E^T V E = \begin{bmatrix} W_{1,1} & W_{1,2} &  & &  \\
     W_{2, 1} & W_{2, 2} &  W_{2,3} & &  \\
      &   \ddots & \ddots & \ddots &  \\
       & & W_{n-1, n-2} & W_{n-1, n-1} & W_{n-1, n} \\
      & & & W_{n, n-1} & W_{n, n}
    \end{bmatrix} \in \mathbb R^{np \times np},
\]
where each $W_{i, j} \in \mathbb R^{p \times p}$ is symmetric and $W_{i, j} = W_{j, i}$. Furthermore, for each $i=1, \ldots, n$ and $j \in \{i, i+1\}$,
\[
 W_{i, j} = \begin{bmatrix}  (V_{1,1})_{i,j} & (V_{1,2})_{i,j} & \cdots & (V_{1,p})_{i,j} \\
   (V_{2,1})_{i,j} & (V_{2,2})_{i,j} & \cdots & (V_{2,p})_{i,j} \\
   \cdots  &   \cdots   & \cdots & \\
   (V_{p,1})_{i,j} & (V_{p,2})_{i,j} & \cdots & (V_{p,p})_{i,j}
    \end{bmatrix} \in \mathbb R^{p\times p},
\]
where $(V_{r, s})_{i, j}$ denotes the $(i, j)$-entry of the block $V_{r, s}$.
In view of $V_{i, j}= S^{-T}_n \Theta_{i, j} S^{-1}_n$ and (\ref{eqn:triadigonal_mat_product}), we have
that $W_{i, i}=U_i + U_{i+1}$ and $W_{i, i+1}=-U_{i+1}$ for $i=1, \ldots, n-1$, and $W_{n, n}=U_n$.
Moreover, since $W=E^T V E = E^T \mathbf S^{-T} \Theta \, \mathbf S^{-1} E$, $W$ is PD (resp. PSD) if and only if $\Theta$ is PD (resp. PSD), which is also equivalent to that each $U_i$ is PD (resp. PSD); see the comment after (\ref{eqn:U_i_matrix}).

In what follows, we consider PSD (resp. PD) matrix decomposition for a PSD (resp. PD) $W$ generated by $\thetabf_{i, j} \in \mathbb R^n_+$ for $i=1, \ldots, n$ and $j \ge i$. The goal of this decomposition is to construct PSD  matrices $\wt W^s \in \mathbb R^{np \times np}$ for $s=1, \ldots, n$ such that the following conditions hold:
\begin{itemize}
   \item [(i)]
   \[
      \wt W^1 \, = \, \begin{bmatrix} (\wt W^1)_{1,1} & (\wt W^1)_{1,2} &  & &  \\
     (\wt W^1)_{2, 1} & (\wt W^1)_{2, 2} &  & &  \\
     & &   0 & \cdots & 0 &  \\
      & &  \vdots  &  \cdots & \vdots & \\
     &  &  0 & \cdots & 0
    \end{bmatrix}, \quad
    \wt W^n \, = \, \begin{bmatrix}
       0 & \cdots & 0 &  & & \\
      \vdots  &  \cdots & \vdots & & & \\
       0 & \cdots & 0 & & \\
    &  & & (\wt W^n)_{n-1,n-1} & (\wt W^n)_{n-1,n}   \\
    &  & & (\wt W^n)_{n, n-1} & (\wt W^n)_{n, n}  \\
    \end{bmatrix};
   \]

   \item [(ii)] for each $s=2, \ldots, n-1$,
     \[
        \wt W^s \, = \, \begin{bmatrix}
         {\mathbf 0}_{(i-2)p\times (i-2)p}  & & & & \\
      &     (\wt W^s)_{s-1,s-1} & (\wt W^s)_{s-1,s} & 0 & &  \\
      &   (\wt W^s)_{s, s-1} & (\wt W^s)_{s, s} & (\wt W^s)_{s, s+1} & &  \\
      &   0 & (\wt W^s)_{s+1, s} & (\wt W^s)_{s+1, s+1} &  & &  \\
     & & &  &  0 & \cdots & 0 &  \\
     &  & & & \vdots  &  \cdots & \vdots & \\
     & &  &  & 0 & \cdots & 0
    \end{bmatrix}; \ \ \ \mbox{and}
     \]
   \item [(iii)] $W = \sum^n_{s=1} \wt W^s$.
 \end{itemize}
For notational simplicity, let $\wh W^s$ denote the possibly nonzero block in each $\wt W^s$. Specifically,
\[
  \wh W^1 := \begin{bmatrix} (\wt W^1)_{1,1} & (\wt W^1)_{1,2} \\ (\wt W^1)_{2, 1} & (\wt W^1)_{2, 2} \end{bmatrix} \in \mathbb R^{2p \times 2p}, \qquad \wh W^n :=  \begin{bmatrix} (\wt W^n)_{n-1,n-1} & (\wt W^n)_{n-1,n} \\ (\wt W^n)_{n, n-1} & (\wt W^n)_{n, n} \end{bmatrix} \in \mathbb R^{2p \times 2p},
\]
and for each $s=2, \ldots, n-1$,
\[
  \wh W^s := \begin{bmatrix} (\wt W^s)_{s-1,s-1} & (\wt W^s)_{s-1,s} & 0 \\  (\wt W^s)_{s, s-1} & (\wt W^s)_{s, s} & (\wt W^s)_{s, s+1} \\  0 & (\wt W^s)_{s+1, s} & (\wt W^s)_{s+1, s+1} \end{bmatrix} \in \mathbb R^{3 p\times 3p}.
\]
When $W$ is PD, we also want  each $\wh W^s$ in the above decomposition to be PD.

\begin{proposition}
 Let $W$ be a PSD matrix generated by $\thetabf_{i, j} \in \mathbb R^n_+$ for $i=1, \ldots, n$. Then there exist PSD matrices $\wt W^s, \ s=1, \ldots, n$ satisfying the above conditions.
 Moreover, suppose $W$ is PD. Then there exist PD matrices $\wh W^s, \ s=1, \ldots, n$ such that their corresponding $\wt W^s$'s satisfy the above conditions.
\end{proposition}

\begin{proof}
 Let $W$ be generated by $\thetabf_{i, j}$'s such that $W$ is PSD, and let $U_i$'s be defined in (\ref{eqn:U_i_matrix}) corresponding to $\thetabf_{i, j}$'s. Note that each $U_i$ is PSD as $W$ is PSD. Let
 \[
      \wt W^1 \, = \, \begin{bmatrix} U_1 & 0 &  & &  \\
     0 & 0 &  & &  \\
     & &   0 & \cdots & 0   \\
      & &  \vdots  &  \cdots & \vdots  \\
     &  &  0 & \cdots & 0
    \end{bmatrix}, \qquad
    \wt W^n \, = \, \begin{bmatrix}
       0 & \cdots & 0 &  &  \\
      \vdots  &  \cdots & \vdots & &  \\
       0 & \cdots & 0 & & \\
    &  & & U_n & -U_n   \\
    &  & & -U_n & U_n  \\
    \end{bmatrix},
   \]
   and   for each $s=2, \ldots, n-1$,
     \[
        \wt W^s \, = \, \begin{bmatrix}
         {\mathbf 0}_{(i-2)p\times (i-2)p}  & & & & \\
%
      &     U_s & -U_s & 0 & &  \\
      &   -U_s & U_s & 0 & &  \\
      &   0 & 0 & 0 &  & &  \\
     & & &  &  0 & \cdots & 0 &  \\
     &  & & & \vdots  &  \cdots & \vdots & \\
     & &  &  & 0 & \cdots & 0
    \end{bmatrix}.
     \]
  Since each $U_i$ is PSD, so is $\wt W^s$ for each $s=1, \ldots, n$. Clearly, $W=\sum^n_{s=1} \wt W^s$.

 Now suppose $W$ is PD. Hence, each $U_i$ given by (\ref{eqn:U_i_matrix}) is PD.
Define
\[
  \breve W^1 := \frac{1}{2} \begin{bmatrix} U_1 + U_2 & -U_2 \\ -U_2 & U_2 \end{bmatrix},  \qquad \breve W^2 := \frac{1}{2} \begin{bmatrix} U_1+U_2 & -U_2 & 0 \\  -U_2 & U_2+U_3 & -U_3 \\  0 & -U_3 & U_3 \end{bmatrix}, \qquad \breve W^n := \frac{1}{2} \begin{bmatrix} U_n & -U_n \\ -U_n & U_n \end{bmatrix},
\]
and for each $s=3, \ldots, n-1$,
\[
   \breve W^s := \frac{1}{2} \begin{bmatrix} U_s & -U_s & 0 \\  -U_s & U_s+U_{s+1} & -U_{s+1} \\  0 & -U_{s+1} & U_{s+1} \end{bmatrix} \in \mathbb R^{3 p\times 3p}.
\]
Note that
$
  \Breve W^1 = \frac{1}{2} \left\{ \begin{bmatrix} U_1 & 0 \\ 0 & 0 \end{bmatrix} +  \begin{bmatrix}  U_2 & -U_2 \\ -U_2 & U_2 \end{bmatrix} \right\}
$
and the two matrices on the right hand side are both PSD and the intersection of their null spaces is the zero subspace. Hence, $\breve W^1$ is PD. Similarly,  $\breve W^2$ is PD, and the other $\breve W^s$'s are PSD. Since $\breve W^1$ is PD, we see that for an arbitrary $\delta_1 \in (0, \lambda_{\min}(\breve W^1))$, $\wh W^1:= \breve W^1 -\delta_1 \cdot I_{2p}$ is PD. Hence,
\[
  \Grave W^2 : = \breve W^2 + \delta_1 \cdot \begin{bmatrix} I_p &  &  \\  & I_p &  \\  &  & 0 \end{bmatrix}  = \frac{1}{2} \begin{bmatrix} U_1+U_2 + 2 \delta_1 \cdot I_p  & -U_2 & 0 \\  -U_2 & U_2+U_3 + 2 \delta_1 \cdot I_p  & -U_3 \\  0 & -U_3 & U_3 \end{bmatrix}
\]
is also PD. Therefore, for an arbitrary $\delta_2 \in (0, \lambda_{\min}(\Grave W^2))$, the matrix $\wh W^2 := \Grave W^2 - \delta_2 \cdot \begin{bmatrix} 0 &  &  \\  & I_p &  \\  &  & I_p \end{bmatrix}$
is PD.
Further, it is easy to show that the  matrix
$\Grave W^3 : = \breve W^3 + \delta_2 \cdot \begin{bmatrix} I_p &  &  \\  & I_p &  \\  &  & 0 \end{bmatrix} $
 is PD
such that for any $\delta_3 \in (0, \lambda_{\min}(\Grave W^3))$, the matrix
$
  \wh W^3 := \Grave W^3 - \delta_3 \cdot \begin{bmatrix} 0 &  &  \\  & I_p &  \\  &  & I_p \end{bmatrix}
$
is PD. Continuing this process by induction, we see that $\wh W^s$ is PD for all $s=4, \ldots, n-1$ and $\wh W^{n-1} := \Grave W^{n-1} - \delta_{n-1} \cdot \begin{bmatrix} 0 &  &  \\  & I_p &  \\  &  & I_p \end{bmatrix}$ is PD for an arbitrary $\delta_{n-1} \in (0, \lambda_{\min}(\Grave W^{n-1}))$, where $\Grave W^{n-1}$ is PD. Finally, define
$
 \wh W^n := \breve W^n + \delta_{n-1} \cdot I_{2p},
$
which is clearly PD. Using these PD $\wh W^s, \, s=1, \ldots, n$, we construct $\wt W^s$ by setting the possibly nonzero block in each $\wt W^s$ as $\wh W^s$.
Specifically,
\[
   \begin{bmatrix} (\wt W^1)_{1,1} & (\wt W^1)_{1,2} \\ (\wt W^1)_{2, 1} & (\wt W^1)_{2, 2} \end{bmatrix} = \wh W^1 \in \mathbb R^{2p \times 2p}, \qquad   \begin{bmatrix} (\wt W^n)_{n-1,n-1} & (\wt W^n)_{n-1,n} \\ (\wt W^n)_{n, n-1} & (\wt W^n)_{n, n} \end{bmatrix} = \wh W^n \in \mathbb R^{2p \times 2p}  ,
\]
and for each $s=2, \ldots, n-1$,
\[
  \begin{bmatrix} (\wt W^s)_{s-1,s-1} & (\wt W^s)_{s-1,s} & 0 \\  (\wt W^s)_{s, s-1} & (\wt W^s)_{s, s} & (\wt W^s)_{s, s+1} \\  0 & (\wt W^s)_{s+1, s} & (\wt W^s)_{s+1, s+1} \end{bmatrix} = \wh W^s  \in \mathbb R^{3 p\times 3p}.
\]
A straightforward calculation shows that $W=\sum^n_{s=1} \wt W^s$, yielding the desired result.
\end{proof}


To  obtain a desired decomposition using the above proposition, we observe that the matrix $V$ in (\ref{eqn:V_matrix}) is given by $\mathbf S^{-T} \Theta \mathbf S^{-1}$ for some matrix $\Theta$ of the form given below (\ref{eqn:triadigonal_mat_product}) whose blocks are positive combinations of $Q_{z, s}, Q_{z', s}$ and $Q_{w, s}$. Since $Q_{z, s}$ and $Q_{z', s}$ are diagonal and PSD and $Q_{w, s}$ are diagonal and PD, each  block of $\Theta$ is diagonal and PD or PSD. Moreover, by Lemma~\ref{lem:W_PD}, $W$ is  PD. Hence,  there are uncountably many ways to construct positive $\delta_s$, and thus PD $\wh W^s$, as shown in the above proposition.
Therefore,
we obtain the following strongly convex decomposition for the objective function $J$ in (\ref{eq:Multi_user}), where we set the constant $\gamma=0$ without loss of generality:
\[
  J(\mathbf u) = \frac{1}{2} \ubld^T W \ubld + c^T \ubld = \sum^n_{i=1} \frac{1}{2} \ubld^T \wt W^i \ubld + \sum^n_{i=1} c^T_{\Ical_i} \ubld_i = J_1(\ubld_1, \ubld_2) + \sum^{n-1}_{i=2} J_i(\ubld_{i-1}, \ubld_i, \ubld_{i+1}) + J_n(\ubld_{n-1}, \ubld_n)
\]
where the strongly convex functions $J_i$ are given by
\begin{eqnarray}
   J_1(\ubld_1, \ubld_2) & := & \frac{1}{2} \begin{bmatrix} \ubld^T_1 & \ubld^T_2 \end{bmatrix} \wh W_1 \begin{bmatrix} \ubld_1 \\ \ubld_2 \end{bmatrix} + c^T_{\Ical_1} \ubld_1, \notag \\
   J_i (\ubld_{i-1}, \ubld_{i}, \ubld_{i+1}) & := & \frac{1}{2} \begin{bmatrix} \ubld^T_{i-1} & \ubld^T_i & \ubld^T_{i+1} \end{bmatrix} \wh W_i \begin{bmatrix} \ubld_{i-1} \\ \ubld_i \\ \ubld_{i+1} \end{bmatrix} + c^T_{\Ical_i} \ubld_i, \qquad \forall \ i=2, \ldots, n-1, \label{eqn:J_i_linear_case} \\
   J_n(\ubld_{n-1}, \ubld_n) & := & \frac{1}{2} \begin{bmatrix} \ubld^T_{n-1} & \ubld^T_n \end{bmatrix} \wh W_n \begin{bmatrix} \ubld_{n-1} \\ \ubld_n \end{bmatrix} + c^T_{\Ical_n} \ubld_n. \notag
\end{eqnarray}
%

%
\begin{remark} \rm \label{remark:decompostion_extension}
Clearly, the above decomposition method is applicable to any vehicle communication network satisfying the assumption $\bf A.1$ in Section~\ref{sect:dynamics_constraint_topology},
i.e., $(i, i+1) \in \mathcal E$ for all $i=1, \ldots, n-1$.
Besides, various alternative approaches can be developed to construct PD matrices $\wh W^s$ using the similar idea given in the above proposition. Further, a similar decomposition method can be developed for other vehicle communication networks different from the the cyclic-like graph. For instance, if such a graph contains edges other than $(i, i+1) \in \mathcal E$, one can add or subtract certain small terms pertaining to these extra edges in relevant matrices, which will preserve the desired PD property.
\end{remark}

In what follows,  we write each $J_i$ as $J_i(\ubld_i, (\mathbf u_j)_{j \in \mathcal N_i})$ for notational convenience, where $\mathcal N_i$ denotes the set of neighbors of vehicle $i$ in a vehicle network such that $i-1, i+1 \in \mathcal N_i$ for $i=2, \ldots, n-1$ and $2 \in \mathcal N_1, n-1 \in \mathcal N_n$.

%
\subsection{Operator Splitting Method based Fully Distributed Algorithms} \label{subsect:ist_scheme_linear}

For illustration simplicity, we consider the cyclic like network topology through this subsection, although the proposed schemes can be easily extended to other network topologies under a suitable assumption ((cf. Remark~\ref{remark:decompostion_extension}). In this case, $\mathcal N_1 = \{ 2\}$, $\mathcal N_{n} = \{ n-1 \}$, and $\mathcal N_{i} = \{ i-1, i+1 \}$ for $i=2, \ldots, n-1$.

Define the constraint set
\begin{align*}
 \mathcal P & \, := \, \{ \mathbf u=(\mathbf u_1, \ldots, \mathbf u_n) \in \mathbb R^{np} \, | \, \mathbf u_i \in \mathcal X_i, \ \mathbf u \in \mathcal Z_i, \ i=1, \ldots, n\}.
\end{align*}
Recall that $\mathcal P$ is defined by convex quadratic functions. The underlying optimization problem (\ref{eq:Multi_user}) at time $k$ becomes $\min_{\mathbf u} J(\mathbf u)$ subject to $\mathbf u \in \mathcal P$.

We formulate this problem as a locally coupled convex optimization problem \cite{HuXiaoLiu_CDC18} and  solve it using fully distributed algorithms. Specifically, in view of the decompositions given by (\ref{eqn:J_i_linear_case}), the objective function in (\ref{eq:Multi_user}) can be written as
\[
   J(\mathbf u) = \sum^n_{i=1} J_i ( \mathbf u_i, (\mathbf u_j)_{j \in \mathcal N_i}),
\]
In view of Remark~\ref{remark:local_coupled_constraint}, the safety constraints are also locally coupled.
Let  $\deltabf_S$ denote the indicator function of a (closed convex) set $S$. Define, for each $i=1, \ldots, n$,
\begin{align*}
  \wh J_{i}( \mathbf u_i, (\mathbf u_j)_{j \in \mathcal N_i}) & \, := \,   J_i (\mathbf u_i, (\mathbf u_j)_{j \in \mathcal N_i} ) + \deltabf_{\mathcal X_i}(\ubld_i) + \deltabf_{\mathcal Z_i}(\mathbf u_{i-1}, \mathbf u_i).
\end{align*}
%
As in \cite{HuXiaoLiu_CDC18}, define $\wh{\mathbf u}_i:= \big(\mathbf u_i, (\mathbf u_{i, j} )_{j \in \mathcal N_i} \big)$, where the new variables $\mathbf u_{i, j}$ represent the predicted values of $\mathbf u_j$ of vehicle $j$ in the neighbor of vehicle $i$, and let $\wh {\mathbf u} := ( \wh{\mathbf u}_i)_{i=1, \ldots, n} \in \mathbb R^{\ell}$.
 Define the consensus subspace
\[
   \mathcal A \, := \, \{ \wh{\mathbf u} \, | \,  \mathbf u_{i, j} = \mathbf u_j, \ \forall \, (i, j) \in  \Ecal \}.
\]
Then the underlying optimization problem  (\ref{eq:Multi_user}) can be equivalently written as
\[
\min_{\wh{\mathbf u} } \ \ \sum^n_{i=1} \wh J_{i}(\wh{\mathbf u}_i), \quad \mbox{ subject to } \quad \wh{\mathbf u} \in \mathcal A.
\]
Let  $\Pcal_{i}:=\{ \wh{\mathbf u}_i \, | \, \mathbf u_i \in \mathcal X_i, \ (H_i(\mathbf u_{i, i-1}, \mathbf u_i))_s\le 0, \, \forall \, s=1, \ldots, p \}$ for notational simplicity. Then  the underlying optimization \tblue{problem} becomes
%
\begin{equation} \label{eqn:opt_lin_slitting}
\min_{\wh{\mathbf u} } \, F(\wh\ubld):= \sum^n_{i=1} J_{i}(\wh{\mathbf u}_i) +  \sum^n_{i=1} \deltabf_{\Pcal_{i}}(\wh{\mathbf u}_i)+ \deltabf_{\mathcal A}(\wh{\mathbf u}),
\end{equation}
%
 where $F:\mathbb R^{\ell} \rightarrow \mathbb R\cup\{+\infty\}$ denotes the extended-valued objective function. Thus $F$ is the sum of two indictor functions of closed convex sets and the convex quadratic function given by $J(\wh{\mathbf u}):= \sum^n_{i=1} J_{i} (\wh{\mathbf u}_i)$, by slightly abusing the notation. Note that $\mathcal A$ is polyhedral. It is easy to show via Corollary~\ref{coro:nonempty_interior_general_p} that the Slater's condition holds under the mild assumptions given in Corollary~\ref{coro:nonempty_interior_general_p}, e.g., $v_0(k)>v_{\min}$ for all $k \in \mathbb Z_+$.
Hence, by \cite[Corollary 23.8.1]{Rockafellar_book70}, $\partial F(\wh\ubld) = \sum^n_{i=1}\Big( \nabla J_i (\wh{\mathbf u}_i) +  \mathcal N_{\Pcal_{i}} (\wh{\mathbf u}_i) \Big) + \mathcal N_{\mathcal A}(\wh{\mathbf u})$  in light of $\partial \deltabf_C(x) = \mathcal N_C(x)$, where $\mathcal N_C(x)$ denotes the normal cone of a closed convex set $C$ at $x \in C$. Finally, the formulation given by (\ref{eqn:opt_lin_slitting}) is  a locally coupled convex optimization problem; see Section~\ref{subsect:locally_coupled_opt_formulation}. This formulation allows one to develop fully distributed schemes. Particularly,
in fully distributed computation, each vehicle $i$ only knows $\wh{\mathbf u}_i$ and $\wh J_i$ (i.e., $J_i$ and $\mathcal P_i$) but does not know  $\wh{\mathbf u}_j$ and $\wh J_j$ with $j\ne i$. Each vehicle $i$ will exchange information with its neighboring vehicles to update $\wh{\mathbf u}_i$ via a distributed scheme.

\begin{algorithm}
\caption{Generalized Douglas-Rachford Splitting Method based Fully Distributed Algorithm}
\begin{algorithmic}[1]
\label{algo:DR_distributed_splitting}
%
\STATE Choose constants $0<\alpha <1$ and $\rho>0$

\STATE Initialize $k=0$, and choose an initial point $z^0$

\WHILE{the stopping criteria is not met}

  \FOR {$i=1, \ldots, n$}

   \STATE Compute $\ol z^k_i$ using equation (\ref{eqn:consensus_projection}), and let $w^{k+1}_i \leftarrow \ol z^k_i$

  \ENDFOR

  \FOR {$i=1, \ldots, n$}

   \STATE {$z^{k+1}_i \leftarrow z^k_i +  2 \alpha \cdot \Big[ \mbox{Prox}_{\rho \wh J_i} \big( 2 w^{k+1}_i - z^k_i\big ) - w^{k+1}_i \Big]$
   }

  \ENDFOR

   \STATE $k\leftarrow k+1$

\ENDWHILE

\RETURN $\wh{\mathbf u}^* = w^k$

\end{algorithmic}
\end{algorithm}

We introduce more notation first. For a proper, lower semicontinuous convex function $f:\mathbb R^n \rightarrow \mathbb R \cup \{ +\infty \}$,  let $\mbox{Prox}_{f}(\cdot)$ denote the proximal operator, i.e., for any given $x \in \mathbb R^n$,
\[
  \mbox{Prox}_{f}(x) \, := \, \argmin_{z \in \mathbb R^n} \ f(z) + \frac{1}{2} \| z - x \|^2_2.
\]
Further, $\Pi_C$ denotes the Euclidean projection onto a closed convex set $C$. Using this notation, we present a specific operator splitting method based distributed scheme for solving (\ref{eqn:opt_lin_slitting}). By grouping the first two sums (with separable variables) in the objective function of (\ref{eqn:opt_lin_slitting}), we apply the generalized Douglas-Rachford splitting algorithm \cite{HuXiaoLiu_CDC18}. Recall that $\wh J_i(\wh{\mathbf u}_i) := J_i(\wh{\mathbf u}_i) + \deltabf_{\Pcal_{i}}(\wh{\mathbf u}_i)$ for each $i=1, \ldots, n$. For any constants $\alpha$ and $\rho$ satisfying $0<\alpha< 1$ and $\rho>0$, this algorithm is given by:
\begin{align*}
   w^{k+1} \, = \, \Pi_{\mathcal A} (z^k), \qquad
   z^{k+1} \, = \, z^k + 2 \alpha \cdot \Big[ \mbox{Prox}_{\rho \wh J_1+ \cdots + \rho \wh J_n} \big( 2 w^{k+1} - z^k\big ) - w^{k+1} \Big].
\end{align*}
It is shown in \cite{DavisYin_SVA17, HuXiaoLiu_CDC18} that the sequence $(w^k)$ converges to the unique minimizer $\mathbf {\wh u}^*$ of the underlying optimization problem (\ref{eqn:opt_lin_slitting}). In the above scheme, $\Pi_{\mathcal A}$ is the orthogonal projection onto the consensus subspace $\mathcal A$ such that the following holds: for any $\wh {\mathbf u} := ( \wh{\mathbf u}_1, \ldots, \wh{\mathbf u}_n)$ where $\wh{\mathbf u}_i:= \big(\mathbf u_i, (\mathbf u_{ij} )_{j \in \mathcal N_i} \big)$, $\ol  {\mathbf u} := \Pi_{\mathcal A}(\wh {\mathbf u})$ is given by  \cite[Section IV]{HuXiaoLiu_CDC18}:
\begin{equation} \label{eqn:consensus_projection}
   \ol {\mathbf u}_j = \ol {\mathbf u}_{ij} =  \frac{1}{1+|\mathcal N_j| } \Big(  \wh{\mathbf u}_j  + \sum_{k \in \mathcal N_j} \wh{\mathbf u}_{kj} \Big), \qquad \forall \, (i, j) \in \mathcal E.
\end{equation}
Furthermore, due to the decoupled structure of $\wh J_i$'s, we obtain the distributed version of the above algorithm, which is also summarized in Algorithm~\ref{algo:DR_distributed_splitting}:
\begin{subequations} \label{eqn:DR_scheme}
\begin{align}
   w^{k+1}_i & \, = \, \ol z^k_i, \qquad i=1, \ldots, n; \\
   z^{k+1}_i & \, = \, z^k_i + 2 \alpha \cdot \Big[ \mbox{Prox}_{\rho \wh J_i} \big( 2 w^{k+1}_i - z^k_i\big ) - w^{k+1}_i \Big], \quad i=1, \ldots, n.
\end{align}
\end{subequations}

In the distributed scheme (\ref{eqn:DR_scheme}), the first step (or Line 5 of Algorithms~\ref{algo:DR_distributed_splitting}) is a consensus step, and the consensus computation is carried out in a fully distributed and synchronous manner as indicated in \cite[Section IV]{HuXiaoLiu_CDC18}. The second step in (\ref{eqn:DR_scheme})
(or Line 8  of Algorithms~\ref{algo:DR_distributed_splitting}) does not need inter-agent communication \cite{HuXiaoLiu_CDC18} and is performed using local computation only.  For effective computation in the second step, recall that $\wh J_i(\wh{\mathbf u}_i) := J_i(\wh{\mathbf u}_i) + \deltabf_{\Pcal_{i}}(\wh{\mathbf u}_i)$  for each $i=1, \ldots, n$ such that the proximal operator $\mbox{Prox}_{\rho \wh J_i}(\wh{\mathbf u}_i)$ becomes
\[
     \mbox{Prox}_{\rho \wh J_i}(\wh{\mathbf u}_i) \, = \, \argmin_{z \in \Pcal_{i} }  \ J_i(z) + \frac{1}{2 \rho} \| z - \wh{\mathbf u}_i \|^2_2.
\]
Since $\Pcal_i$ is the intersection of a polyhedral set and a quadratically constrained convex set and $J_i$ is a quadratic convex function, $\mbox{Prox}_{\rho \wh J_i}(\wh{\mathbf u}_i)$ is formulated as a QCQP and can be solved via a second order cone program \cite{AlizadehG_MPB03} or a  semidefinite program.
Efficient numerical packages, e.g., SeDuMi \cite{Sturm_OMS99}, can be used for solving the QCQP.
Lastly, a typical (global) stopping criterion in the scheme (\ref{eqn:DR_scheme}) (or Algorithm~\ref{algo:DR_distributed_splitting}) is defined by the error tolerance of two neighboring $z^k$'s, i.e.,  $\| z^{k+1}- z^k\|_2 \le \varepsilon$, where $\varepsilon>0$ is an error tolerance. For distributed computation, one can use its local version, i.e., $\| z^{k+1}_i- z^k_i\|_2 \le \varepsilon/n$, as a stopping criterion for each vehicle.


\begin{remark} \rm \label{remark:other_schemes}
  Other distributed algorithms can be used to solve the underlying optimization problem (\ref{eqn:opt_lin_slitting}).
   For example, the three operator splitting schemes developed in \cite{DavisYin_SVA17} can be applied. To describe such schemes,
    let $\wh L:=\max_{i=1, \ldots, n} \|\wh W^i\|_2>0$.  the Hessian $H J(\wh{\mathbf u}) = \mbox{diag}( \wh W^i)_{i=1, \ldots, n}$.
Hence, $\nabla J$ is $\wh L$-Lipschitz continuous and thus $1/L$-cocoercive. Further, the two indicator functions are proper, closed, and convex functions. Choose the constants $\gamma, \lambda$ such that $0<\gamma < 2/\wh L$ and $0< \lambda < 2- \frac{\gamma \wh L}{2}$. Then for any initial condition $z^0$, the iterative scheme is given by  \cite[ Algorithm 1]{DavisYin_SVA17}:
\begin{align*}
   w^{k+1} \, = \, \Pi_{\mathcal A} (z^k), \qquad
   z^{k+1} \, = \, z^k + \lambda \cdot \Big[ \Pi_{\Pcal_{1} \times \cdots \times \Pcal_{n}} \big( 2 w^{k+1} - z^k - \gamma \nabla J(w^{k+1}) \big ) - w^{k+1} \Big].
\end{align*}
In view of the similar discussions for consensus computation and decoupled structure of
the projection $\Pi_{\Pcal_{1} \times \cdots \times \Pcal_{n}}$, we obtain the distributed version of the above algorithm:
\begin{align*}
   w^{k+1}_i & \, = \, \ol z^k_i, \qquad i=1, \ldots, n; \\
   z^{k+1}_i & \, = \, z^k_i + \lambda \cdot \Big[ \Pi_{\Pcal_{i}} \big( 2 w^{k+1}_i - z^k_i - \gamma \cdot [ \wh W^{i} w^{k+1}_i + c_{\Ical_i} ] \big ) - w^{k+1}_i \Big], \quad i=1, \ldots, n.
\end{align*}
In this scheme, the Euclidean projection $\Pi_{\Pcal_{q,i}}$ can be formulated as a QCQP and be solved via a second order cone program.

When each $\wh W^i$ is PD, each $J_i$ is strongly convex. Thus $\nabla J$ is $\mu$-strongly monotone with $\mu=\min_{i=1, \ldots, n}\lambda_{\min}( \wh W^i )$, i.e., $(x-y)^T(\nabla J(x) - \nabla J(y)) \ge \mu \|x- y \|^2_2, \forall \, x, y$. Since the subdifferential of the indicator function of a closed convex set is monotone, an accelerated scheme developed in \cite[Algorithm 2]{DavisYin_SVA17} can be exploited. In particular,  let $\eta$ be a constant with $0<\eta<1$, and $\gamma_0 \in (0, 2/(\wh L\cdot (1-\eta))$. Set the initial points for an arbitrary $z^0$, $w^0=\Pi_{\mathcal A}(z^0)$ and $v^0=(z^0-w^0)/\gamma_0$.
The distributed version of this scheme is  given by:
\begin{align*}
   w^{k+1}_i & \, = \, \ol z^k_i + \gamma_k \ol v^k_i, \qquad i=1, \ldots, n; \\
   v^{k+1}_i & \, = \, \frac{1}{\gamma_k} \big( z^k_i + \gamma_k v^k_i- w^{k+1}_i \big), \qquad i=1, \ldots, n; \\
   \gamma_{k+1} &  \, = \, -  \wt\mu \gamma^2_k + \sqrt{  (\wt \mu \gamma^2_k)^2 + \gamma^2_k}; \\
   z^{k+1}_i & \, =  \Pi_{\Pcal_{i}} \Big( w^{k+1}_i - \gamma_{k+1} v^{k+1}_i - \gamma_{k+1}[ \wh W^i w^{k+1}_i + c_{\Ical_i}] \Big), \quad i=1, \ldots, n,
\end{align*}
where $\wt \mu :=\eta \cdot \mu$. It is shown in  \cite[Theorem 1.2]{DavisYin_SVA17} that $(w^k)$ converges to the unique minimizer $\mathbf {\wh u}^*$ with $\| w^k - \wh {\mathbf u}^*\|_2 = O(1/(k+1))$.
\end{remark}

%
\section{Control Design and Stability Analysis of the Closed Loop  Dynamics} \label{sect:control_analysis}

In this section, we discuss how to choose the weight matrices $Q_{z, s}$,  $Q_{z', s}$ and $Q_{w, s}$ to achieve the desired closed loop performance, including stability and traffic transient dynamics. For the similar reasons given in  \cite[Section 5]{GShenDu_TRB16}, we focus on the constraint free case.

Under the linear vehicle dynamics, the closed-loop system is also a linear system. Specifically, the linear closed-loop dynamics are given by
\begin{equation} \label{eqn:dynamics_vector_form}
    z(k+1) \, = \, z(k) + \tau z'(k) + \frac{\tau^2}{2} w(k), \qquad z'(k+1) \, = \, z'(k) + \tau w(k),
\end{equation}
where $w(k)$ is a unique solution to an unconstrained optimization problem arising from the MPC and is a linear function of $z(k)$ and $z'(k)$ to be determined as follows.

\gap

\noindent{\bf Case (i): $p=1$}. In this case, we write $Q_{z, 1}, Q_{z', 1}, Q_{w, 1}$ as $Q_{z}, Q_{z'}, Q_{w}$ respectively. Then the objective function becomes
\[
   J(w(k)) \, = \, \frac{1}{2} \Big[ z^T(k+1) Q_z, z(k+1) + (z'(k+1))^T Q_{z'} z'(k+1) \Big] + \frac{\tau^2}{2} \wt w^T(k) Q_w \wt w(k),
\]
where we recall that $\wt w(k) = w(k) - u_0(k) \mathbf e_1$.
It follows from the similar argument in \cite[Section 5]{GShenDu_TRB16} that the the closed-loop  system is given by the following linear system:
%
\begin{eqnarray} \label{eqn:linear_closed_loop}
\begin{bmatrix} z(k+1) \\ z'(k+1) \end{bmatrix} \, = \, \underbrace{ \begin{bmatrix} I_n - \frac{\tau^2}{4} \wh W Q_z  & \tau  I_n -  \wh W \Big( \frac{\tau^3}{4} Q_z + \frac{\tau}{2} Q_{z'} \Big)  \\ -\frac{\tau}{2} \wh W Q_z & I_n - \wh W \Big( \frac{\tau^2}{2} Q_z +  Q_{z'} \Big) \end{bmatrix} }_{A_{\mbox{c}}} \begin{bmatrix} z(k) \\ z'(k) \end{bmatrix}+ \begin{bmatrix} \frac{\tau^2}{2} I_n \\ \tau I_n \end{bmatrix} \wh W Q_w \mathbf e_1 \cdot u_0(k),
\end{eqnarray}
%
where $A_{\mbox{c}}$ is the closed loop dynamics matrix, and
\begin{equation} \label{eqn:W_matrix}
  \wh W \, := \,  \left[ \frac{ \tau^2 Q_z }{4} +  Q_{z'}  + Q_w \right]^{-1}.
\end{equation}
%

%
%
The matrix $A_{\mbox{c}}$ in (\ref{eqn:linear_closed_loop}) plays an important role in the closed loop stability and desired transient dynamical performance.
Since $Q_{z}, Q_{z'}$ and $Q_{w}$ are all  diagonal and PSD (resp. PD),
we have $Q_{z} = \mbox{diag}(\alphabf)$, $Q_{z'}=\mbox{diag}(\betabf)$, and $Q_{w}=\mbox{diag}(\zetabf)$, where $\alphabf, \betabf\in \mathbb R^n_{+}$ and $\zetabf \in \mathbb R^n_{++}$ with $\alphabf =(\alpha_i)^n_{i=1}$, $\betabf=(\beta_i)^n_{i=1}$, and $\zetabf=(\zeta_i)^n_{i=1}$.
Hence, we write the matrix $A_{\mbox{c}}$ as $A_{\mbox{c}}(\alphabf, \betabf, \zetabf, \tau)$ to emphasize its dependence on these parameters.
The following result asserts the asymptotic stability of the linear closed-loop dynamics; its proof resembles that for \cite[Proposition 5.1]{GShenDu_TRB16} and  is thus omitted.

\begin{proposition} \label{prop:stability_closed_loop}
 Given any $\tau \in \mathbb R_{++}$ and any $\alphabf, \betabf, \zetabf \in \mathbb R^n_{++}$, the matrix $A_{\mbox{c}}(\alphabf, \betabf, \zetabf, \tau)$ is Schur stable, i.e., each eigenvalue $\mu \in \mathbb C$ of $A_{\mbox{c}}(\alphabf, \betabf, \zetabf, \tau)$ satisfies $|\mu |<1$. Moreover, for any eigenvalue $\mu_i$ of $A_{\mbox{c}}$,  the following hold:
 \begin{itemize}
   \item [(1)] if $\mu_i$ is non-real, then $|\mu_i|^2 = \frac{\zeta_i }{d_i}$, where $d_i:=\frac{\alpha_i \tau^2 }{4} + \beta_i  + \zeta_i$;
   \item [(2)] if $\mu_i$ is real, then $ 1 -  ( \frac{\alpha_i \tau^2 }{2}   + \beta_i ) \frac{1}{d_i} < \mu_i < 1- \frac{\alpha_i \tau^2}{4 d_i}$.
 \end{itemize}
\end{proposition}

\gap

\noindent{\bf Case (ii): $p>1$}. Fix $k$. For a general $p \in \mathbb N$, let $\mathbf w:=(w(k), \ldots, w(k+p-1))\in \mathbb R^{np}$. Recall that for each $s=1, \ldots, p$,
\[
z(k+s) \, = \, z(k) + s \tau z'(k) + \tau^2 \sum^{s-1}_{j=0} \frac{2(s-j)-1}{2} w(k+j), \quad
z'(k+s) = z'(k) + \tau \sum^{s-1}_{j=0} w(k+j),
\]
and with slightly abusing notation, the objective function is
\begin{eqnarray*}
    J( \underbrace{ w(k), \ldots, w(k+p-1)}_{ \mathbf w} ) & = & \frac{1}{2} \sum^p_{s=1} \Big( \tau^2 \wt w^T(k+s-1) Q_{w, s} \wt w(k+s-1) \\
   & &  \qquad \quad + \ z^T(k+s) Q_{z, s} z(k+s) \, +  \, (z'(k+s))^T Q_{z', s} z'(k+s) \Big),
\end{eqnarray*}
where $\wt w(k+s) := w(k+s) - u_0(k) \cdot \mathbf e_1$  introduced in Remark~\ref{remark:objective_func}. It follows from the similar development in Section~\ref{subsect:optimization_linear_dynamics} that
\[
  J(\mathbf w) = \frac{1}{2} \mathbf w^T \mathbf H \mathbf w + \mathbf w^T \left( \mathbf G \begin{bmatrix} z(k) \\ z'(k) \end{bmatrix} - u_0(k) \mathbf g \right) + \wt \gamma,
\]
where $\wt \gamma$ is a constant. By a similar argument as in Lemma~\ref{lem:W_PD}, it can be shown that $\mathbf H \in \mathbb R^{pn\times pn}$ is a symmetric PD matrix. Further, it resembles the matrix $V$ in (\ref{eqn:V_matrix}) (by replacing $S^{-1}_n$ with $I_n$), i.e.,
\begin{equation} \label{eqn:Hbld_matrix}
    \mathbf H \, = \,  \begin{bmatrix} \breve\Hbld_{1,1} +\tau^2 Q_{w, 1}  & \breve\Hbld_{1, 2} & \breve\Hbld_{1,3} & \cdots & \cdots & \breve\Hbld_{1, p} \\ \breve\Hbld_{2,1} & \breve\Hbld_{2,2} +\tau^2 Q_{w, 2}  & \breve\Hbld_{2, 3} & \cdots & \cdots & \breve\Hbld_{2, p} \\ \cdots &  & \cdots &  & \cdots & \\ \cdots  &  & \cdots &  & \cdots & \\ \breve\Hbld_{p,1} & \breve\Hbld_{p, 2} & \breve\Hbld_{p, 3} & \cdots & \cdots & \breve\Hbld_{p, p} +\tau^2 Q_{w, p}  \end{bmatrix} \in \mathbb R^{np \times np},
\end{equation}
where $\breve\Hbld_{i, j}$'s are diagonal PD matrices given by
\[
 \breve\Hbld_{i, j}  \, := \, \sum^p_{s=\max(i, j)} \Big( \frac{\tau^4}{4}[2(s-i)+1]\cdot [2 (s-j)+1] Q_{z, s} + \tau^2 Q_{z', s} \Big)  \in \mathbb R^{n\times n}.
\]
%
Moreover, it follows from (\ref{eqn:d_s_f_s}) and (\ref{eqn:lin_term_in_J}) that the matrix $\mathbf G$ and  constant vector $\mathbf g$ are
\begin{equation} \label{eqn:G_bld}
  \mathbf G \, := \, \begin{bmatrix} \mathbf G_{1, 1} &  \mathbf G_{1, 2} \\ \vdots & \vdots \\ \mathbf G_{p, 1} & \mathbf G_{p, 2} \end{bmatrix} \in \mathbb R^{pn \times 2 n}, \qquad
  \mathbf  g \, := \, \tau^2 \begin{bmatrix} Q_{w, 1} \mathbf e_1 \\ \vdots \\ Q_{w, p} \mathbf e_1 \end{bmatrix} \in \mathbb R^{pn}.
\end{equation}
where $\mathbf G_{i, 1}, \mathbf G_{i, 2} \in \mathbb R^{n\times n}$ are given by: for each $i=1, \ldots, p$,
\begin{eqnarray*}
 \mathbf G_{i, 1} & = & \tau^2 \sum^p_{s=i} \frac{ 2(s-i) + 1 }{2} Q_{z, s}, \qquad \quad
 \mathbf G_{i,2} \, = \, \tau^3 \sum^p_{s=i}s  \frac{ 2(s-i) + 1 }{2}  Q_{z, s} + \tau \sum^p_{s=i} Q_{z', s}.
\end{eqnarray*}
%
Hence, the optimal solution is  $\mathbf w_*=(w_*(k), w_*(k+1), \ldots, w_*(k+p-1))=-\mathbf H^{-1} \Big(\mathbf G \begin{bmatrix} z(k) \\ z'(k) \end{bmatrix} -  u_0(k) \mathbf g \Big)$, and
$w_*(k)= - \begin{bmatrix} I_n & 0 & \cdots & 0 \end{bmatrix} \mathbf H^{-1} \Big(\mathbf G \begin{bmatrix} z(k) \\ z'(k) \end{bmatrix} -  u_0(k) \mathbf g \Big)$.
Define the matrix $\mathbf K$ and the vector $\mathbf d$ as
\begin{equation} \label{eqn:K_mat_d_vec}
 \mathbf K := - \begin{bmatrix} I_n & 0 & \cdots & 0 \end{bmatrix} \mathbf H^{-1} \mathbf G \in \mathbb R^{n\times 2n}, \qquad \mathbf d:= \begin{bmatrix} I_n & 0 & \cdots & 0 \end{bmatrix} \mathbf H^{-1} \mathbf g \in \mathbb R^{n}.
\end{equation}
The closed loop system becomes
\[
   \begin{bmatrix} z(k+1) \\ z'(k+1) \end{bmatrix} = \underbrace{\left\{ \begin{bmatrix} I_n & \tau I_n \\ 0 & I_n \end{bmatrix} + \begin{bmatrix} \frac{\tau^2}{2} I_n \\ \tau I_n \end{bmatrix} \mathbf  K \right \} }_{A_{\mbox{c}}}  \begin{bmatrix} z(k) \\ z'(k) \end{bmatrix} +  \begin{bmatrix} \frac{\tau^2}{2} I_n \\ \tau I_n \end{bmatrix}  u_0(k) \cdot \mathbf d,
\]
where $A_{\mbox{c}}$ is the closed loop dynamics matrix, and the subscript of $A_{\mbox{c}}$ represents the closed loop.

Since  $Q_{z, s}, Q_{z', s}$ are diagonal PSD and $Q_{w, s}$ are diagonal PD for all $s=1, \ldots, p$, we write them as $Q_{z, s} = \mbox{diag}(\alphabf^s)$, $Q_{z', s} = \mbox{diag}(\betabf^s)$, and $Q_{w, s} = \mbox{diag}(\zetabf^s)$, where $\alphabf^s, \betabf^s \in \mathbb R^n_{+}$ and $\zetabf^s \in \mathbb R^n_{++}$ for all $s=1, \ldots, p$ with  $\alphabf^s =(\alpha^s_i)^n_{i=1}$, $\betabf^s=(\beta^s_i)^n_{i=1}$, and $\zetabf^s=(\zeta^s_i)^n_{i=1}$ for each $s$.
Let $\boldsymbol\alpha:=(\alphabf^1, \ldots, \alphabf^p)$, $\boldsymbol\beta:=(\betabf^1, \ldots, \betabf^p)$, and $\boldsymbol\zeta:=(\zetabf^1, \ldots, \zetabf^p)$. We write the matrix $A_{\mbox{c}}$ as $A_{\mbox{c}}(\boldsymbol\alpha, \boldsymbol\beta, \boldsymbol\zeta, \tau)$ to emphasize its dependence on these parameters.

It can be shown that there exists a permutation matrix $\wh E \in \mathbb R^{2n \times 2n}$ such that $\wt A := \wh E^T  A_{\mbox{c}} \wh E$ is a block diagonal matrix, i.e., $\wt A = \mbox{diag}( \wt A_1, \wt A_2, \ldots, \wt A_n)$
whose each block $\wt A_i \in \mathbb R^{2\times 2}$ is given by
\[
 \wt A_i \, = \, \begin{bmatrix} 1 & \tau \\ 0 & 1 \end{bmatrix} +  \begin{bmatrix} \frac{\tau^2}{2} \\ \tau \end{bmatrix} \wt{ \mathbf  K}_i, \qquad \forall \ i=1, \ldots, n.
\]
Here for each $i=1, \ldots, n$, $\wt { \mathbf  K}_i := - \mathbf e^T_1 \, \wt{\mathbf H}^{-1}_i \, \wt{\mathbf G}_i \in \mathbb R^{1\times 2}$, where
\[
   \wt \Hbld_i \, := \, \begin{bmatrix} (\breve\Hbld_{1,1} +\tau^2 Q_{w, 1})_{i, i}  & (\breve\Hbld_{1, 2})_{i, i} & (\breve\Hbld_{1,3})_{i, i} & \cdots & \cdots & (\breve\Hbld_{1, p})_{i, i} \\ (\breve\Hbld_{2,1})_{i, i} & (\breve\Hbld_{2,2} +\tau^2 Q_{w, 2})_{i, i}  & (\breve\Hbld_{2, 3})_{i, i} & \cdots & \cdots & (\breve\Hbld_{2, p})_{i, i} \\ \cdots &  & \cdots &  & \cdots & \\ \cdots  &  & \cdots &  & \cdots & \\ (\breve\Hbld_{p,1})_{i, i} & (\breve\Hbld_{p, 2})_{i, i} & (\breve\Hbld_{p, 3})_{i, i} & \cdots & \cdots & (\breve\Hbld_{p, p} +\tau^2 Q_{w, p})_{i, i}  \end{bmatrix} \in \mathbb R^{p \times p},
\]
and
\[
  \wt \Gbld_i \, := \, \begin{bmatrix} (\mathbf G_{1, 1})_{i, i} &  (\mathbf G_{1, 2})_{i, i} \\ \vdots & \vdots \\ (\mathbf G_{p, 1})_{i, i} & (\mathbf G_{p, 2})_{i, i} \end{bmatrix} \in \mathbb R^{p \times 2}.
\]
Note that $\wt{\mathbf H} = \mbox{diag}( \wt{\mathbf H}_1, \wt{\mathbf H}_2, \ldots, \wt{\mathbf H}_n)  = E^T \mathbf H  E $ for the permutation matrix $E \in \mathbb R^{pn\times pn}$ given by (\ref{eqn:E_matrix}). Since $\mathbf H$ is PD, so are all the $\wt \Hbld_i$'s.

As examples, we give the closed form expressions of $\wt \Hbld_i$ and $\wt \Gbld_i$ for some small $p$'s.
When $p=1$, $\wt{\mathbf H}_i = \tau^2 \big( \frac{\tau^2}{4} \alpha^1_i + \beta^1_i + \zeta^1_i\big)$ and $\wt{\mathbf G}_i =\begin{bmatrix} \frac{\tau^2}{2} \alpha^1_i & \frac{\tau^3}{2} \alpha^1_i + \tau \beta^1_i \end{bmatrix}$ for each $i=1, \ldots, s$. When $p=2$, we have, for each $i=1, \ldots, n$,
\[
  \wt{\mathbf H}_i = \tau^2 \begin{bmatrix} \frac{\tau^2}{4} \alpha^1_i +  \frac{9\tau^2}{4} \alpha^2_i + \beta^1_i + \beta^2_i + \zeta^1_i &  \frac{3\tau^2}{4} \alpha^2_i + \beta^2_i \\
  \frac{3\tau^2}{4} \alpha^2_i + \beta^2_i &  \frac{\tau^2}{4} \alpha^2_i + \beta^2_i + \zeta^2_i \end{bmatrix} \in \mathbb R^{2\times 2},
\]
and
\[
   \wt{\mathbf G}_i =  \begin{bmatrix} \tau^2 \frac{ \alpha^1_i + 3 \alpha^2_i }{2} &  \frac{\tau^3}{2} \alpha^1_i + 3 \tau^3 \alpha^2_i + \tau (\beta^1_i+ \beta^2_i) \\
  \frac{\tau^2}{2} \alpha^2_i  &  \tau^3 \alpha^2_i + \tau \beta^2_i \end{bmatrix} \in \mathbb R^{2\times 2}.
\]
%

\begin{lemma} \label{lem:MPC_stability_p=2}
 Let $p=2$. For any $\tau>0$, $(\alpha^1_i, \beta^1_i, \zeta^1_i) >0$ and  $0\ne (\alpha^2_i, \beta^2_i, \zeta^2_i) \ge 0$ for each $i=1, \ldots, n$, the matrix $A_{\mbox{c}}(\boldsymbol\alpha, \boldsymbol\beta, \boldsymbol\zeta, \tau)$ is Schur stable, i.e., its spectral radius is strictly less than 1.
\end{lemma}

\begin{proof}
 By the previous argument, it suffices to show that each $\wt A_i$ is Schur stable for $i=1, \ldots, n$. Fix an arbitrary $i$. Letting $\wt {\mathbf  K}_i =\begin{bmatrix} c_1 & c_2 \end{bmatrix}$, we have
 \[
    \wt A_i = \begin{bmatrix} 1 + \frac{\tau^2}{2} c_1 & \tau + \frac{\tau^2}{2} c_2 \\ \tau c_1 & 1+\tau c_2 \end{bmatrix},
 \]
 where
 \[
    c_1 = -  \frac{d_2 \alpha^1_i + \alpha^2_i( 2 \beta^2_i + 3 \zeta^2_i)}{2d'}, \qquad c_2 = -\frac{d_2 (\frac{\tau^2}{2} \alpha^1_i+ \beta^1_i) + \frac{3\tau^2}{2} \alpha^2_i \beta^2_i + \zeta^2_i( 3 \tau^2 \alpha^2_i + \beta^2_i)}{\tau d'},
 \]
 and $d':=\det(\wt{\mathbf H}_i)/\tau^4$, and $d_s = \frac{\tau^2}{4} \alpha^s_i + \beta^s_i + \zeta^s_i$ for $s=1, 2$. Hence, $d'=d_1 d_2 + \tau^2 \alpha^2_i(\beta^2_i + \frac{9}{4}\zeta^2_i)+\beta^2_i \zeta^2_i$.
  Define
  \[
    \alpha':=d_2 \alpha^1_i + \alpha^2_i( 2 \beta^2_i + 3 \zeta^2_i),\  \quad \
    \beta':= d_2 \beta^1_i +\frac{\tau^2}{2} \alpha^2_i \beta^2_i + \zeta^2_i\Big( \frac{3}{2} \tau^2 \alpha^2_i + \beta^2_i \Big), \ \quad \
    \gamma':= d_2 \zeta^1_i.
  \]
%
  Clearly, $\alpha', \beta', \gamma'$ are all positive for any $\tau>0$, $(\alpha^1_i, \beta^1_i, \zeta^1_i) >0$ and  $0\ne (\alpha^2_i, \beta^2_i, \zeta^2_i) \ge 0$. Moreover, we deduce
  from a somewhat lengthy but straightforward calculation that $d'=\frac{\tau^2}{4} \alpha' + \beta'+ \gamma'>0$. Hence, $c_1 = - \frac{\alpha'}{2d'}$, $c_2 = -\frac{\tau^2 \alpha'/2 + \beta'}{\tau d'}$, and
 \[
    \wt A_i = \begin{bmatrix} 1 - \frac{\alpha' \tau^2 }{4 d'}    & \tau \Big( 1 -  \big( \frac{  \alpha' \tau^2 }{4} + \frac{\beta'}{2} \big)\frac{1}{d'} \Big) \\
    -\frac{\alpha' \tau }{2d'}  & 1 -  \Big( \frac{\alpha' \tau^2 }{2}   + \beta' \Big) \frac{1}{d'} \end{bmatrix} \in \mathbb R^{2\times 2}.
 \]
 It follows from   \cite[Proposition 5.1]{GShenDu_TRB16} that $\wt A_i$ is Schur stable, and so is $A_{\mbox{c}}(\boldsymbol\alpha, \boldsymbol\beta, \boldsymbol\zeta, \tau)$.
\end{proof}

Using a similar technique but more lengthy calculations, it can be shown that when $p= 3$, the matrix $A(\boldsymbol\alpha, \boldsymbol\beta, \boldsymbol\zeta, \tau)$ is Schur stable for $\tau>0$, $(\alpha^1_i, \beta^1_i, \zeta^1_i) >0$,  $0\ne (\alpha^2_i, \beta^2_i, \zeta^2_i) \ge 0$ and $0\ne (\alpha^3_i, \beta^3_i, \zeta^3_i) \ge 0$ for each $i=1, \ldots, n$. For $p>4$, we expect that the same result holds (supported by numerical experience) although its proof becomes much more complicated. Nevertheless,
it is observed that in the $p$-horizon MPC, when the parameters $\alpha^s_i$, $\beta^s_i$ (and possibly including $\zeta^s_i$) with $s\ge 3$ are medium or large, large control inputs are generated, which causes control or speed saturation and may lead to undesired close-loop dynamics.
Motivated by this observation, we obtain the following stability result for small $(\alpha^s_i, \beta^s_i)\ge 0$ for $s=3, \ldots, p$.

\begin{proposition}
Let $p \ge 3$. For any $\tau>0$, $(\alpha^1_i, \beta^1_i, \zeta^1_i) >0$ and $0\ne (\alpha^2_i, \beta^2_i, \zeta^2_i) \ge 0$ for each $i=1, \ldots, n$, and $\zeta^s_i>0$ for $s=3, \ldots, p$ and $i=1, \ldots, n$, there exists a positive constant $\ol \varepsilon$ such that for any $\alpha^s_i, \beta^s_i \in [0, \ol\varepsilon)$ for $s=3, \ldots, p$ and $i=1, \ldots, n$,
the matrix $A_{\mbox{c}}(\boldsymbol\alpha, \boldsymbol\beta, \boldsymbol\zeta, \tau)$ is Schur stable.
\end{proposition}

\begin{proof}
Consider $p \ge 3$. Fix arbitrary $\tau>0$, $(\alpha^1_i, \beta^1_i, \zeta^1_i) >0$, and $0\ne (\alpha^2_i, \beta^2_i, \zeta^2_i) \ge 0$ for each $i=1, \ldots, n$ and $\zeta^s_i>0$ for $s=3, \ldots, p$ and $i=1, \ldots, n$. Suppose $\alpha^s_i=\beta^s_i=0$ for all $s=3, \ldots, p$ and $i=1, \ldots, n$. Then $Q_{z, s} =Q_{z', s}=0$ for all $s \ge 3$. Hence, $\Hbld_{i, j}=0$ for all $i \ge 3$ and any $j$. Thus it is easy to show that for each $i=1, \ldots, n$,
\[
   \wt{\mathbf H}_i = \begin{bmatrix} \wt{\mathbf H}^2_i & & & \\ & \tau^2 \zeta^3_i &  & \\ & & \ddots & \\ & & & \tau^2 \zeta^p_i \end{bmatrix} \in \mathbb R^{p\times p}, \qquad
   \wt{\mathbf G}_i = \begin{bmatrix} \wt{\mathbf G}^2_i  \\ 0 \\ \vdots \\ 0 \end{bmatrix} \in \mathbb R^{p \times 2},
\]
where $\wt{\mathbf H}^2_i \in \mathbb R^{2\times 2}$ and $\wt{\mathbf G}^2_i \in \mathbb R^{2\times 2}$ correspond to $p=2$ given before. Hence,  $\wt { \mathbf  K}_i := - \mathbf e^T_1 \wt{\mathbf H}^{-1}_i \wt{\mathbf G}_i = - \mathbf e^T_1 (\wt{\mathbf H}^2_i)^{-1} \wt{\mathbf G}^2_i$. It follows from Lemma~\ref{lem:MPC_stability_p=2} that $A_{\mbox{c}}(\boldsymbol\alpha, \boldsymbol\beta, \boldsymbol\zeta, \tau)$ is Schur stable, i.e., its spectral radius is strictly less than 1. Since the spectral radius of $A_{\mbox{c}}(\boldsymbol\alpha, \boldsymbol\beta, \boldsymbol\zeta, \tau)$ is continuous in $\alpha^s_i, \beta^s_i$ for all $s=3, \ldots, p$ and $i=1, \ldots, n$, a small perturbation to $\alpha^s_i, \beta^s_i$ for all $s=3, \ldots, p$ and $i=1, \ldots, n$ still leads to the Schur stable matrix $A_{\mbox{c}}$. This yields the desired result.
\end{proof}

Based on the above results, one may choose $Q_{z, s}, Q_{z', s}, Q_{w, s}$ in the following way.
Let $u_s, v_s  \in \mathbb R^n_+$ and $w_s \in \mathbb R^n_{++}$ be  positive or nonnegative vectors of the same order.
Let $\eta>1$ (e.g., $\eta=5$ or higher) be a constant and let $\kappa_z, \kappa_{z'}$ and $\kappa_w$ be some positive constants. Then let
\[
   Q_{z, s} = \frac{\kappa_z}{(s-1)^\eta} \mbox{diag}(u_s), \quad
   Q_{z', s} = \frac{\kappa_{z'}}{(s-1)^\eta} \mbox{diag}(v_s), \quad
   Q_{w, s} = \frac{\kappa_{w} }{(s-1)^\eta}\mbox{diag}(w_s), \qquad s=2, \ldots, p.
\]

%
\section{Numerical Results} \label{sect:numerical_results}

\subsection{Numerical Experiments and Weight Matrices Design} \label{subsect:Numerical_test_description}

We conduct numerical tests to evaluate the performance of the proposed fully distributed schemes and the CAV platooning control. In these tests, we consider a platoon of an uncontrolled leading vehicle labeled by the index 0 and ten (i.e., $n=10$) CAVs following the leading vehicle. The following physical parameters are used for the CAVs and their constraints throughout this section unless otherwise stated: the desired spacing $\Delta = 50 m$, the vehicle length $L=5 m$, the sample time $\tau = 1 s$, the reaction time $r=\tau=1 s$, the acceleration and deceleration limits $a_{\max}= 1.35m/s^2$ and $a_{\min}=-8 m/s^2$, and the speed limits  $v_{\max}=27.78 m/s$ and $v_{\min}=10 m/s$. The initial state of the platoon is $z(0)=z'(0)=0$ and $v_i(0)=25 m/s$ for all $i=0, 1, \ldots, n$. Further, the vehicle communication network is given by the cyclic-like graph, i.e., the  bidirectional edges of the graph are $(1, 2), (2, 3), \ldots, (n-1, n) \in \mathcal E$.

When $n=10$, a particular choice of these weight matrices is given as follows: for $p=1$,
\begin{eqnarray*}
\alphabf^1 & = & \big( 38.85, 40.2, 41.55,   42.90,   44.25,   45.60,  46.95, 48.30,   49.65, 51.00 \big):= \wt \alphabf, \\
\betabf^1 & = & \big( 130.61,   136.21,   141.82,   147.42,   153.03,   158.64,   164.24,   169.85,  175.46,  181.06 \big) :=\wt \betabf, \\
\zetabf^1 & = & \big( 62,    74,    90,    92,   106,   194,   298,   402,   454,   480 \big):=\wt \zetabf.
\end{eqnarray*}
For $p \ge 2$, we choose $\alphabf^1 = \wt\alphabf - \mathbf 1$, $\betabf^1 = \wt \betabf - \mathbf 1$, $\zetabf^1 = \wt\zetabf - \mathbf 1$, and
\[
   \alphabf^s = \frac{0.0228}{(s-1)^4} \times \wt \alphabf, \quad
   \betabf^s = \frac{0.044}{(s-1)^4} \times \wt \betabf, \quad
   \zetabf^s = \frac{0.0026}{(s-1)^4} \times \wt \zetabf, \quad s=2, \ldots, p.
\]
The above vectors $\alphabf^s, \betabf^s, \zetabf^s$ define the weight matrices $Q_{z, s}, Q_{z', s}, Q_{w, s}$ for $s=1, \ldots, 5$, which further yield the closed loop dynamics matrix $A_{\mbox{c}}$; see the discussions below (\ref{eqn:K_mat_d_vec}).
It is shown that when these weights are used, the spectral radius of $A_{\mbox{c}}$ is $0.8498$ for $p=1$, and $0.8376$ for $p=2, \ldots, 5$, respectively.

\gap

\noindent{\bf Discussion on the selection of MPC horizon}. We discuss the choice of the MPC prediction horizon $p$ based on numerical tests as follows. Our numerical experience shows that for $p>1$, the weight matrices $Q_{z, 1}, Q_{z', 1}$ and $Q_{w, 1}$ play a more important role for the closed loop dynamics. For fixed $Q_{z, 1}, Q_{z', 1}$ and $Q_{w, 1}$ with the large penalties in $Q_{z, s}, Q_{z', s}$ and $Q_{w, s}$ for $s>1$, the closed loop dynamics may be mildly improved but at the expense of undesired large control. Hence, we choose smaller penalties in $Q_{z, s}, Q_{z', s}$ and $Q_{w, s}$ for $s>1$, which only lead to slightly better closed loop performance compared with the case of $p=1$. Further, when a large $p$ is used, the underlying optimization problem has a larger size, resulting in longer computation time and slow convergence of the proposed distributed scheme. Besides, the current MPC model assumes that the future $u_0(k+s)=u_0(k)$ for all $s=1, \ldots, p-1$ at each $k$. This assumption is invalid when the true $u_0(k+s)$ is substantially different from $u_0(k)$, which implies that the prediction performance is poor for a large $p$.
For these reasons, it is recommended that a smaller $p$ be used, for example, $p \le 5$.

\gap

The following scenarios  are used to evaluate the proposed CAV platooning control.

\noindent $\bullet$ {\bf Scenario 1}: The leading vehicle performs instantaneous deceleration/acceleration and then keeps a constant speed for a while. The goal of this scenario is to test if the platoon can maintain stable spacing and speed when the leading vehicle is subject to acceleration or deceleration disturbances. The motion profile of the leading vehicle is as follows: the leading vehicle decelerates from $k=51s$ to $k=54s$ with the deceleration $-2m/s^2$, and maintains a constant speed till $k=100s$. After $k=100s$, it restores to its original speed $25m/s$ with the acceleration $1m/s^2$.

\gap

\noindent $\bullet$ {\bf Scenario 2}: The leading vehicle performs periodical acceleration/deceleration. The goal of this scenario is to test whether the proposed control scheme can reduce periodical spacing and speed fluctuation. The motion profile of the leading vehicle in this scenario is as follows: the leading vehicle periodically changes its acceleration and deceleration from $k=51s$ to $k=100s$ with the period $T=4s$ and acceleration/deceleration $\pm 1m/s^2$. Then  it maintains its original constant speed $25m/s$ after $k=100 s$.

\gap

\noindent $\bullet$ {\bf Scenario 3}:  In this scenario, we aim to test the performance of the proposed control scheme in a real traffic environment, particularly when the leading vehicle undergoes traffic oscillations. We use real world trajectory data from an oscillating traffic flow to generate the leading vehicle's motion profile. Specifically, we consider NGSIM data on eastbound I-80 in San Francisco Bay area in California.
%
%
We use the data of position and speed of a real vehicle to generate its control input at each second and treat this vehicle as a leading vehicle. Since the maximum of acceleration of this vehicle is close to $2 m/s^2$, we choose
$a_{\max}= 2 m/s^2$. All the other parameters or physical limits remain the same.
The experiment setup of this scenario is:  $z_i(0) = 0 m$, $v_i(0) = 25 m/s$ for each $i$, and the time length is $45 s$.
To further test the proposed CAV platooning control in a more realistic traffic setting in Scenario 3, random noise is added to each CAV to simulate dynamical disturbances, model mismatch,  signal noise, communication delay, and road condition perturbations. In particular, at each $k$, the random noise with the normal distribution $\mathcal N(0.04, 0)$ is added to the first CAV, and the noise with the normal distribution $\mathcal N(0.02, 0)$ is added to each of the rest of the CAVs. Here a larger noise is imposed to the first CAV since there are more noises and disturbances between the leading vehicle and the first CAV.

%
\subsection{Performance of Fully Distributed Schemes and CAV Platooning Control} \label{subsect:test_performance}

The generalized Douglas-Rachford splitting method based distributed algorithm (i.e., Algorithm~\ref{algo:DR_distributed_splitting}) is tested. For each MPC horizon $p$,  the parameters $\alpha$, $\rho$, and the error tolerance for the stopping criteria in this algorithm are chosen  to achieve desired numerical accuracy and efficiency; see the discussions below (\ref{eqn:DR_scheme}) for error tolerances and Table~\ref{Table0} for a list of these parameters and error tolerances.
In particular, we choose a larger error tolerance for a larger $p$ to meet the desired computation time requirement of one second per vehicle.
For comparison, we also test the three operator splitting based distributed scheme and its accelerated version given in Remark~\ref{remark:other_schemes}, where we
 choose $\delta_i = \lambda_{\min}(\grave W^i)/2$, $\gamma=1.9/{\wh L}$ and $\lambda =1.05$. Here $\wh L$ is the Lipschitz constant defined in Remark~\ref{remark:other_schemes}. For
the accelerated scheme, we let $\eta = 0.2$ and $\gamma_0 = 1.9/(0.8\times\wh L)$.

\begin{table}
\caption{Parameters in Algorithm~\ref{algo:DR_distributed_splitting} for different MPC horizon $p$'s}
\newcommand{\tabincell}[2]{\begin{tabular}{@{}#1@{}}#2\end{tabular}}
\begin{center}
\begin{tabular}{|c|c|c|c|c|c|}
\hline
%
%
MPC horizon & \ \ $p=1$ \ \ & $p=2$ & $p=3$ & $p=4$ & $p=5$ \\
\hline

$\alpha$ &  0.95   &  0.95   &  0.95  &  0.8   &  0.8   \\
\hline
$\rho$ & 0.3 & 0.3 & 0.3 & 0.1 & 0.1 \\
\hline
Error tolerance & \ $10^{-3}$ \ & $2\times 10^{-3}$ & $5\times 10^{-3}$ & $7\times 10^{-3}$ & $1.25\times 10^{-2}$ \\
\hline
\end{tabular}
\end{center}
\label{Table0}
\end{table}

\gap

\noindent {\bf Initial guess warm-up}. \ For a given $p$, the augmented locally coupled optimization problem (\ref{eqn:opt_lin_slitting}) has nearly $3np$ scalar variables and $3np$ scalar constraints when the cyclic-like network topology is considered. These sizes can be even larger for other network topologies satisfying the assumption {\bf A1}.
 Hence, when $p$ is large, the underlying optimization problem is of large size, which may affect the numerical performance of the distributed schemes.
Several techniques are developed to improve the efficiency of the proposed Douglas-Rachford distributed schemes for real-time computation, particularly for a large $p$. For illustration, we discuss the initial guess warm-up technique as follows.
When implementing the proposed scheme, we often choose a numerical solution obtained from the last step as an initial guess for the current step and run the proposed Douglas-Rachford  scheme. Such the choice of an initial guess usually works well when two neighboring control solutions are relatively close. However,  it is observed that the convergence of the proposed distributed scheme is sensitive to an initial guess, especially when the CAV platoon is subject to significant traffic oscillations, which results in highly different control solutions between two neighboring instants. In this case, using a neighboring control solution as an initial guess leads to rather slow convergence.
To solve this problem, we propose an initial guess warm-up technique, motivated by the fact that control solutions are usually unconstrained for most of $k$'s. Specifically,
we first compute an unconstrained solution in a fully distributed manner, which can be realized by setting $\mathcal P_i$ as the Euclidean space in Algorithm~\ref{algo:DR_distributed_splitting}.
This step can be efficiently computed since the proximal operator is formulated by an unconstrained quadratic program and has a closed form solution. In fact, letting $J_i(\wh{\mathbf u}_i) = \frac{1}{2} \wh{\mathbf u}^T_i \wh W_i \wh{\mathbf u}_i + c^T_{\Ical_i} \wh{\mathbf u}_i$, the closed form solution to the proximal operator is given by $\mbox{Prox}_{\rho J_i}(\wh{\mathbf u}'_i) = - (\rho \wh W_i + I)^{-1} (\rho c_{\Ical_i} - \wh{\mathbf u}'_i)$, where $\wh W_i$ is PD.
We then project this unconstrained solution onto the constrained set in one step. Due to the decoupled structure of the problem (\ref{eqn:opt_lin_slitting}), this one-step projection can be computed in a fully distributed manner. We thus use this projected solution as an initial guess for the Douglas-Rachford scheme. Numerical experience shows that this new initial guess significantly improves computation time and solution quality when $p$ is large. 

\begin{table}
\caption{Scenario 1: computation time and numerical accuracy}
\newcommand{\tabincell}[2]{\begin{tabular}{@{}#1@{}}#2\end{tabular}}
\begin{center}
\begin{tabular}{|c|c|c|c|c|}
\hline
\multirow{2}{*}{MPC horizon} &
\multicolumn{2}{c|}{Computation time per CAV (s)} &
\multicolumn{2}{c|}{Relative numer. error}\\
\cline{2-5}
   & \ \ Mean \ \ & Variance & \ \ Mean \ \ & \ Variance \ \\
\hline
$p=1$ &   0.0248   & 0.0017 & \ $3.4\times 10^{-4}$ \ & \ $1.9\times 10^{-7}$ \ \\
\hline
$p=2$ & 0.0603 & 0.0034 & $1.5\times 10^{-3}$ & $2.6\times 10^{-6}$\\
\hline
$p=3$ & 0.1596 & 0.0764 & $3.2\times 10^{-3}$ & $1.1\times 10^{-5}$ \\
\hline
$p=4$ & 0.1528 & 0.1500  & $4.0\times 10^{-3}$ & $1.7\times 10^{-5}$ \\
\hline
$p=5$ & 0.2365 &  0.2830 & $6.6\times 10^{-3}$ & $5.7\times 10^{-5}$ \\
\hline
%
\end{tabular}
\end{center}
\label{Table1}
\end{table}


\begin{table}
\caption{Scenario 2: computation time and numerical accuracy}
\newcommand{\tabincell}[2]{\begin{tabular}{@{}#1@{}}#2\end{tabular}}
\begin{center}
\begin{tabular}{|c|c|c|c|c|}
\hline
\multirow{2}{*}{MPC horizon} &
\multicolumn{2}{c|}{Computation time per CAV (s)} &
\multicolumn{2}{c|}{Relative numer. error}\\
\cline{2-5}
   & \ \ Mean \ \ & Variance & \ \ Mean \ \ & Variance \\
\hline
$p=1$ & \ \ 0.0464 \ \ & 0.0039 & \ $4.0\times 10^{-4}$ \ & \ $1.9\times 10^{-7}$ \ \\
\hline
$p=2$ & 0.1086 & 0.0153 & $1.1 \times 10^{-3}$ & $1.4\times 10^{-6}$\\
\hline
$p=3$ & 0.3296 & $0.2593$ & $3.2 \times 10^{-3}$ & $1.13\times 10^{-5}$ \\
\hline
$p=4$ & 0.5049 & 0.6257 & $5.9 \times 10^{-3}$ & $4.6\times 10^{-5}$ \\
\hline
$p=5$ & 0.5784 & 0.7981 & $1.13 \times 10^{-2}$ & $1.3\times 10^{-5}$ \\
\hline
\end{tabular}
\end{center}
\label{Table2}
\end{table}

\begin{table}
\caption{Scenario 3: computation time and numerical accuracy}
\newcommand{\tabincell}[2]{\begin{tabular}{@{}#1@{}}#2\end{tabular}}
\begin{center}
\begin{tabular}{|c|c|c|c|c|}
\hline
\multirow{2}{*}{MPC horizon} &
\multicolumn{2}{c|}{Computation time per CAV (s)} &
\multicolumn{2}{c|}{Relative numer. error}\\
\cline{2-5}
   & \ \ Mean \ \  & Variance & \ \ Mean \ \ & Variance \\
\hline
$p=1$ &  0.0825  & 0.0023 & \ $1.30\times 10^{-3}$ \ & \ $3.5\times 10^{-6}$ \ \\
\hline
$p=2$ & 0.2011 & 0.0051 & $7.5 \times 10^{-3}$ & $1.6\times 10^{-4}$\\
\hline
$p=3$ & 0.5830 & 0.3462 & $1.20 \times 10^{-2}$ & $4.2\times 10^{-4}$ \\
\hline
$p=4$ & 0.8904 & 0.4685 & $1.69 \times 10^{-2}$ & $3.3\times 10^{-4}$ \\
\hline
$p=5$ & 0.9967 & 0.7467 & $3.25 \times 10^{-2}$ & $1.3\times 10^{-4}$ \\
\hline
\end{tabular}
\end{center}
\label{Table3}
\end{table}

\begin{table}
\caption{Scenario 3: computation time and numerical accuracy with initial guess warm-up}
\newcommand{\tabincell}[2]{\begin{tabular}{@{}#1@{}}#2\end{tabular}}
\begin{center}
\begin{tabular}{|c|c|c|c|c|}
\hline
\multirow{2}{*}{MPC horizon} &
\multicolumn{2}{c|}{Computation time per CAV (s)} &
\multicolumn{2}{c|}{Relative numer. error}\\
\cline{2-5}
   &  Mean & Variance & \ \ Mean \ \ & Variance \\
\hline
$p=1$ & \ \ 0.0243 \ \ & 0.0023 & \ $5.0\times 10^{-4}$ \ & \ $7.0\times 10^{-7}$ \ \\
\hline
$p=2$ & 0.0097 & $0.0017$ & $2.6 \times 10^{-3}$ & $1.6\times 10^{-5}$\\
\hline
$p=3$ & 0.0579 & $0.0253$ & $2.2 \times 10^{-3}$ & $1.1\times 10^{-5}$ \\
\hline
$p=4$ & 0.1063 & 0.1103 & $3.7 \times 10^{-3}$ & $2.4\times 10^{-5}$ \\
\hline
$p=5$ & 0.1258 & 0.1155 & $8.5 \times 10^{-3}$ & $1.5\times 10^{-5}$ \\
\hline
\end{tabular}
\end{center}
\label{Table4}
\end{table}

\gap

\noindent{\bf Performance of distributed schemes}.
Distributed algorithms are implemented on MATLAB and run on a computer of the following processor with 4 cores: Intel(R) Core(TM) i7-8550U CPU @ $1.80 GHz$ and RAM: $16.0 GB$.
We test the fully distributed  Algorithm~\ref{algo:DR_distributed_splitting} for Scenarios 1-3.
At each $k \in \mathbb N$, we use the optimal solution obtained from the last step as an initial guess unless otherwise stated.
To evaluate the numerical accuracy of the proposed distributed scheme, we compute the relative error between the numerical solution from the distributed scheme and that from a high precision centralized scheme when the latter solution, labeled as the true solution, is nonzero.
The mean and variance of computation time per vehicle  and relative errors for different MPC horizon $p$'s in noise-free Scenarios 1-3  are displayed in Table~\ref{Table1}-\ref{Table3}, respectively. The numerical performance for Scenario 3 under noises is similar to that without noise; it is thus omitted due to space limit.

It is observed from the numerical results that when the MPC horizon $p$ increases, more computation time is needed with mildly deteriorating numerical accuracy. This observation agrees with the discussion on the choice of $p$ given in Section~\ref{subsect:Numerical_test_description}, which suggests a relatively small $p$ for practical computation.
Besides, we have tested the proposed initial guess warm-up technique on Scenario 3 for different $p$'s using the same parameters and error tolerances for Algorithm~\ref{algo:DR_distributed_splitting}; see Table~\ref{Table0}. To compute a warm-up initial guess using an iterative distributed scheme, we use the same $\alpha$ and $\rho$ for each $p$ with error tolerance $5\times 10^{-4}$ for $p=1$ and $10^{-3}$ for the other $p$'s.
 A summary of  the numerical results is shown in Table~\ref{Table4}. Compared with the results given in Table~\ref{Table3} without initial guess warm-up, the averaging computation time is reduced by at least $80\%$ and the relative numerical error is reduced by at least two thirds for $p \ge 2$ when the initial guess warm-up is used. This shows that the initial guess warm-up technique considerably improves the numerical efficiency and accuracy, and it is especially suitable for real-time computation when a large $p$ is used. Hence, we conclude that Algorithm~\ref{algo:DR_distributed_splitting}, together with the initial guess warm-up technique, is suitable for real-time computation with satisfactory numerical precision.

We have also tested the three-operator splitting based distributed scheme and its accelerated version given in Remark~\ref{remark:other_schemes}. These schemes provide satisfactory computation time and numerical accuracy when $p$ is small. For example, when $p=1$, the mean of computation time per CAV is 0.0553 seconds with the variance 0.0284 for Scenario 1 and 0.219 seconds with the variance 0.138 for Scenario 2, respectively.
 However, for a slightly large $p$, e.g., $p\ge 3$, it takes much longer than 1 second for an individual CAV to complete computation. These schemes are thus not suitable for real-time computation for $p \ge 3$.


\begin{figure}[htbp]
\centering
\subfigure[Time history of spacing changes.]{ \label{S1_p1_spacing}
\includegraphics[width=0.48\columnwidth]{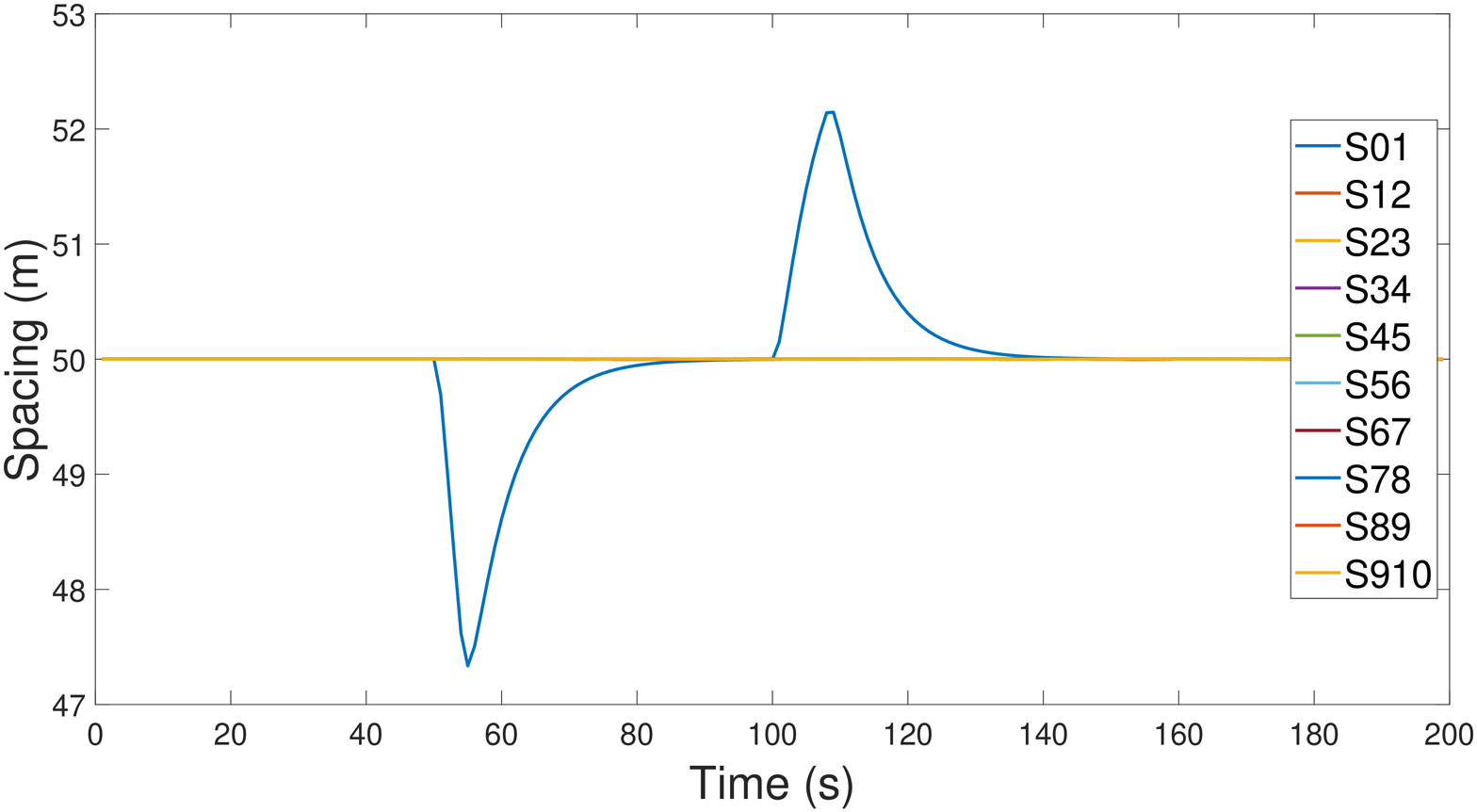}
}
\subfigure[Time history of spacing changes.] { \label{S1_p5_spacing}
\includegraphics[width=0.48\columnwidth]{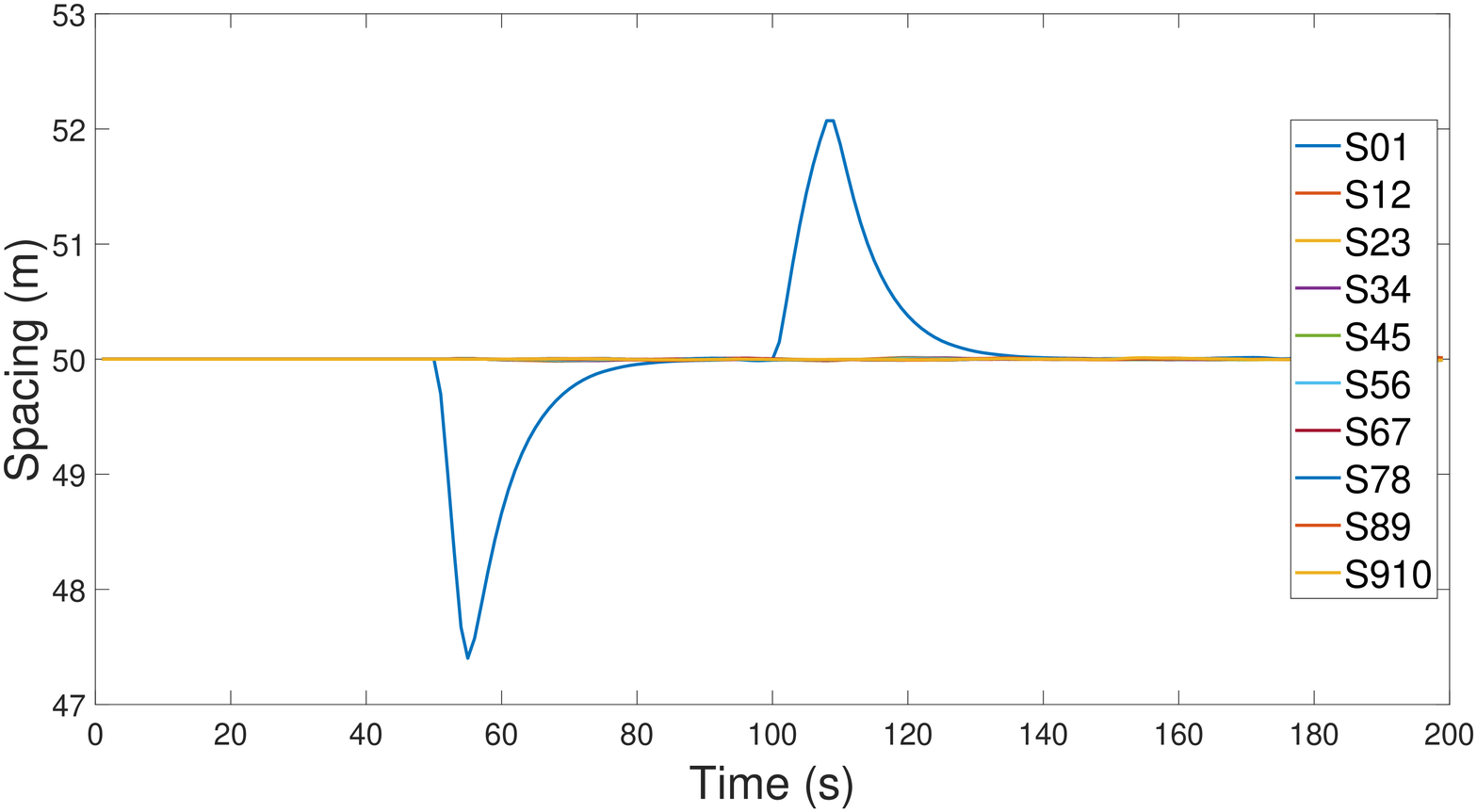}
}
\subfigure[Time history of vehicle speed. ]{ \label{S1_p1_speed}
\includegraphics[width=0.48\columnwidth]{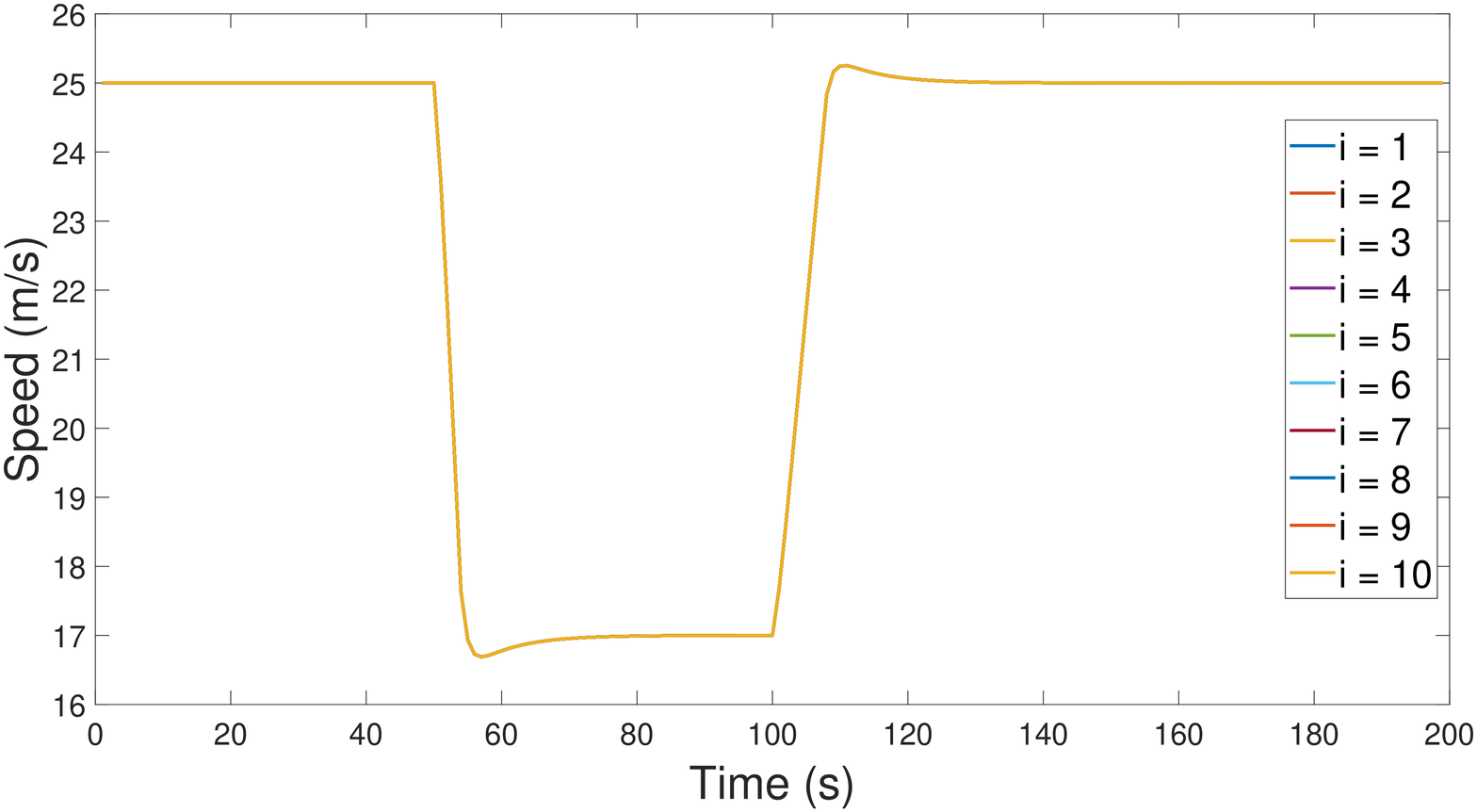}
}
\subfigure[Time history of vehicle speed.] { \label{E11}
\includegraphics[width=0.48\columnwidth]{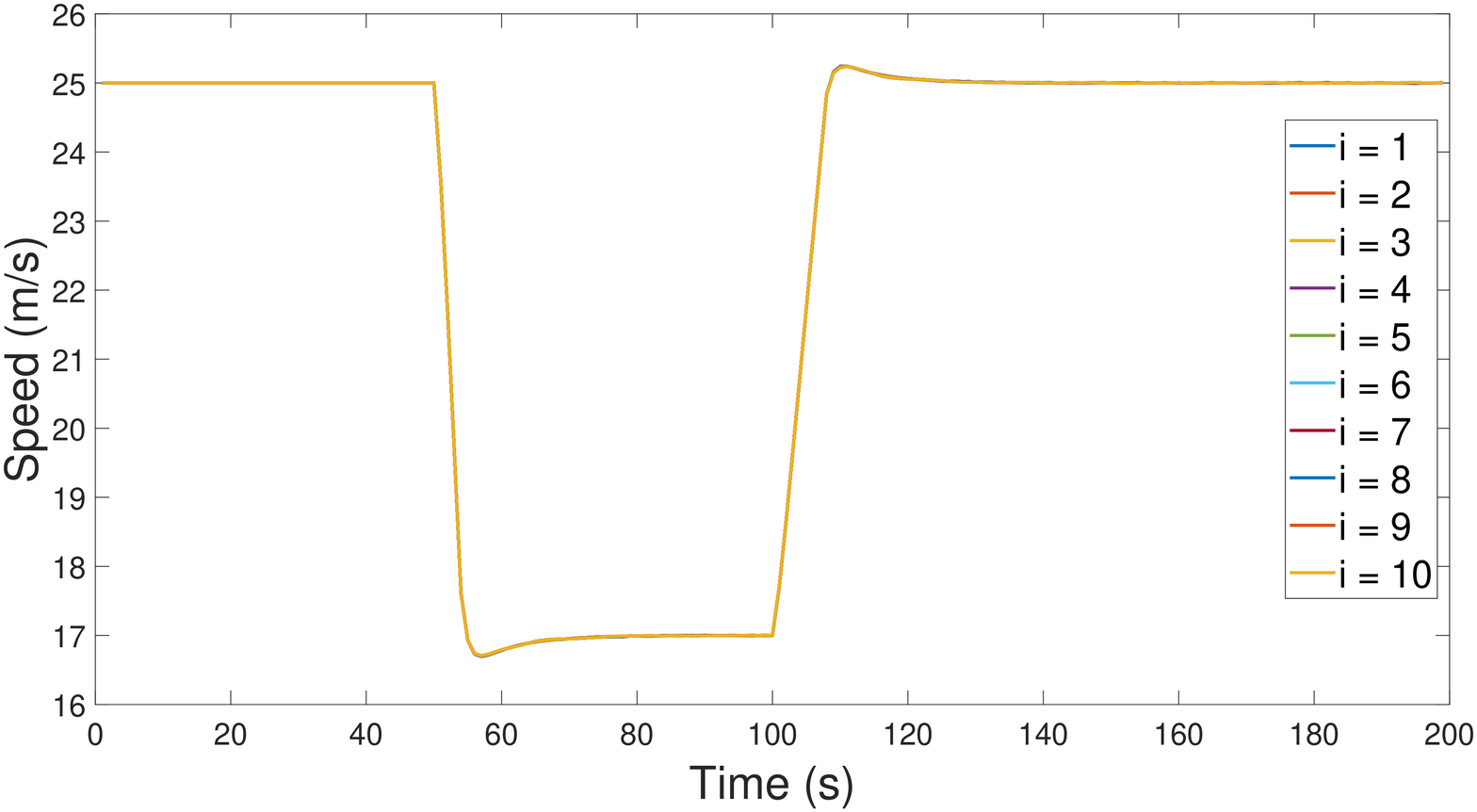}
}
\subfigure[Time history of control input.]{ \label{E8AD}
\includegraphics[width=0.48\columnwidth]{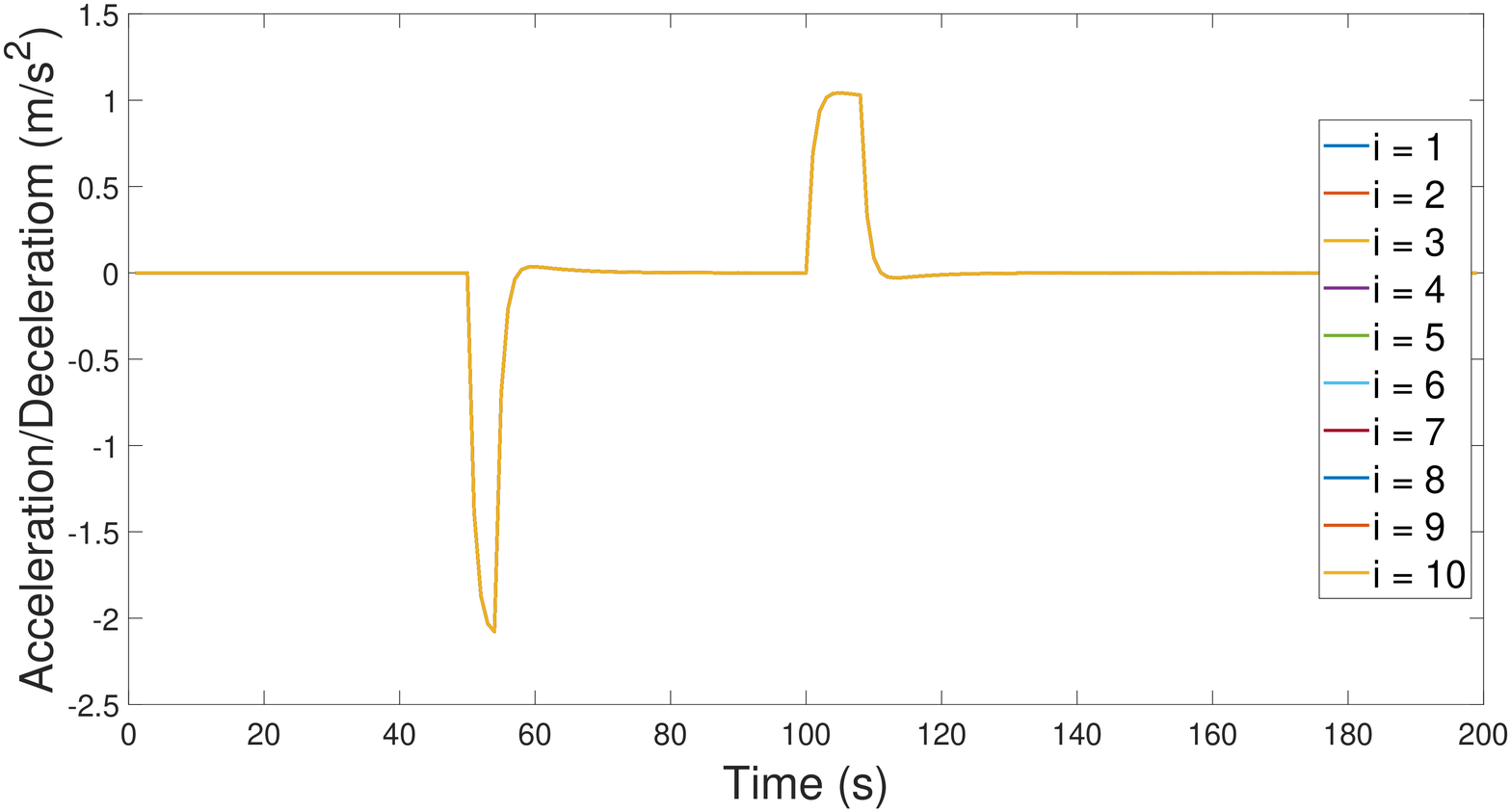}
}
\subfigure[Time history of control input] { \label{E11AD}
\includegraphics[width=0.48\columnwidth]{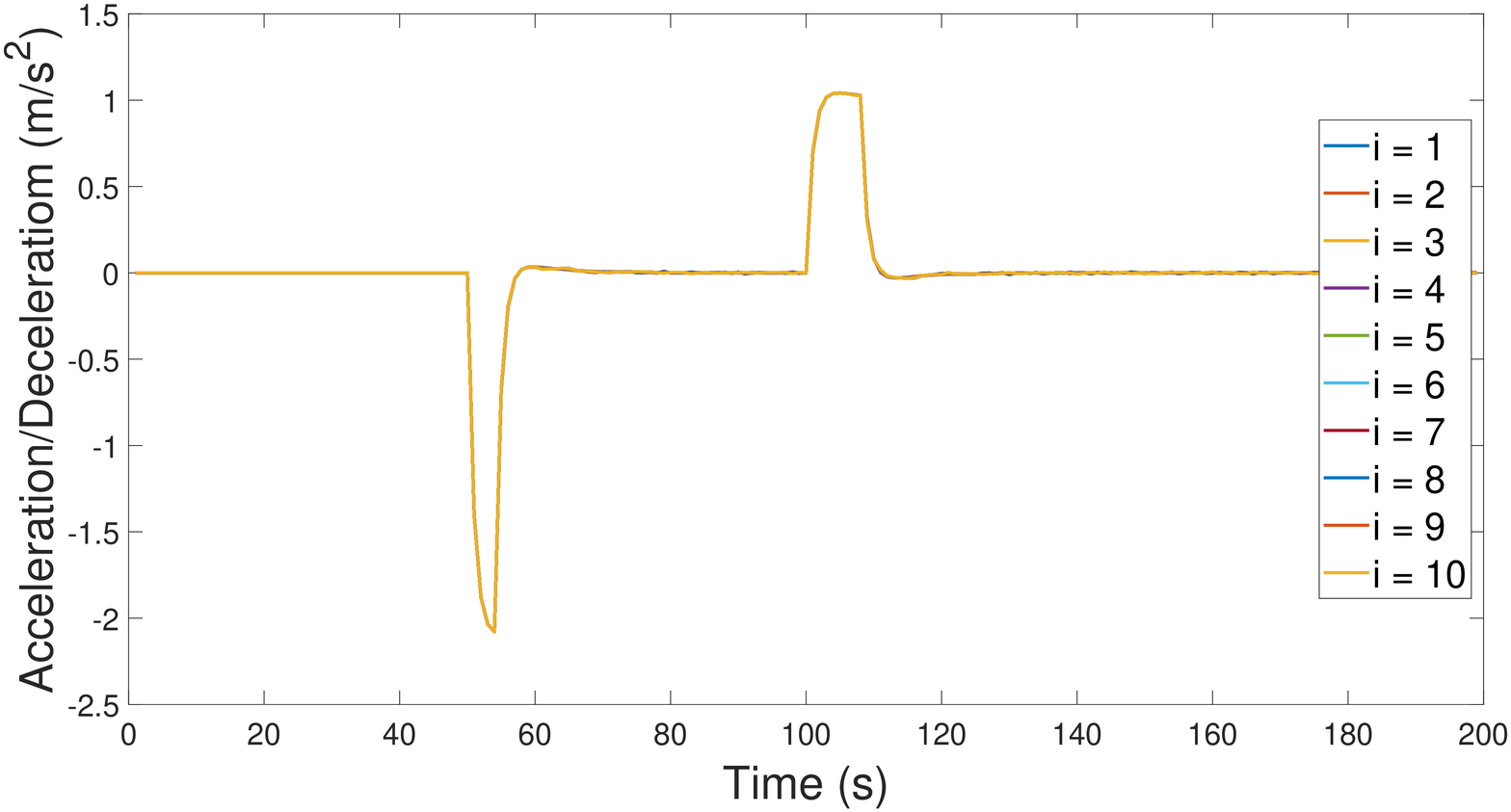}
}
%
\caption{Scenario 1: the proposed CAV platooning control with $p=1$ (left column) and $p=5$ (right column).}
\label{Fig:S1}
\end{figure}


\begin{figure}[htbp]
\centering
\subfigure[Time history of spacing changes.]{ 
\includegraphics[width=0.48\columnwidth]{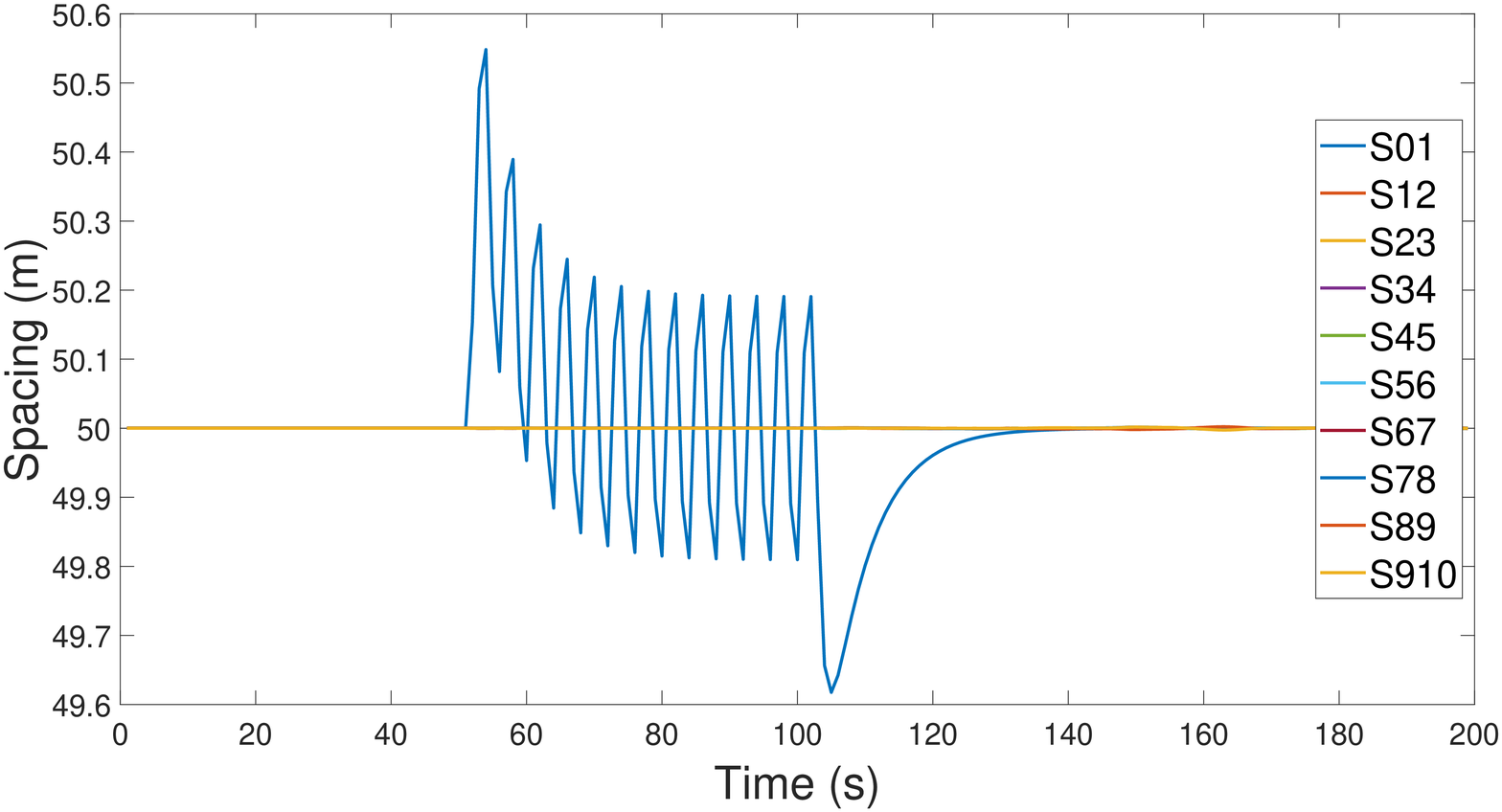}
}
\subfigure[Time history of spacing changes.] { 
\includegraphics[width=0.48\columnwidth]{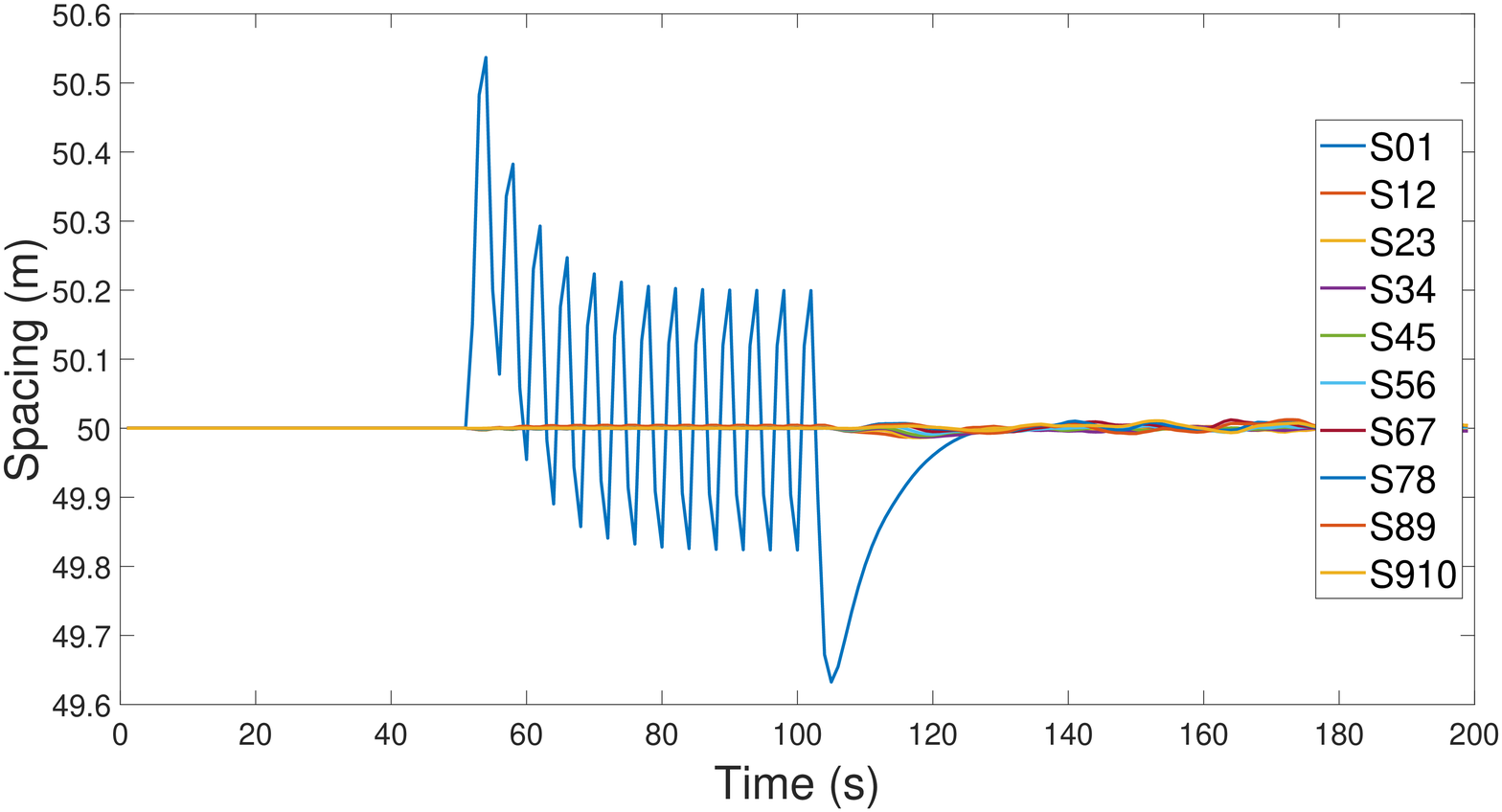}
}
\subfigure[Time history of vehicle speed. ]{ 
\includegraphics[width=0.48\columnwidth]{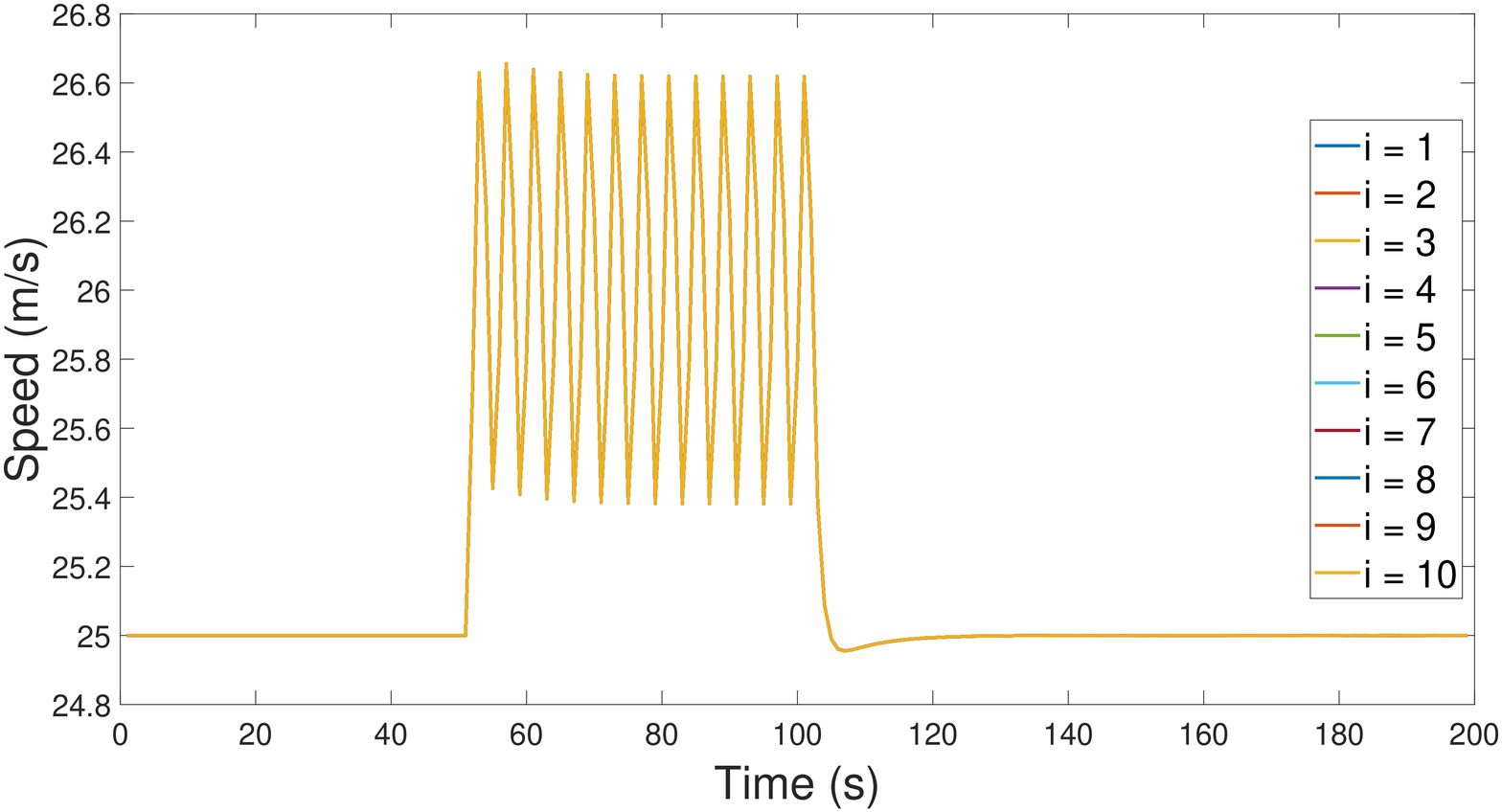}
}
\subfigure[Time history of vehicle speed.] { 
\includegraphics[width=0.48\columnwidth]{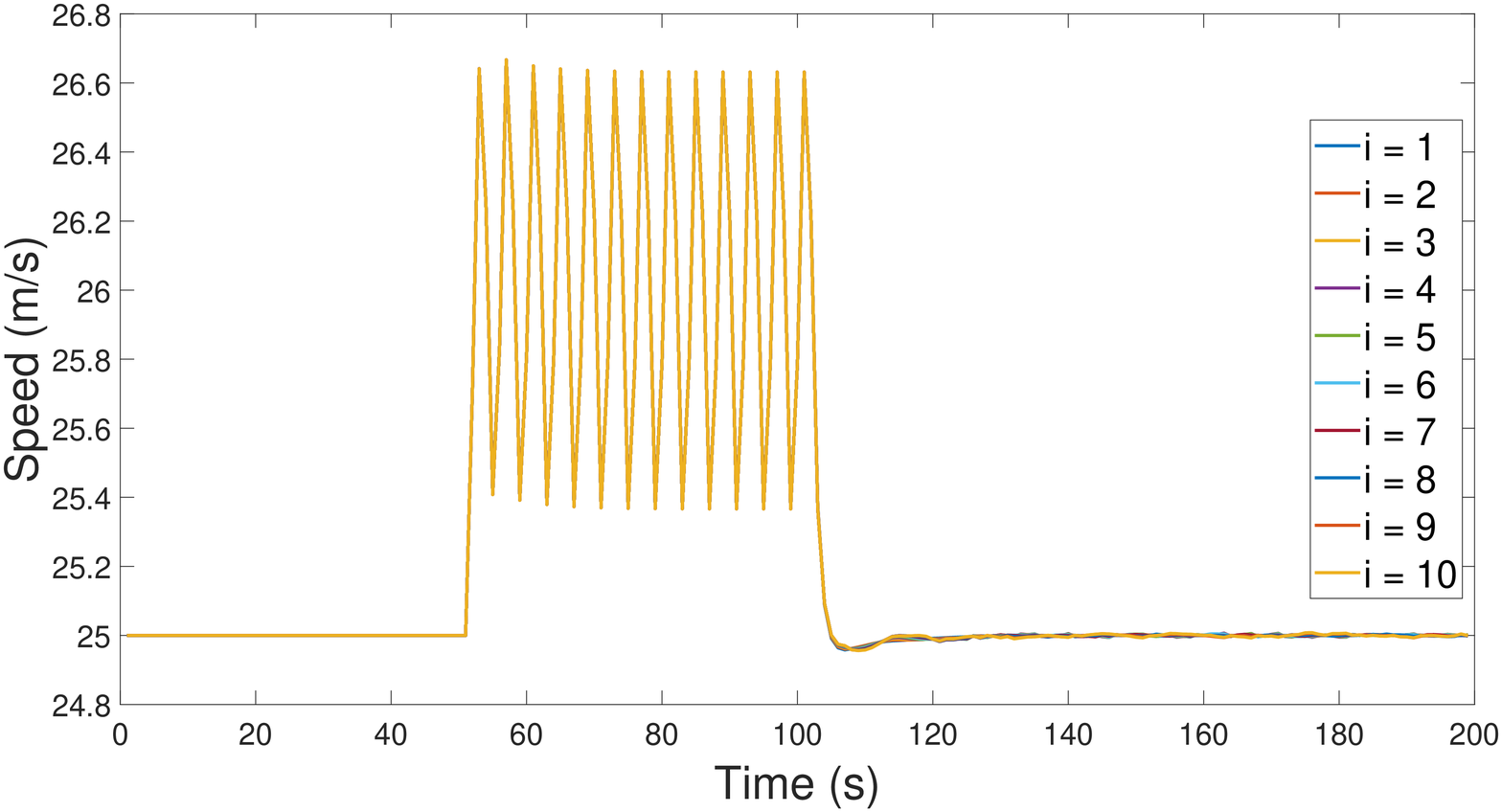}
}
\subfigure[Time history of control input.]{ 
\includegraphics[width=0.48\columnwidth]{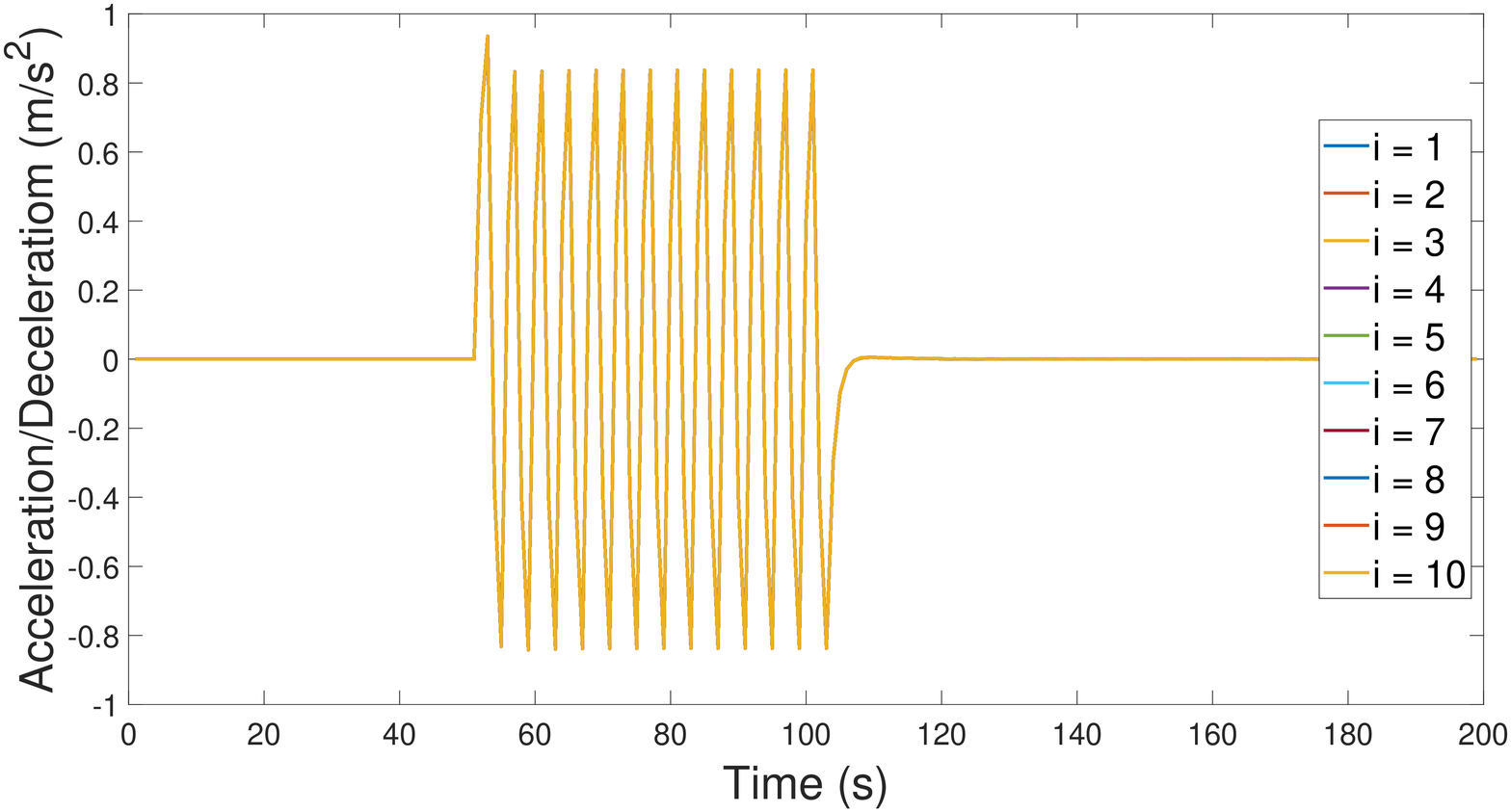}
}
\subfigure[Time history of control input] { 
\includegraphics[width=0.48\columnwidth]{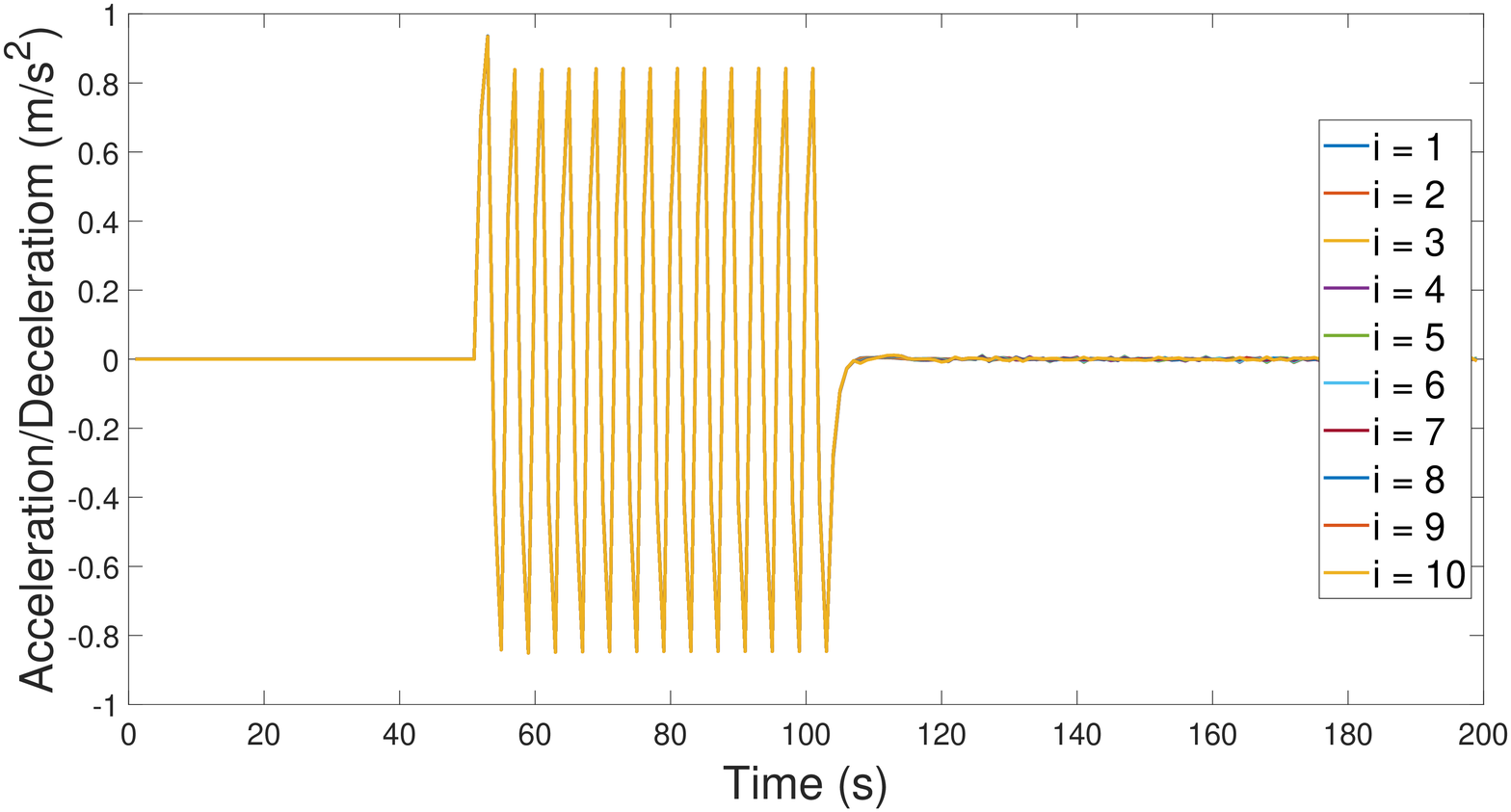}
}
%
\caption{Scenario 2: the proposed CAV platooning control with $p=1$ (left column) and $p=5$ (right column).}
\label{Fig:S2}
\end{figure}


\begin{figure}[htbp]
\centering
\subfigure[Time history of spacing changes.]{ 
\includegraphics[width=0.48\columnwidth]{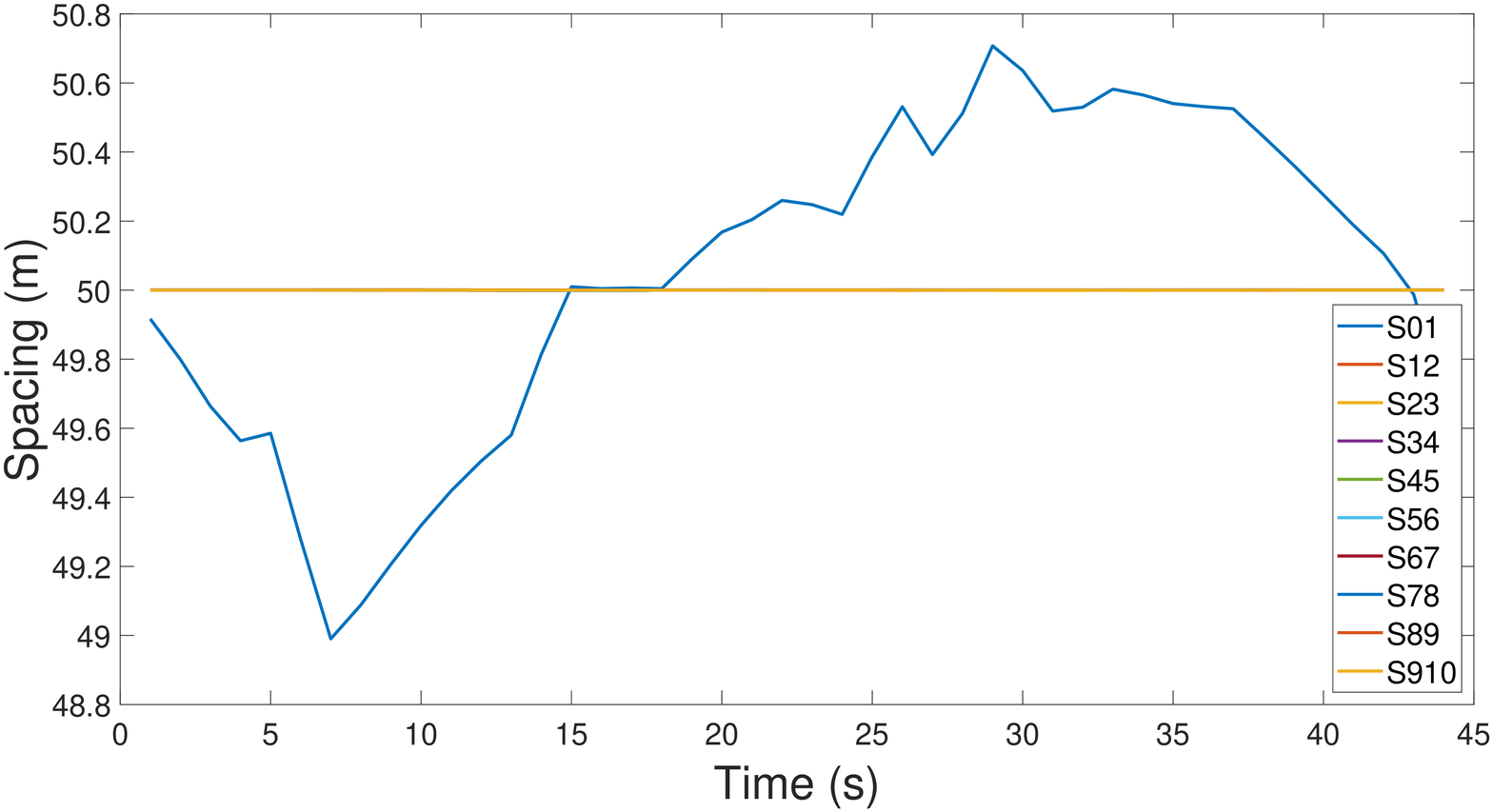}
}
\subfigure[Time history of spacing changes.] { 
\includegraphics[width=0.48\columnwidth]{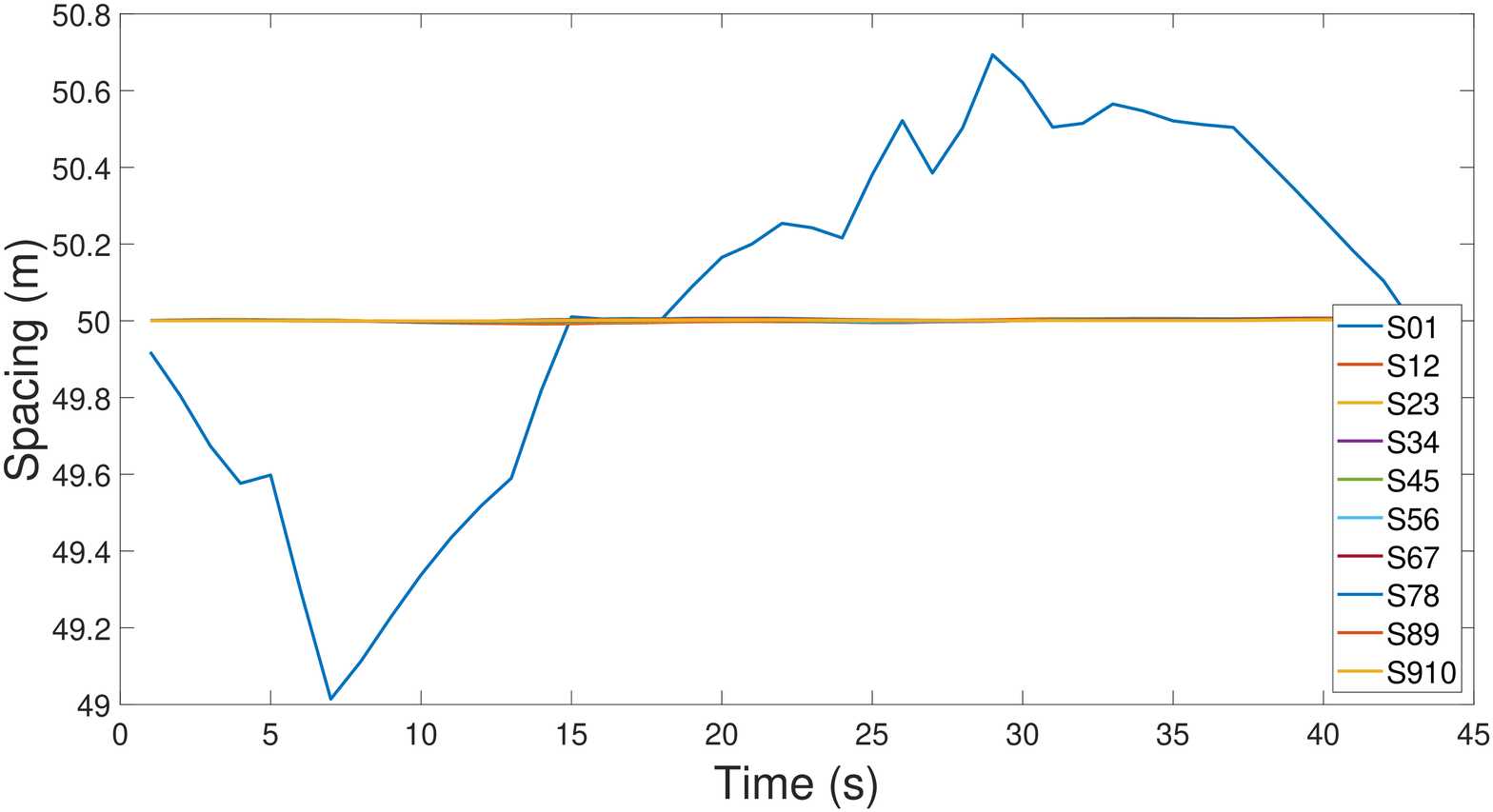}
}
\subfigure[Time history of vehicle speed. ]{ 
\includegraphics[width=0.48\columnwidth]{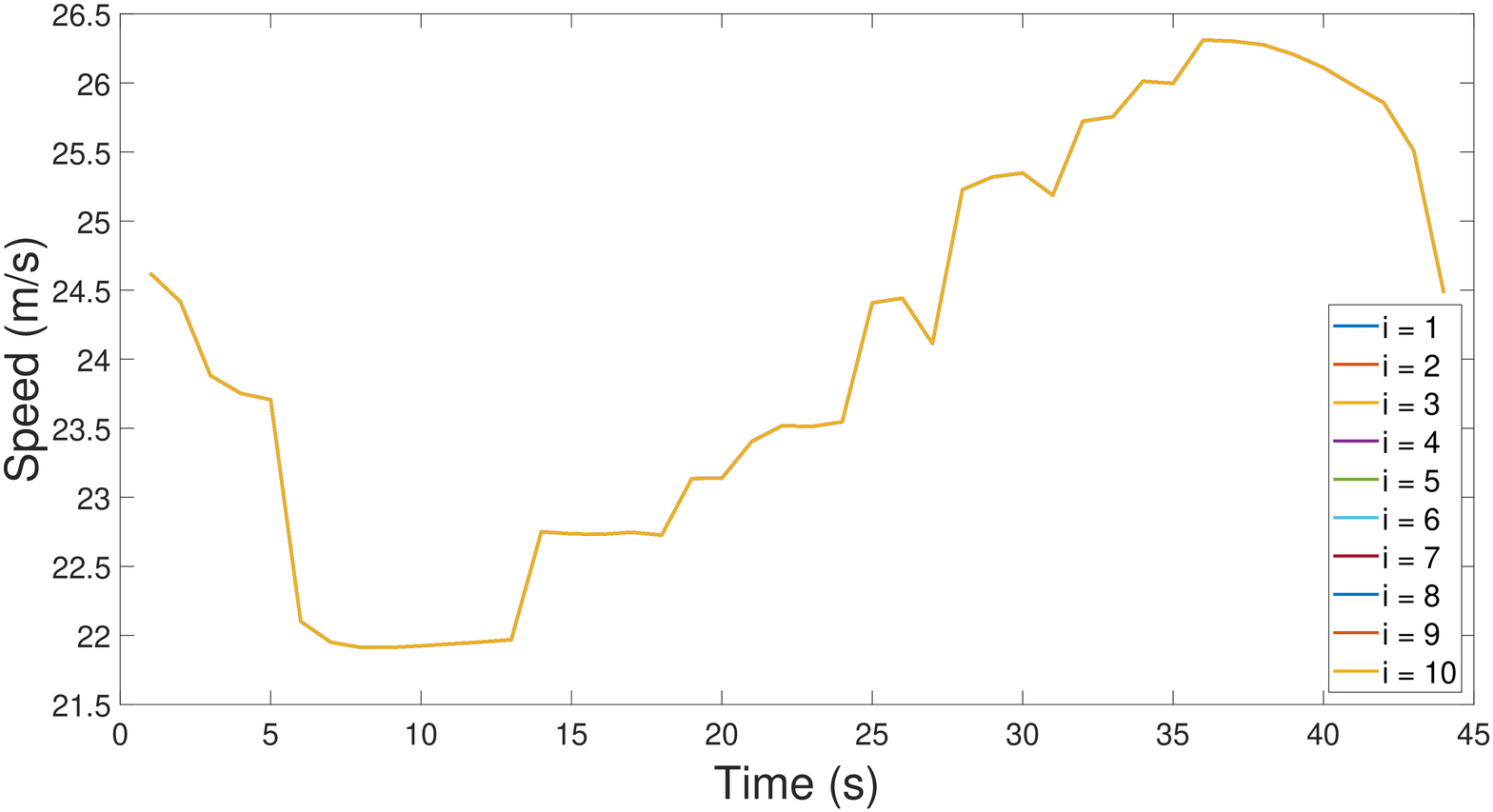}
}
\subfigure[Time history of vehicle speed.] { 
\includegraphics[width=0.48\columnwidth]{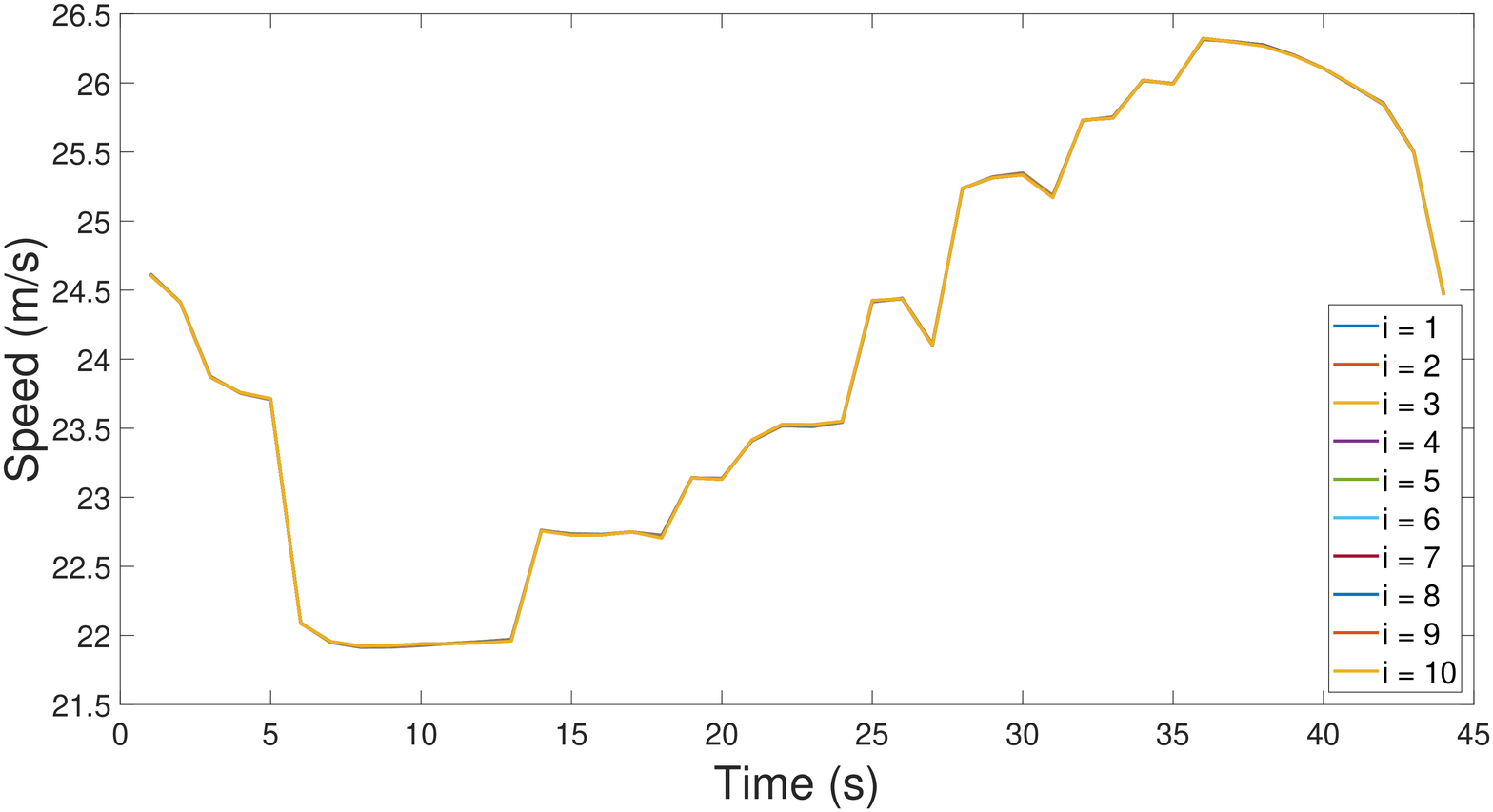}
}
\subfigure[Time history of control input.]{ 
\includegraphics[width=0.48\columnwidth]{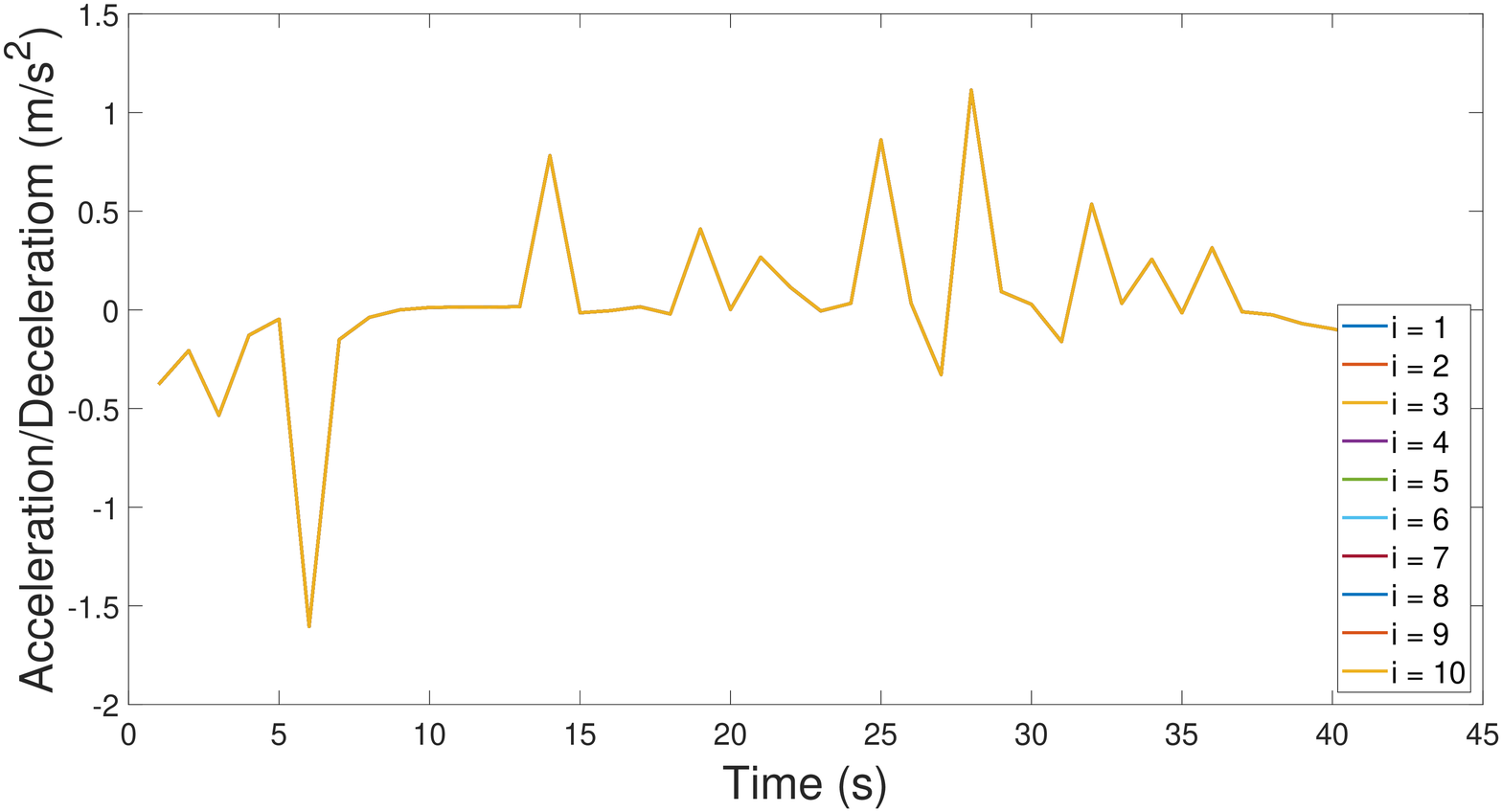}
}
\subfigure[Time history of control input] { 
\includegraphics[width=0.48\columnwidth]{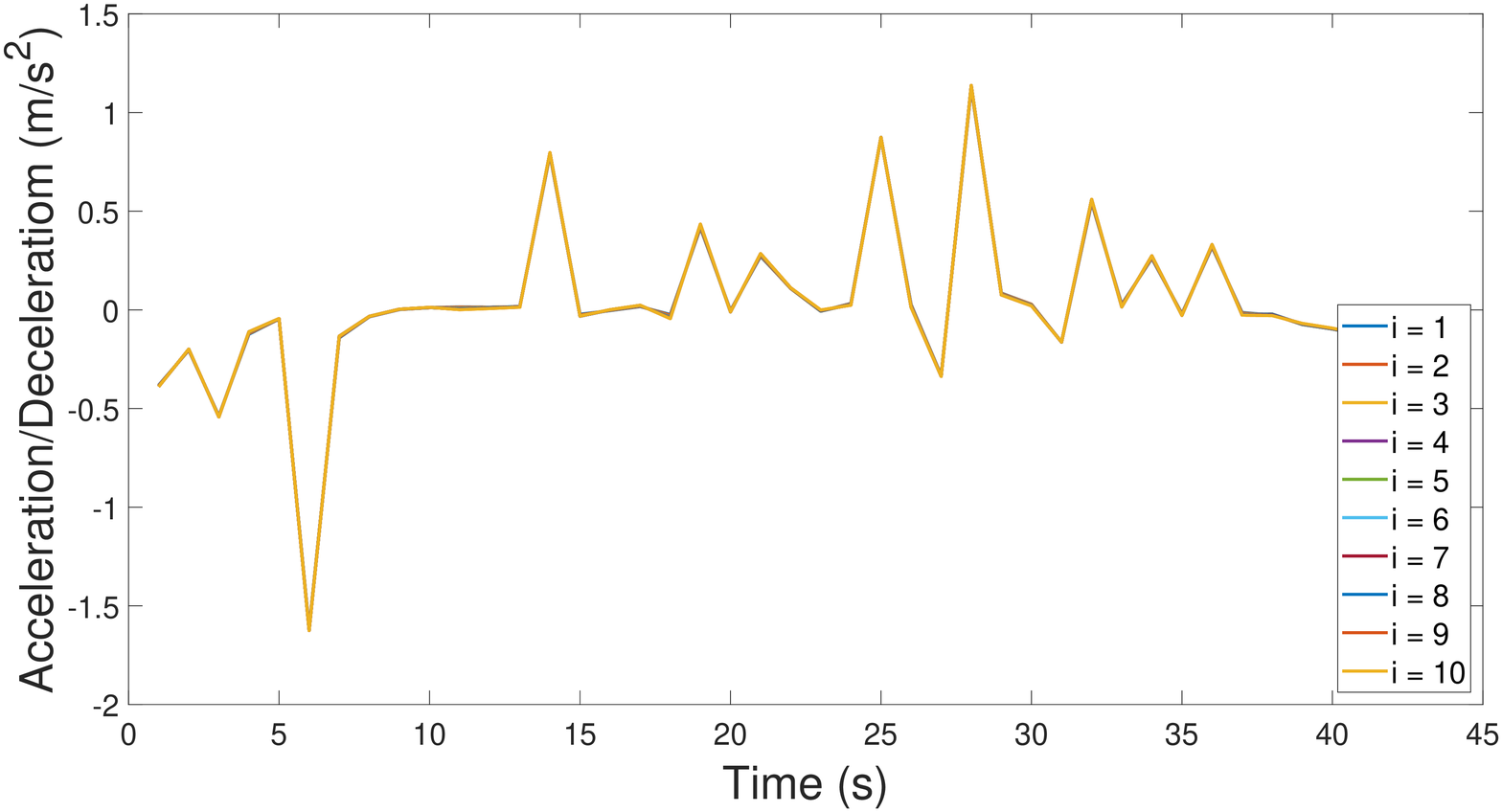}
}
%
\caption{Scenario 3: the proposed CAV platooning control with $p=1$ (left column) and $p=5$ (right column).}
\label{Fig:S3}
\end{figure}

\begin{figure}[htbp]
\centering
\subfigure[Time history of spacing changes.]{ 
\includegraphics[width=0.48\columnwidth]{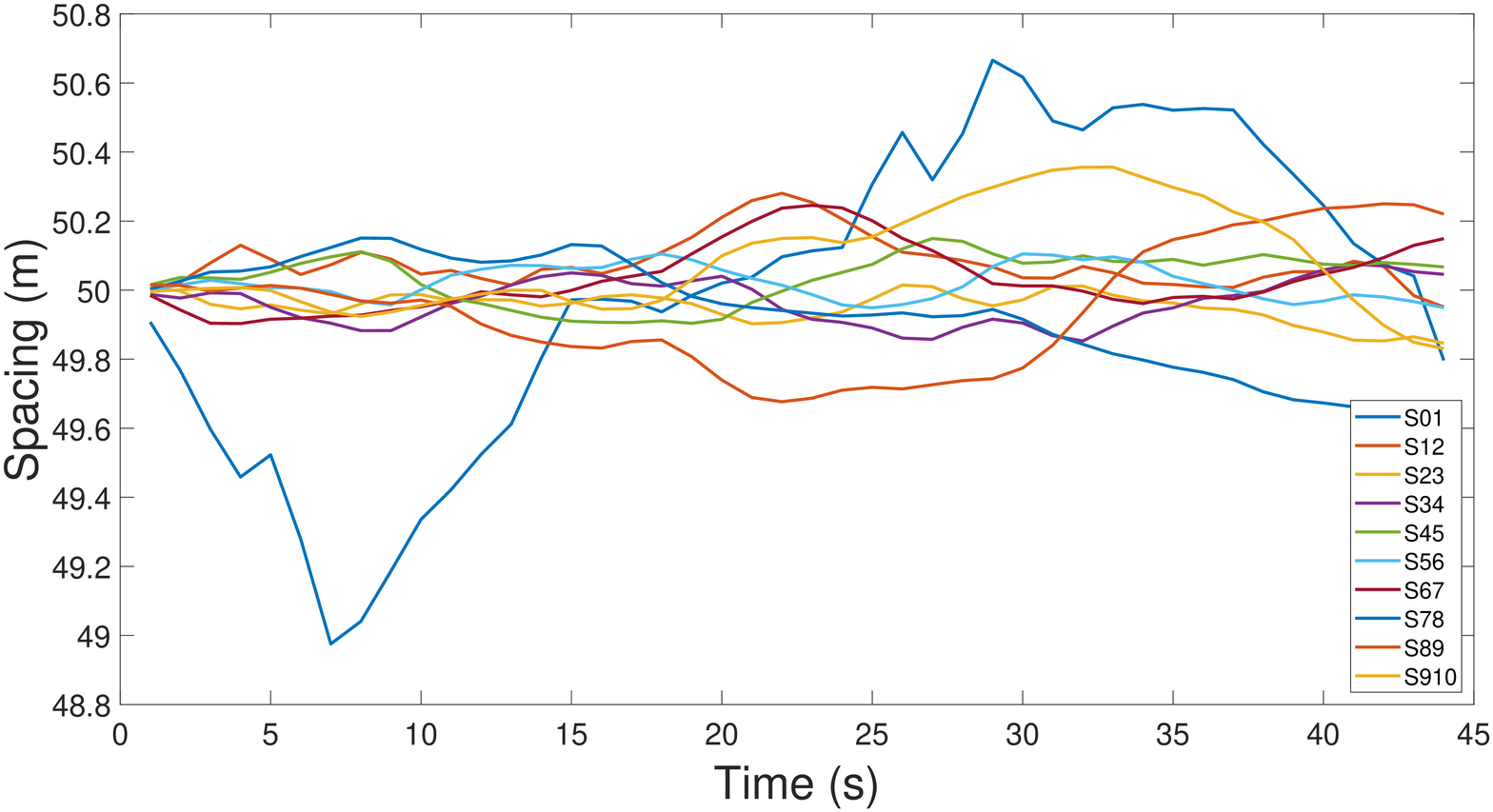}
}
\subfigure[Time history of spacing changes.] { 
\includegraphics[width=0.48\columnwidth]{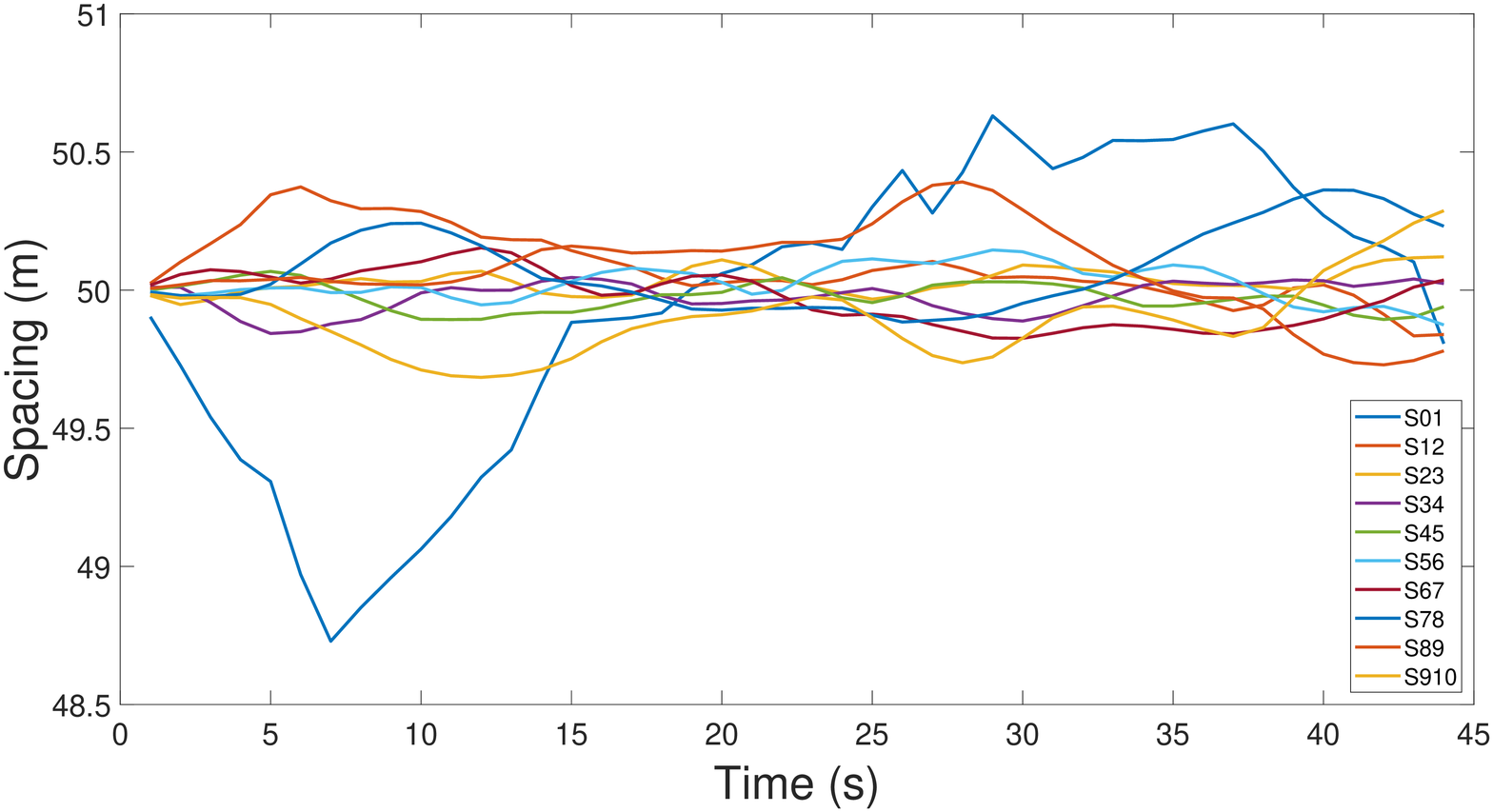}
}
\subfigure[Time history of vehicle speed. ]{ 
\includegraphics[width=0.48\columnwidth]{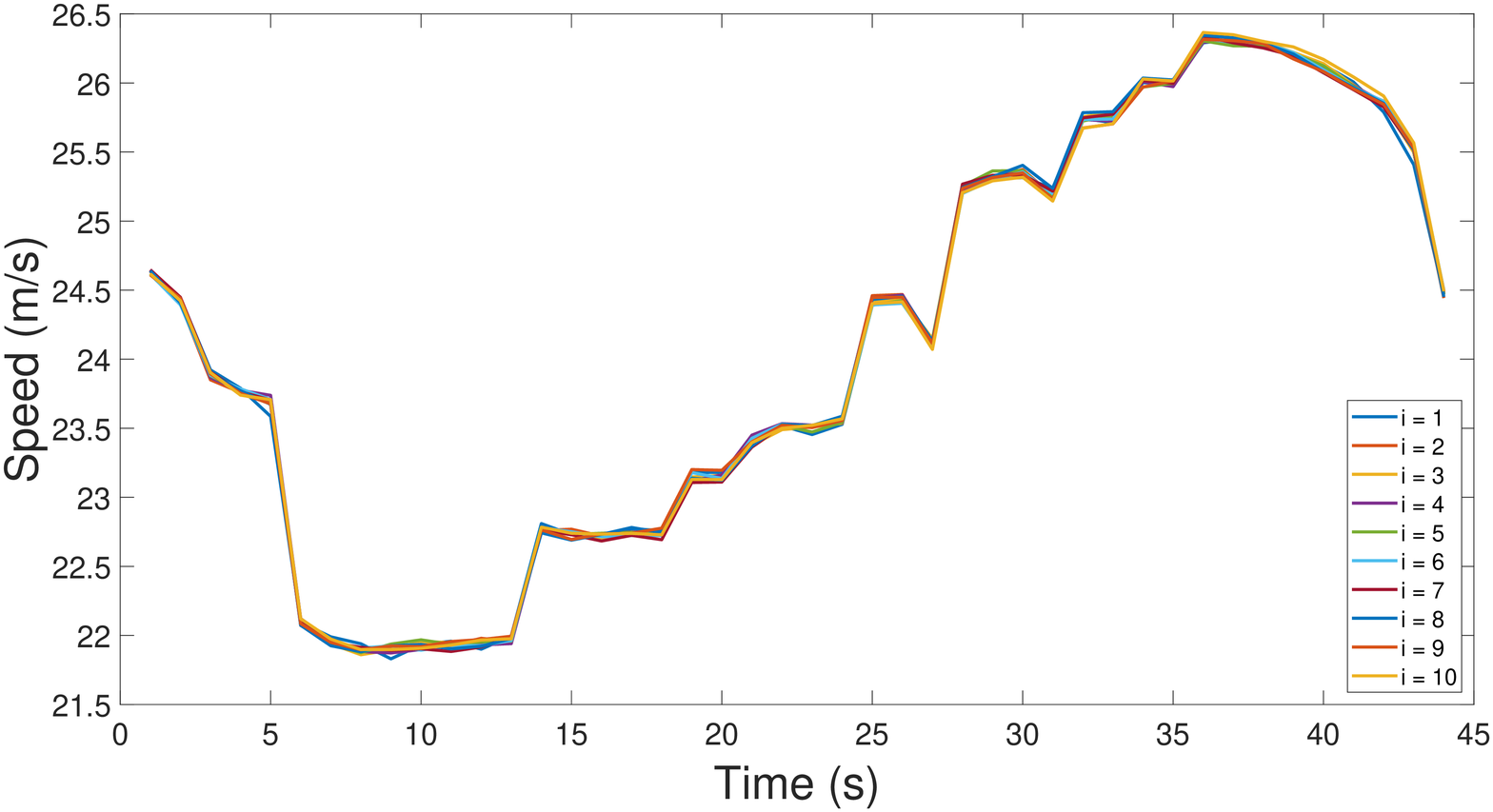}
}
\subfigure[Time history of vehicle speed.] { 
\includegraphics[width=0.48\columnwidth]{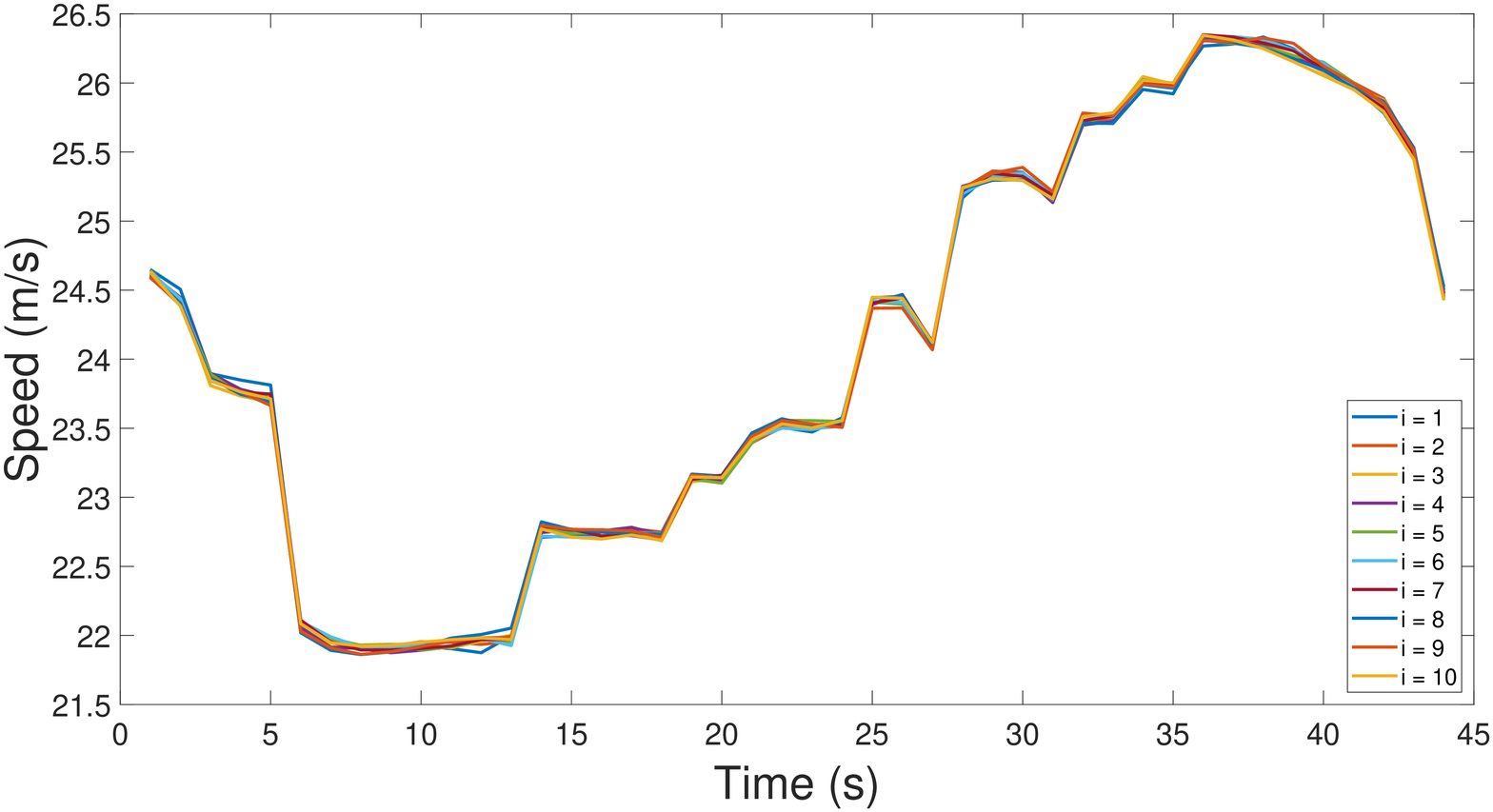}
}
\subfigure[Time history of control input.]{ 
\includegraphics[width=0.48\columnwidth]{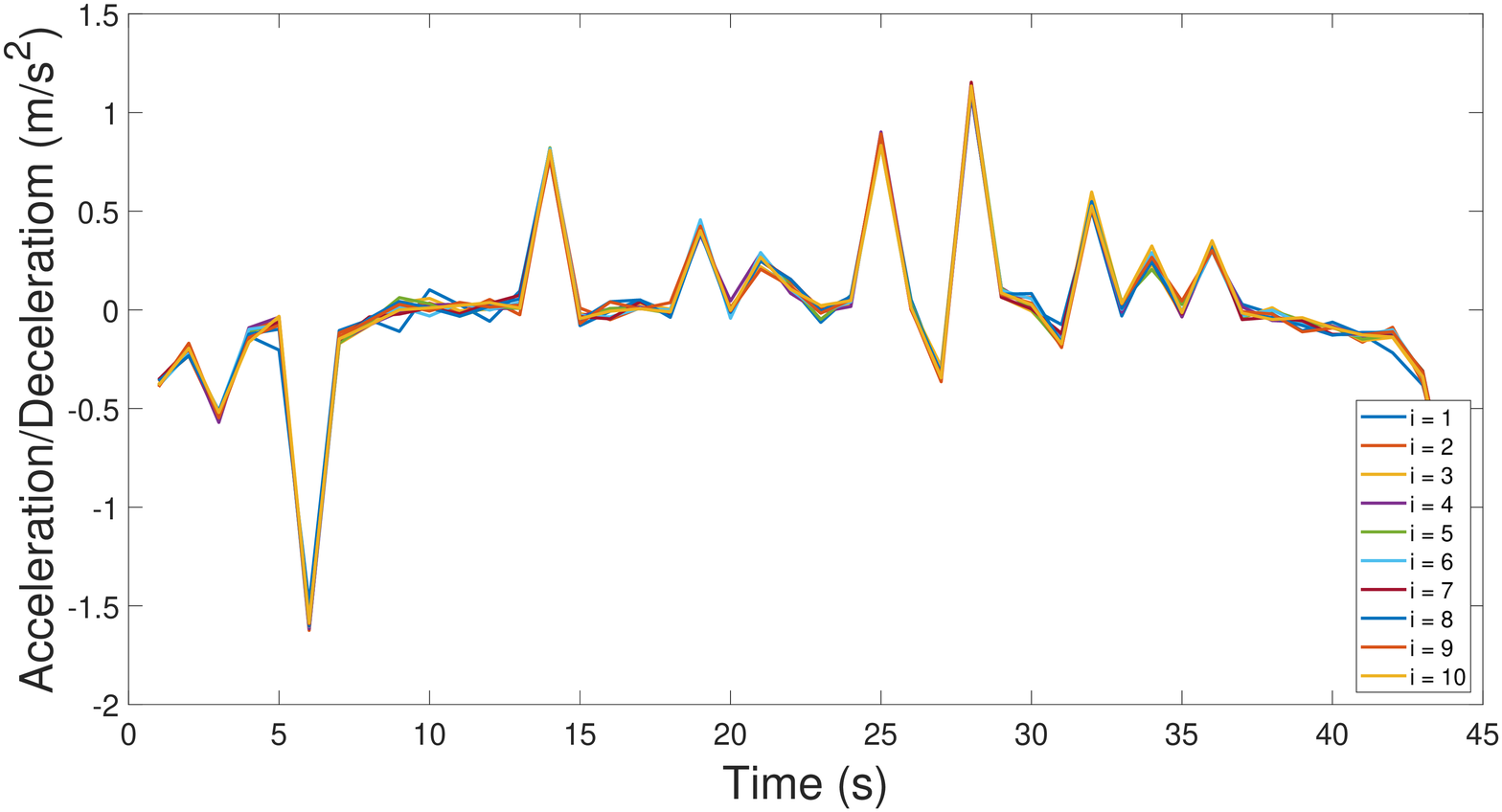}
}
\subfigure[Time history of control input] { 
\includegraphics[width=0.48\columnwidth]{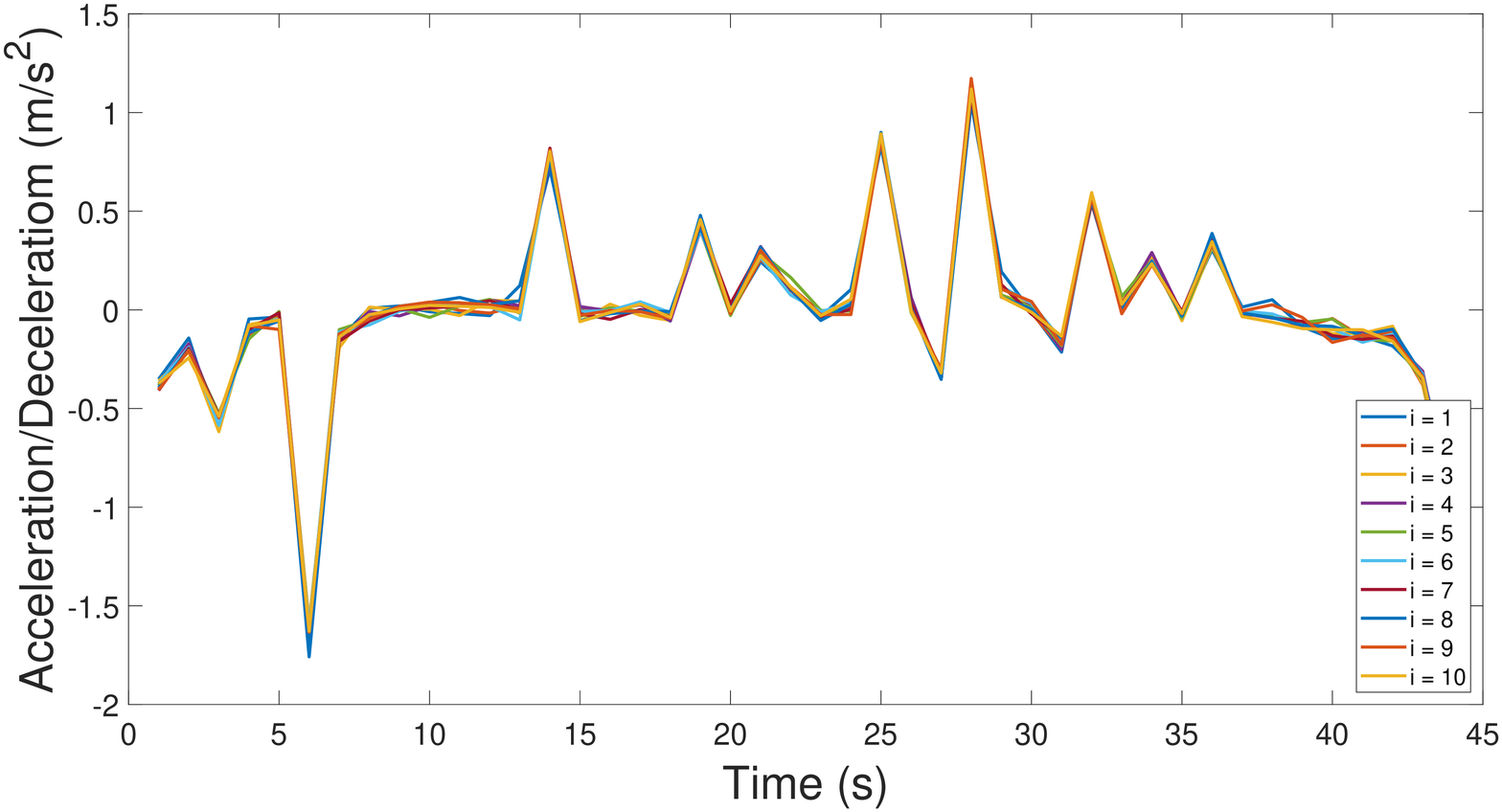}
}
%
\caption{Scenario 3 under noises: the proposed CAV platooning control with $p=1$ (left column) and $p=5$ (right column).}
\label{Fig:S3_noise}
\end{figure}

\gap

\noindent{\bf Performance of the CAV platooning control}. \
We discuss the closed-loop performance of the proposed CAV platooning control for the  three aforementioned scenarios with different MPC horizon $p$'s. In each scenario, we evaluate the performance of the spacing between two neighboring vehicles (i.e., $S_{i-1, i}(k):=x_{i-1}(k) - x_{i}(k)=z_{i}(k)+\Delta$), the vehicle speed $v_i(k)$, and the control input $u_i(k), \, i=1, \ldots, n$ for $p=1, 2, 3, 4, 5$. Due to the paper length limit, we present the closed-loop performance for $p=1$ and $p=5$ only for each scenario; see Figures~\ref{Fig:S1}-\ref{Fig:S3} for (noise free) Scenarios 1-3 respectively, and Figure~\ref{Fig:S3_noise} for Scenario 3 with noises.
In fact, it is observed from these figures (and the other tests) that there is little difference in control performance between $p=1$ and a higher $p$, e.g., $p=5$.
 We comment more on the closed-loop performance of each scenario as follows:
\begin{itemize}
  \item [(i)] Scenario 1.
   Figure~\ref{Fig:S1} shows that the spacing between the uncontrolled leading vehicle and the first CAV, i.e., $S_{0, 1}$, has mild deviation from the desired $\Delta$ when the leading vehicle performs instantaneous acceleration or deceleration, while the spacings between the other CAVs remain the desired constant $\Delta$. Further, it takes about $35s$ for $S_{0, 1}$ to converge to the steady state with the maximum spacing deviation $2.66m$. The similar performance can be observed for the vehicle speed and control input. In particular, it can be seen that all the CAVs show the exactly same speed change and control, implying that the CAV platoon performs a ``coordinated'' motion with ``consensus'' under the proposed platooning control.
 \item [(ii)]
Scenario 2.  Figure~\ref{Fig:S2} displays that under the periodic acceleration/deceleration of the leading vehicle, the CAV platoon also demonstrates a ``coordinated'' motion with ``consensus''. For example, only $S_{0, 1}$ demonstrates mild fluctuation, whereas the spacings between the other CAVs remain the desired constant, and all the CAVs show the exactly same speed change and control. Moreover, under the proposed platooning control, the oscillations of $S_{0, 1}$ are relatively small with the maximal magnitude less than $0.22m$. Such oscillations quickly decay to zero within $30s$ when the leading vehicle stops its periodical acceleration/deceleration.

  \item [(iii)] Scenario 3. In this scenario, the leading vehicle undergoes various traffic oscillations through the time window of $45s$. In spite of such oscillations, it is seen from Figure~\ref{Fig:S3}  that
     only $S_{0, 1}$ demonstrates small spacing variations with the maximum magnitude less than $1m$, but the spacings between the other CAVs remain almost constant $\Delta$ through the entire time window. This shows that the CAV platoon also demonstrates a ``coordinated'' motion with ``consensus'' as in Scenarios 1-2.

   \item [(iv)] Scenario 3 subject to noises.
     Figure~\ref{Fig:S3_noise} shows the control performance of the CAV platoon in Scenario 3 under noises.
     It can be seen that there are more noticeable spacing  deviations from the desired constant $\Delta$ for all CAVs due to the noises. However, the variation of $S_{0, 1}$ remains to be within $1 m$, and
     the maximum deviation of each $S_{i-1, i}$ with $i\ge 2$ is less than $0.5 m$. Further, the profiles of the CAV speed and control still demonstrate   a nearly ``coordinated'' motion in spite of the noises.
\end{itemize}
%


In summary, the proposed platooning control effectively mitigates traffic oscillations of the spacing and vehicle speed of the platoon; it actually achieves a (nearly) consensus motion of the entire CAV platoon even under small random noises and perturbations.
Compared with other platoon centered approaches, e.g., \cite{GShenDu_TRB16},  the proposed control  scheme performs better since it uses different weight matrices that
lead to decoupled closed loop dynamics; this choice of the weight matrices also facilitates the development of fully distributed computation.

%
\section{Conclusion} \label{sect:conclusion}

The present paper develops fully distributed optimization based MPC schemes for CAV platooning control under the linear vehicle dynamics. Such schemes do not require centralized data processing or computation and are thus applicable to a wide range of vehicle communication networks.
New techniques are exploited to develop these schemes, including a new formulation of the MPC model, a decomposition method for a strongly convex quadratic objective function, formulating the underlying optimization problem as locally coupled optimization, and Douglas-Rachford method based distributed schemes. Control design and stability analysis of the closed loop dynamics is carried out for the new formulation of the MPC model. Numerical tests are conducted to illustrate the effectiveness of the proposed fully distributed schemes and CAV platooning control. Our future research will address nonlinear vehicle dynamics and robust issues under uncertainties, e.g., model mismatch, sensing errors, and communication delay.

%
\section{Acknowledgements}

This research work is supported by the NSF grants  CMMI-1901994 and  CMMI-1902006. We thank Benjamin Hyatt of University of Maryland Baltimore County for his contribution to the  proof of the closed loop stability when $p=3$.




%

{\small

\bibliographystyle{abbrv}
\bibliography{CAV_linear_bib02}
}

\end{document}